\documentclass[11pt,oneside,english,jou]{amsart}
\usepackage[T1]{fontenc}
\usepackage[latin9]{inputenc}
\usepackage{babel}
\usepackage{verbatim}
\usepackage{mathrsfs}
\usepackage{amstext}
\usepackage{amsthm}
\usepackage{amssymb}
\usepackage{amsfonts}
\usepackage{amsmath}
\usepackage{amscd}
\usepackage{esint}
\usepackage{enumerate}
\usepackage{relsize} 
\usepackage[unicode=true,pdfusetitle,
bookmarks=true,bookmarksnumbered=false,bookmarksopen=false,
breaklinks=false,pdfborder={0 0 1},colorlinks=false]
{hyperref}
\usepackage[margin=2.5cm]{geometry}
\usepackage[foot]{amsaddr}
\usepackage{mathtools}
\usepackage{ dsfont }
\usepackage[shortlabels]{enumitem}
\usepackage{bm}

\makeatletter
\numberwithin{equation}{section}
\numberwithin{figure}{section}
\theoremstyle{plain}
\newtheorem{thm}{\protect\theoremname}[section]
\theoremstyle{definition}
\newtheorem{rem}[thm]{\protect\remarkname}
\theoremstyle{definition}
\newtheorem{defn}[thm]{\protect\definitionname}
\theoremstyle{plain}
\newtheorem{prop}[thm]{\protect\propositionname}
\theoremstyle{plain}
\newtheorem{lem}[thm]{\protect\lemmaname}
\theoremstyle{plain}

\theoremstyle{plain}

\theoremstyle{definition}

\theoremstyle{definition}

\theoremstyle{definition}

\theoremstyle{definition}

\newtheorem{lemma}[thm]{Lemma}

\newtheorem{assumption}[thm]{Assumption}

\usepackage{tikz}
\usetikzlibrary{shapes,arrows}
\usepackage{verbatim}
\usepackage{amsthm}
\usepackage{amstext}

\DeclareMathOperator{\diam}{diam}
\DeclareMathOperator{\dist}{dist}
\DeclareMathOperator{\Arg}{Arg}

\DeclareMathOperator{\supp}{supp}

\newcommand{\D}{\mathbb D}

\newcommand{\R}{\mathbb R}
\newcommand{\Z}{\mathbb Z}
\newcommand{\N}{\mathbb N}
\newcommand{\Q}{\mathbb Q}

\newcommand{\T}{\mathbb T}
\newcommand{\PP}{\mathbb P}
\newcommand{\bP}{\mathbf P}
\newcommand{\bE}{\mathbf E}

\renewcommand{\Re}{\textup{Re}}
\renewcommand{\Im}{\textup{Im}}

\newcommand{\eps}{\varepsilon}

\newcommand{\mU}{\mathcal{U}}
\newcommand{\mG}{\mathcal{G}}

\newcommand{\s}{\mathbb{S}}

\newcommand{\ncal}{\mathcal{N}}
\def\Nk{{\ncal}}

\newcommand{\rcal}{\mathcal{R}}

\newcommand{\fm}{\mathfrak{m}}

\newcommand{\bx}{\mathbf{x}}
\newcommand{\E}{\mathbb{E}}

\newcommand{\X}{\mathbb{X}}
\newcommand{\mP}{\mathcal{P}}

\newcommand{\hdim}{\dim_H}

\DeclareMathOperator{\sdim}{s-dim}

\DeclareMathOperator{\satdim}{sat-dim}

\makeatother

\providecommand{\conjecturename}{Conjecture}
\providecommand{\corollaryname}{Corollary}
\providecommand{\definitionname}{Definition}
\providecommand{\examplename}{Example}
\providecommand{\lemmaname}{Lemma}
\providecommand{\problemname}{Problem}
\providecommand{\propositionname}{Proposition}
\providecommand{\remarkname}{Remark}
\providecommand{\theoremname}{Theorem}
\providecommand{\taskname}{Task}

\newcommand{\Conv}{\mathlarger{\mathlarger{\ast}}}

\def\diam{{\rm diam}}
\def\dist{{\rm dist}}
\def\supp{{\rm supp}}
\def\vphi{\varphi}

\def\wtil{\widetilde}
\def\half{\frac{1}{2}}

\newcommand{\lam}{\lambda}

\newcommand{\om}{\omega}

\def\Om{\Omega}

\newcommand{\sig}{\sigma}

\def\T{{\mathbb T}}
\def\N{{\mathbb N}}
\def\C{{\mathbb C}}

\def\pom{{(\om)}}

\def\bp{{\bf p}}

\def\Ak{{\mathcal A}}
\def\Bk{{\mathcal B}}

\def\Pk{{\mathcal P}}

\def\bx{{\mathbf x}}
\def\be{\begin{equation}}
	\def\ee{\end{equation}}
\def\Dk{{\mathcal D}}
\newcommand{\Ek}{{\mathcal E}}
\newcommand{\Fk}{{\mathcal F}}
\def\Gk{{\mathcal G}}

\def\ov{\overline}

\def\what{\widehat}
\def\wt{\widetilde}

\begin{document}
\title{Absolute continuity of self-similar measures on the plane}
\author{Boris Solomyak$^1$}
\address{$^1$Department of Mathematics, Bar-Ilan University , Ramat Gan, 5290002, Israel }
\email{bsolom3@gmail.com}
\author{Adam \'Spiewak$^{1,2}$}
\address{$^2$Institute of Mathematics, Polish Academy of Sciences,
	ul.~\'Sniadeckich 8, 00-656 Warszawa, Poland}
\email{ad.spiewak@gmail.com}

\thanks{Both authors acknowledge support from the Israel Science Foundation, grant 911/19.}

%\date{\today}

\begin{abstract} Consider an iterated function system consisting of similarities on the complex plane of the form $g_{i}(z) = \lam_i z + t_i,\ \lam_i, t_i \in \C,\ |\lam_i|<1, i=1,\ldots, k$. We prove that for almost every choice of $(\lam_1, \ldots, \lam_k)$ in the super-critical region (with fixed translations and probabilities), the corresponding self-similar measure is absolutely continuous. This extends results of Shmerkin-Solomyak (in the homogenous case)  and Saglietti-Shmerkin-Solomyak (in the one-dimensional non-homogeneous case). As the main steps of the proof, we obtain results on the dimension and power Fourier decay of random self-similar measures on the plane, which may be of independent interest.
\end{abstract}

\maketitle

\section{Introduction and main results}

Self-similar sets and measures are among the most basic objects in fractal geometry that have been studied for almost a century. The standard way to define and study them goes via the notion of an iterated function system (IFS). Dimension, measure, and topological properties are well-understood
when the IFS is ``non-overlapping'' or ``just-touching'' (satisfies the Open Set Condition), whereas the ``overlapping'' case is still full of difficult open questions, in spite of great progress achieved in the
last decades. The work of Hochman brought the methods of additive combinatorics into the field, which led to breakthrough results on the dimension of self-similar sets and
measures, first in $\R$ \cite{H}, and then in $\R^d$ \cite{HRd}; in particular, he established dimension formulas  for ``typical'' IFS in a very strong sense. It should be noted that in this development, as
in many other areas of fractal geometry, there is a big difference and many extra challenges as the ambient dimension $d$ increases from $d=1$ to $d>1$, and many times from $d=2$ to $d>2$ as well.

The current paper focuses on the question of absolute continuity of self-similar measures in the ``super-critical'' regime, that is, when the expected dimension exceeds $d$, and this requires new ideas.
Shmerkin \cite{ShmerkinBC} found a way to combine the results of Hochman \cite{H} on the dimension with old results of Erd\H{o}s and Kahane on power Fourier decay to obtain significant progress
for ``typical''  (with respect to the contraction parameter) Bernoulli convolutions and their immediate generalizations.  Another direction, pioneered by
 Varj\'u \cite{Var19} and continued by Kittle \cite{Kittle}, vastly increased the set of {\em specific} parameters, for which Bernoulli convolutions are known to be absolutely continuous.
It is based on a very delicate extension of Hochman's method, combined with number-theoretic and Diophantine techniques,  and is so far limited to the case of homogeneous self-similar measures, 
that is, when the linear parts of the similarities are equal to each other. It is, in a sense, ``orthogonal'' to Shmerkin's approach, which is aimed at getting results for ``typical'' parameters. So far, the
latter approach has been the one which allowed for a greater scope of generalizations, although, at its core, it is limited to 
measures having a convolution structure and thus harder to implement in the non-homogeneous case. A novel idea of representing a non-homogeneous self-similar measure as an integral of
measures, which are only ``statistically self-similar,'' but have a convolution structure, was introduced by Galicer, Saglietti, Shmerkin, and Yavicoli \cite{GSSY}. It was then used by
Saglietti, Shmerkin, and Solomyak \cite{SSS} to obtain definitive results on typical absolute continuity for typical (non-homogeneous) self-similar measures on $\R$.
Our main result is an extension
of these results to the case of non-homogeneous self-similar measures in $\R^2\cong \C$, for orientation preserving IFS. As we explain below, this method does not allow a direct generalization to $d>2$.

\subsection{Absolute continuity of planar self-similar measures}\ \\

For $\lam \in \D_* := \{ z \in \C :  0 < |z| < 1 \}$ and $t \in \C$, let $g_{\lam,t}$ denote the contracting similarity map on $\C$ defined as
\[  g_{\lam, t} (z) = \lam z + t. \]
Given  $\lam = (\lam_1, \ldots, \lam_k) \in \D_*^k$ and $t = (t_1, \ldots, t_k) \in \C^k$, together with a probability vector $\bp = (p_1, \ldots, p_k)$, there exists a unique Borel probability measure $\nu = \nu^{\bp}_{\lam, t}$ on $\C$, stationary for the iterated function system $\{g_{\lam_1, t_1}, \ldots, g_{\lam_k, t_k}\}$ with probabilities $(p_1, \ldots, p_k)$, i.e. satisfying
\[ \nu = \sum \limits_{i=1}^k p_i g_{\lam_i, t_i} \nu\]
(for a measurable map $f : X \to Y$ and a measure $\nu$ on $X$, we denote by $f\nu$ the push-forward of $\nu$ by $f$ defined as $f\nu(A) = \nu(f^{-1}(A))$ for all measurable sets $A \subset Y$). Measures of this form are called \textit{self-similar} measures. The \textit{similarity dimension} of the measure $\nu^{\bp}_{\lam,t}$ is
\begin{equation}\label{eq:sim dim} s(\lam,\bp) := \frac{\sum \limits_{i=1}^k p_i \log p_i}{\sum \limits_{i=1}^k p_i \log |\lam_i|}.
\end{equation}
It is well-known that the similarity dimension is an upper bound for the Hausdorff dimension of $\nu^\bp_{\lam, t}$ and hence  $\nu^\bp_{\lam, t}$ is singular with respect to the two-dimensional Lebesgue measure on $\C$ provided that $s(\lam, p) < 2$. In this paper we study absolute continuity of $\nu^\bp_{\lam, t}$ for typical parameters in the \textit{super-critical case} $s(\lam , \bp) > 2$. Our main result is the following.
\begin{thm}\label{thm:main ac}
Fix $k \geq 2$, distinct translations $t_1, \ldots, t_k \in \C$, and a probability  vector $\bp = (p_1, \ldots, p_k)$ with strictly positive entries. There exists a set
\[ E^{\bp}_t \subset \rcal^{\bp}_t := \{  (\lam_1, \ldots, \lam_k) \in \D_*^k : s(\lam, \bp) > 2\} \]
of zero Lebesgue measure such that for every $\lam \in \rcal^{\bp}_t \setminus E^{\bp}_t$, the self-similar measure $\nu^{\bp}_{\lam, t}$ is absolutely continuous (with respect to the $2$-dimensional Lebesgue measure on $\C$).
\end{thm}

This (partially) extends \cite[Theorem B]{SS16} from the \textit{homogeneous} case (i.e. $\lam_1 = \lam_2 = \ldots = \lam_k$) to the \textit{non-homogeneous} case (not all $\lam_i$ equal) and \cite[Theorem 1.1]{SSS} from the real line to the complex plane. As in \cite{SSS}, we do not get any information about the density of the absolutely continuous measures $\nu^{\bp}_{\lam, t}$ or
estimates on the dimension of the exceptional set of parameters. However, the exceptional set is in general at least of complex
codimension one, since, for example, if $\frac{t_1}{1-\lam_1} = \frac{t_2}{1-\lam_2}$, then
the first two maps of the IFS share the fixed point, hence commute, resulting in the similarity dimension drop for the 2nd iterate of the IFS. It should be also possible to find  exceptional parameters,
by analogy with singular non-homogeneous self-similar measures on $\R$, having similarity dimension greater than one. Such examples were found by Neunh\"auserer \cite{Neun06,Neun11} 
using certain algebraic curves. We expect that analogous results hold in 2 dimensions as well. There are also exceptions arising from self-similar measures that are not Rajchman,
that is, measures for which the Fourier transform does not vanish at infinity, going back to Pisot reciprocal Bernoulli convolutions \cite{ErdosSingular}.
Br\'emont \cite{Bremont} gave a complete characterization of self-similar IFS on $\R$ for which there exists a probability vector such that the self-similar measure is non-Rajchman. They all have the 
property that the contraction ratios are powers of a single Pisot number. This was generalized by Rapaport \cite{Rapaport22Adv} to the case of $\R^d$; in the case of $\R^2$ and complex contraction
coefficients, they should all be powers of a complex Pisot number. Thus, these examples form a set of zero Hausdorff dimension in the context of Theorem~\ref{thm:main ac}.

\subsection{Background}

The study of geometric properties of self-similar measures (and, more generally, invariant measures for iterated function systems), in particular their dimension and absolute continuity, has been an object of intensive study. It dates back to the seminal works of Erd\H{o}s \cite{ErdosSingular, ErdosSmooth} on \textit{Bernoulli convolutions}, i.e., the family of self-similar measures corresponding to the homogeneous system $\{ x \mapsto \lam x, x \mapsto \lam x +1\},\ \lam \in (0,1)$, on the real line with equal probabilities $p_1 = p_2 = 1/2$.
 We do not review the long history, but instead refer the reader to the surveys  \cite{Sixty,VarjuSurvey}. In the mid-1990's a {\em transversality technique} was 
introduced in the context of homogeneous self-similar IFS on the line by Pollicott and Simon in \cite{PollicottSimonDeleted}, which was further extended by many other authors.
Since then, transversality has been widely studied and applied in much greater generality --- see \cite{BSSStypical} and references therein for the latest developments.

\begin{comment}
Erd\H{o}s proved in \cite{ErdosSingular} that there exist parameters $\lam \in (1/2, 1)$, for which the Bernoulli convolution is singular with respect to the one-dimensional Lebesgue measure (even though the attractor of the system is an interval). Parameters found by Erd\H{o}s are reciprocals of Pisot-Vijayaraghavan numbers and it remains a major open problem whether these are the only parameters in the super-critical region $\lam  > 1/2$ for which Bernoulli convolutions are singular.

In \cite{SolAC}, the first author has proved that Bernoulli convolutions are absolutely continuous for \textit{almost every} $\lam \in (1/2, 1)$. The proof was based on the transversality technique, introduced in the context of iterated function systems by Pollicott and Simon in \cite{PollicottSimonDeleted}, and the infinite convolution structure of Bernoulli convolutions. Since then, transversality has been widely studied and applied --- both in the context of self-similar sets and measures, see e.g. \cite{SolPeresBernoulli, SolFamilies, SolPeresIntersections, NeunhausererTrans, NgaiWang, SolXu}, as well as for more general non-linear iterated function systems, see e.g. \cite{SimonSolHorseshoes, PeresSchlag, SSUPacific, SSUTrans, BPSBlackwell, BaranyKolossvaryBlackwell, BSSSTypical}.
\end{comment}

 The drawback of this technique is that it is rarely able to provide almost sure results in the full region of parameters, even for self-similar measures. This was the case for non-homogeneous self-similar systems $\{ \lam_1 x, \lam_2 x + 1 \}$ on the real line, for which  Neunh\"auserer \cite{NeunhausererTrans} and Ngai-Wang \cite{NgaiWang} proved absolute continuity for almost every pair $(\lam_1, \lam_2)$ in a certain simply-connected subset of the super-critical region $\lam_1\lam_2 > 1/4$, far from covering the whole region. Similarly, Solomyak and Xu \cite{SolXu} proved absolute continuity of homogeneous complex Bernoulli convolutions, i.e. self-similar measures for the system $\{\lam x, \lam x + 1\}$ on $\C$ for almost every $\lam \in \D_*$ such that $2^{-1/2} \leq |\lam| \leq 2 \times 5^{-5/8}$. Again, this does not cover the whole super-critical region $|\lam| \geq 2^{-1/2}$. Let us also mention that it is often easier to obtain typical absolute continuity if one fixes the linear parts and varies the translations - see \cite{SS02, JPS07, KaenmakiOrponenTranslation} for results valid for typical translations.

An improvement over the transversality technique was possible due to remarkable results of Hochman \cite{H, HRd}, who proved inverse theorems for entropy and applied them to study the dimension of self-similar measures. In dimension one, he proved that if an IFS consisting of similarities on $\R$ satisfies the \textit{exponential separation condition}, then each corresponding self-similar measure has dimension equal to the minimum of one and the similarity dimension, see \cite[Theorem 1.3]{H}. In (reasonably) parametrized families of self-similar iterated function systems, the exponential separation condition holds outside of an exceptional set of Hausdorff dimension zero \cite[Theorem 1.8]{H}, hence the formula for the dimension of self-similar measures is established for typical parameters. Based on this result, Shmerkin \cite{ShmerkinBC} was able to prove that for a homogeneous self-similar IFS $\{\lam x + t_1,  \ldots, \lam x + t_k,\}$ on $\R$, there is an exceptional set $E$ of Hausdorff dimension zero, such that for every $\lam \in (0,1) \setminus E$ and a probability vector $\bp$ such that $s(\lam, \bp) > 1$, the corresponding self-similar measure is absolutely continuous on $\R$. The proof is based on the fact (see \cite[Lemma 2.1]{ShmerkinBC}) that if $\mu,\nu$ are probability measures such that $\mu$ has full dimension and $\nu$ has power Fourier decay (i.e. $\hat{\nu}(\xi) \leq O(\xi^{-\sigma})$ for some $\sigma>0$), then the convolution $\mu * \nu$ is absolutely continuous. As a homogeneous self-similar measure is an infinite convolution, which can be in turn represented as a convolution of two self-similar measures with properly scaled parameters, Hochman's results, combined with the Erd\H{o}s-Kahane argument for typical power Fourier decay (\cite{ErdosSmooth, Kahane}, see also \cite{Sixty}), give absolute continuity outside of a set of dimension zero. 
This was later extended by Shmerkin and Solomyak %in \cite[Theorem A]{SS16'} by obtaining additionally absolute continuity with density in $L^q$ for some $q > 1$ and 
in \cite[Theorem B]{SS16} to the case of homogeneous self-similar measures in the complex plane. A crucial feature required for  the above proofs of typical absolute continuity was the convolution structure of homogeneous self-similar measures. An extension to the non-homogeneous self-similar measures (which lack the convolution structure) was performed in the one-dimensional case by Saglietti, Shmerkin and Solomyak in \cite{SSS}. Their approach is based on disintegrating a non-homogeneous self-similar measure into a family of homogeneous random self-similar measures (a technique pioneered in \cite[Section 6.4]{GSSY}) and extending  the dimension and power Fourier decay results to such class of measures. We review this approach with more details in the next subsection. 
The goal of this paper is to repeat the same strategy in two dimensions, where new difficulties emerge.

The disintegration mentioned above leads one to the study of \textit{models} - certain classes of random self-similar measures, where at each step one is allowed to apply a map from a finite collection of homogeneous iterated functions systems. Several classes of random fractal measures were studied in the literature, see e.g. \cite{Hambly, StenfloRandomIFS, FalconerJinExact, VFractals, PeresSimonSolRandom, BaranyPerssonRandom}; the particular model considered in both \cite{SSS} and the present paper was introduced in \cite{GSSY}. As this class of random measures turns out to be useful in the study of non-homogeneous self-similar iterated function systems \cite{SSS, ASSNormal, KaenmakiOrponenTranslation}, our results on the dimension (Theorem \ref{assmp:non-deg and rotation}) and the power Fourier decay (Theorem \ref{thm:fourier main}) of such measures in two dimensions might be of independent interest. Having that in mind, we prove a more general version of the dimension result than the one sufficient for the proof of Theorem \ref{thm:main ac} by considering general ergodic shift-invariant measures on the symbolic space of the model (instead of only Bernoulli measures) and possibly reducible case (with $\Arg(\lam_i) \in \pi \Q$ for all $i$). %The argument is similar to the one from the proof of \cite[Lemma 3.7]{HR} with both of the variants mentioned above introducing some additional technical difficulties.

As explained at the end of the next subsection, the methods of our work do not extend to dimensions higher than two. Let us note, however, that using a different approach, Lindenstrauss and Varj\'u \cite[Theorem 1.3]{LindenstraussVarju} were able to prove absolute continuity for many self-similar measures with contraction rate close enough to one. More precisely, they proved that for $d \geq 3$, if $U_1, \ldots, U_k \in SO(d)$ and a probability vector $\bp = (p_1, \ldots, p_k)$ are such that the transfer operator $f \mapsto \sum \limits_{j=1}^k p_j f \circ U_j$ has a spectral gap on $L^2(SO(d))$, then there exists $0 < \overline{r} < 1$ such that for all $r_1, \ldots r_k  \in (\ov{r}, 1)$ and every $t_1, \ldots, t_k \in \R^d$, the self-similar measure corresponding to the system $\{r_1U_1 + t_1, \ldots, r_kU_k + t_k\}$ with probabilities $\bp$ is absolutely continuous (with a smooth density), provided that the maps $r_jU_j + t_j, j = 1, \ldots, k$ do not share a common fixed point. As Benoist and de Saxc\'e \cite{BenoistSaxce} established a spectral gap in the case  when $U_1, \ldots, U_k$ are matrices with algebraic entries and generating a dense subgroup, absolute continuity follows in this case.

For completeness, we now discuss some recent work on Fourier decay of fractal measures. In \cite{Solomyak21} it was shown that non-homogeneous self-similar measures on the line
have power Fourier decay at infinity, outside of a zero Hausdorff dimension set of parameters. The proof was based on a multi-parameter version of the Erd\H{o}s-Kahane method. However, the
power decay obtained there was too weak to imply absolute continuity in any parameter region, and the techniques of that paper are not directly related to the method of \cite{SSS} or the current paper.
A more classical version of the Erd\H{o}s-Kahane argument was employed in \cite{Solomyak22} to show that typical {\em non-self-similar, homogeneous} self-affine measures in $\R^d$ have
power Fourier decay as well.

In a broader area of fractal measures and their Fourier decay, breakthrough results were very recently obtained by Algom, Rodriguez Hertz, Wang \cite{ARW23} and Baker, Sahlsten \cite{BaSa23}
independently. In particular, they showed that a self-conformal measure corresponding to a strictly non-linear sufficiently smooth IFS on the line, has power Fourier decay. A detailed discussion of their
methods, and of other extensive literature on this topic, is beyond the scope of this introduction, but it is interesting that both papers use a model-method disintegration technique.

It is also worth mentioning that a measure may have power decay outside of a sparse set of frequencies, even if it is not a Rajchman measure. In fact, Kaufman \cite{Kaufman} (in the homogeneous case) and Tsujii \cite{Tsujii} (in the non-homogeneous case) proved that for any non-trivial self-similar measure $\mu$ on the real line, for any $\eps>0$ there exists $\delta>0$ such that the set
$
\{t\in [-T,T]:\ |\widehat{\mu}(t)| \ge T^{-\delta}\}
$
can be covered by $T^\eps$ intervals of length $1$. %Mosquera and Shmerkin \cite{MS} made the dependence of $\delta$ on $\eps$ quantitative in the homogeneous case. 
Kaufman \cite{Kaufman} used a version of the Erd\H{o}s-Kahane argument, whereas the proof of Tsujii \cite{Tsujii} is based on large deviation estimates. 
Very recently, Khalil \cite{Khalil23} obtained strong results of this kind for
measures in higher dimension, using methods of additive combinatorics. See, in particular, \cite[Corollary 11.5]{Khalil23}, where this was shown for any compactly supported measure in $\R^d$ that is
affinely non-concentrated at almost every scale (see his paper for the definitions).

\subsection{Strategy of the proof}
The plan of the proof of Theorem \ref{thm:main ac} follows closely the one from \cite{SSS}. Let us recall it. As noted in the introduction, the main difficulty in the non-homogeneous case (i.e. when not all $\lam_i$ are equal) is the fact that the self-similar measure $\nu^\bp_{\lam, t}$ is not an infinite convolution. In order to recover (some of) the convolution structure, Saglietti, Shmerkin and Solomyak consider in \cite{SSS} a disintegration
\begin{equation}\label{eq:disintegration} \nu^{\bp}_{\lam, t} = \int \limits_\Om \eta_\lam^{\pom} d\PP(\om),
\end{equation}
where measures $\eta_\lam{(\om)}$ are given as follows. First, denote $g_i = g_{\lam_i, t_i}$ and define the natural projection map $\Pi_\lam : \{1, \ldots, k\}^\N \to \C$ as
\[ \Pi_\lam (u_1, u_2, \ldots) := \lim \limits_{n \to \infty} g_{u_1} \circ \cdots \circ g_{u_n}(0) = \sum \limits_{n=1}^\infty \left( \prod \limits_{j=1}^{n-1}\lam_{u_j}\right) t_{u_n}. \]
It is easy to see that $\nu^{\bp}_{\lam, t} = \Pi_\lam \bp^{\N}$, where $\bp^{\N}$ denotes the Bernoulli measure on $\{1, \ldots, k\}^\N$ with the marginal $\bp$. Fix $r \in \N$ and define  
\[I := \left\{ (n_1, \ldots, n_k) \in \{0, \ldots, r \}^k : \sum \limits_{j=1}^k n_j = r \right\}\]
and a map $\Psi : \{1, \ldots, k\}^r \to I$ given by $\Psi(w) = \left( N_1(w), \ldots, N_k(w) \right)$, where $N_i(w)$ denotes the number of occurrences of the symbol $i$ in the word $w$. Set $\Om = I^\N$ and consider a random variable $X = (X_0, X_1, \ldots) : \{1, \ldots, k\}^\N \to \Om$, where $X_j(u_1, u_2, \ldots) = \Psi (u_{jr+1}, \ldots u_{(j+1)r} )$ and we treat $u_1, u_2$ as i.i.d. random variables with distribution $\bp$. The distribution $\PP$ of $X$ on $\Om$ is a Bernoulli measure. For $\om \in \Om$, let $\eta_\lam^{\pom}$ be the conditional distribution of $\Pi_\lam$ conditioned on $X = \om$. In other words, we condition on the number of occurrences of each symbol $i$ in the blocks of length $r$ in the symbolic space $\{1, \ldots, k\}^\N$ (and hence the remaining randomness is in the order of the occurrences of symbols in each block). The crucial feature of this construction is the following:  The number of occurrences of symbols $1, \ldots, k$ in the $j$-th block $(u_{jr +1}, \ldots, u_{(j+1)r})$ determines uniquely the linear part of the corresponding map $g_{u_{jr + 1}} \circ \cdots \circ g_{u_{(j+1)r}}$ (as the complex multiplication is commutative), hence after conditioning on $\om$ only the translation part of the random map $g_{u_{jr + 1}} \circ \cdots \circ g_{u_{(j+1)r}}$ remains random. Therefore, for each $\om = (\om_1, \om_2, \ldots)$ the measure $\eta_\lam{(\om)}$ can be realized as a distribution of a random sum of the form
\begin{equation}\label{eq:random sum} \sum \limits_{n=1}^\infty \left( \prod_{j=1}^{n-1}\lam_{\om_j} \right) Y_n^{(\om, \lam)},
\end{equation}
where $Y^{(\om, \lam)}_n$ is a sequence of independent random variables (corresponding to the random choice of the translation part of $g_{u_{jr + 1}} \circ \cdots \circ g_{u_{(j+1)r}}$). Therefore, for a fixed $\om$, the measure $\eta_\lam{\pom}$ is an infinite convolution, as a distribution of a random sum with independent terms. Even though the
measures $\eta_\lam{\pom}$ are no longer self-similar, they are \textit{statistically self-similar} in a suitable sense (see \eqref{eq:dyn self-sim} below). The main body of the proof is extending to such random self-similar measures the dimension result of Hochman \cite{HRd} and the Erd\H{o}s-Kahane argument on typical power decay of the Fourier transform, obtaining their versions valid for $\PP$-a.e.\ $\om$ -- see Theorems \ref{thm:main dim} and \ref{thm:fourier main}. After establishing those, one can follow the plan of proving typical absolute continuity from \cite{ShmerkinBC} -- splitting the sum in \eqref{eq:random sum} into $\sum \limits_{s | n}$ and $\sum \limits_{s \nmid n}$, one realizes $\eta_\lam^{(\om)}$ as a convolution of two measures $\eta'^{(\om)}_\lam, \eta''^{(\om)}_\lam$ from the same general class of random self-similar measures. For a typical $\lam$ in the supercritical region, taking $r$ and $s$ large enough will yield a measure $\eta'^{(\om)}_\lam$ of full Hausdorff dimension and  $\eta''^{(\om)}_\lam$ with a power Fourier decay for almost every $\om$. As a convolution of two such measures is absolutely continuous (see \cite[Lemma 2.1]{ShmerkinBC}; the proof is given there for $d=1$ but it extends to higher dimensions without any change), this gives the absolute continuity of $\eta^{\pom}$ for $\PP$-almost every $\om$, and by \eqref{eq:disintegration} also the absolute continuity of $\nu^\bp_{\lam, t}$.

The strategy explained above was realized in \cite{SSS} for non-homogeneous self-similar measures on the real line. In this paper we follow the same general scheme and extend the result to the complex plane. This amounts to extending the higher dimensional results of Hochman \cite{HRd} and the complex Erd\H{o}s-Kahane argument \cite{SS16} to the model of random self-similar measures. The main new difficulty (with respect to the one-dimensional case of \cite{SSS}) in the dimension part is the phenomena of saturation on one-dimensional subspaces, which can cause a dimension drop for a self-similar measure even in the presence of exponential separation, see \cite{HRd}. In the genuinely self-similar case, this possibility can be excluded by assuming that the IFS generates a non-trivial rotation group. For random self-similar measures we exclude this possibility for $\PP$-a.e.\ $\eta^{(\om)}$ by studying projections onto one-dimensional subspaces and employing some tools from the ergodic theory. The difficulty in extending the Erd\H{o}s-Kahane method lies in combining a rather subtle algebraic and combinatorial argument of \cite{SS16} with the random setting of models. Nevertheless, several points of the arguments follow in the same way as in \cite{SSS}, without any essential changes. In such places we provide detailed sketches and reference reader to \cite{SSS} for full details.

The reason that the disintegration into random self-similar measures with a convolution structure works in dimensions 1 and 2 is that the orthogonal group, hence the group of linear similarities
 in $\R^d$ for $d\le 2$, is of polynomial growth. Grouping together all compositions of $g_i$'s of length $r$ according to the linear part yields a homogeneous self-similar IFS and for large $r$, many of those systems have entropy approximating $rH(\bp)$ (up to an error sublinear in $r$). Applying these IFS in a random order results in the statistically self-similar convolutions $\eta_\lam^{(\om)}$ which we study. When $d\ge 3$,
finitely generated subgroups of the orthogonal group $O(d)$ typically have exponential growth, and hence conditioning on linear parts of compositions of length $r$ in the construction above might lead to a significant entropy drop. In an extreme case of generating a free group, all the systems in the corresponding model will consist of a single map and thus measures $\eta^{(\om)}$ will all be concentrated on singletons, hence we cannot obtain full dimension of $\eta^{\pom}$ in this construction. It therefore remains an open problem to obtain typical absolute continuity in the super-critical region for self-similar measures in higher dimensions, where the non-abelian character of the orthogonal group makes it more challenging to apply the model method for obtaining the convolution structure. Nevertheless, there are works where the lack of commutativity has been used as an advantage.  One instance is the already mentioned work of  Lindenstrauss and Varj\'u \cite{LindenstraussVarju}, who were able to prove absolute continuity of some self-similar measures in $\R^d$ for  $d\geq 3$. Their result is however far from covering the whole super-critical region of the parameter space. Also, a variant of the disintegration into random model approach was employed recently in \cite{ARW23} and \cite{BaSa23} in order to obtain power Fourier decay of self-conformal measures for $C^2$ non-linear iterated function systems on $\R$ (see also \cite{ASSNormal}).

\begin{rem} \label{rem:SSSgap}
There is a minor technical oversight in \cite{SSS}. 
Namely, the dimension result on the random measure is stated in \cite[Theorem 1.3]{SSS} under the assumption that the shift-invariant measure on the symbolic space
of the model is a general ergodic measure. However, the proof used (implicitly) that this measure is Bernoulli in two instances. 
This does not affect the main result of that paper, \cite[Theorem 1.1]{SSS}, which asserts absolute continuity for almost all parameters and only relies on the Bernoulli case.
Moreover, the same approach as in the current paper provides a fix needed to confirm \cite[Theorem 1.3]{SSS}  as stated. We will indicate the relevant places in our proof.
\end{rem}

\subsection{Random self-similar measures}\ \\

Let us now introduce the model of random self-similar measure which we will consider and state our main results on them. We follow the setting from \cite{GSSY} and \cite{SSS}. Let $I$ be a finite set. For each $i \in I$, let $\Phi^{(i)} = (f_1^{(i)}, \ldots, f_{k_i}^{(i)})$ be a homogeneous, orientation preserving and uniformly contracting iterated function system of similarities on the complex plane, i.e., the maps $f_j^{(i)}$ are of the form
\[ f_j^{(i)}(z) = \lam_i z + t_j^{(i)} = r_i \vphi_i z + t_j^{(i)},\ i \in I,\ j = 1, \ldots, k_i,\ z \in \C,\]
where $\lam_i, t_j^{(i)} \in \C$ and $|\lam_i|<1$, with $\lam_i = r_i \vphi_i$ and $0 < r_i < 1,\ \vphi_i \in  \s^1= \{ z \in \C : |z| = 1 \}$. Let $r_{\mathrm{min}} = \min\{ r_i : i \in I \}$ and $r_{\mathrm{max}} = \max\{ r_i : i \in I \}$.

Let $\Om = I^\N$. For $\om = (\om_1, \om_2, \ldots) \in \Om$ and $n \in \N \cup \{ \infty \}$ let
\begin{equation}\label{eq:doubleX def} \X^{(\om)}_n := \prod_{j=1}^n \{1 , \ldots, k_{\om_j} \}.
\end{equation}
For $\om \in \Om,\ n \in \N$, and $u \in \X_n^{(\om)}$, define $f_u^{\pom}(z) = f^{(\om_1)}_{u_1} \circ \cdots \circ f^{(\om_n)}_{u_n}(z) = \lam^{(\om)}_u z + t^{(\om)}_u$. Let $\Pi_{\om} : \X_\infty^{(\om)} \to \C$ be the coding map (natural projection) given by
\[ \Pi_{\om}(u) := \lim \limits_{n \to \infty} f^{(\om)}_{u|_n}(0) = \sum \limits_{n=1}^\infty \left( \prod_{j=1}^{n-1}\lam_{\om_j} \right) t_{u_n}^{(\om_n)}, \]
where $u|_n = (u_1 ,\ldots, u_n)$ is the restriction of the infinite word $u \in \X_\infty^{(\om)}$ to the first $n$ coordinates. We will also consider the $n$-truncated coding map $\Pi^{(n)}_\om : \X_\infty^{(\om)} \to \C$ given by
\begin{equation}\label{eq:truncated pi}
\Pi^{(n)}_\om (u) := \sum \limits_{k=1}^n \left( \prod_{j=1}^{k-1}\lam_{\om_j} \right) t_{u_k}^{(\om_k)} = f^{(\om)}_{u|n}(0).
\end{equation}

For each $i \in I$, let $p_i = (p_1^{(i)}, \ldots, p^{(i)}_{k_i})$ be a probability vector with strictly positive entries. For $\om \in \Om$, define a probability measure on $\X_\infty^{(\om)}$ as the infinite product
\[ \overline{\eta}^{(\om)} := \bigotimes_{n=1}^{\infty} p_{\om_n}. \]
Our main object of interest are the projections
\[ \eta^{(\om)} := \Pi_{\om} \overline{\eta}^{(\om)}. \]
These measures are not self-similar, but they satisfy a ``dynamic self-similarity'' relation
\begin{equation}\label{eq: one step dyn self-sim}
\eta^{(\om)} = \sum \limits_{u \in \X_1^{(\om)}} p_u^{(\om_1)} \cdot f_u^{(\om)}\eta^{(T\om)}, 
\end{equation}
where $T : \Om \to \Om$ is the left shift. Iterating this relation, we obtain for each $k \in \N$:
\begin{equation}\label{eq:dyn self-sim}
\eta^{(\om)} = \sum \limits_{u \in \X_k^{(\om)}} p_u^{\pom} \cdot f_u^{(\om)}\eta^{(T^k\om)} {  = \sum \limits_{u \in \X_k^{(\om)}} p_u^{\pom} \left( \lam_{\om_1} \cdots \lam_{\om_k} \eta^{(T^k\om)} + t_u^{\pom} \right),}
\end{equation}
where we set $p_u^{\pom} = p_{u_1}^{(\om_1)} \cdots p_{u_k}^{(\om_k)}$.

Let $\PP$ be a Borel probability measure on $\Om$. Following \cite{SSS}, we call the triple
\[\Sigma = \left( (\Phi^{(i)})_{i \in I}, (p_i)_{i \in I}, \PP \right)\] a \textit{(planar) model} with \textit{selection measure} $\PP$. We will say that the measures $\eta^{(\om)}$ are {\em
 random measures generated by the model $\Sigma$}.
%\begin{defn}
%We say that a model $\Sigma$ is \textit{non-degenerate} if there exists $i \in I$ and $1 \leq j < j' \leq k_i$ with $t^{(i)}_j \neq t^{(i)}_{j'}$.
%\end{defn}
%By $B(x,r)$ we will denote an open ball in the Euclidean space with center $x$ and radius $r$. 
If $\PP$ is $T$-invariant and ergodic, then the measures $\eta^{(\om)}$ are almost surely exact dimensional, with a constant value of the dimension:

\begin{prop}\label{prop:exact dimension}
Let $\Sigma$ be a model for which $\PP$ is $T$-invariant and ergodic. Then there exists $\alpha \in [0,2]$ such that for $\PP$-almost every $\om$, the measure $\eta^{(\om)}$ has exact dimension $\alpha$, i.e.
\[ \lim \limits_{r \to 0} \frac{\log (\eta^{(\om)}(B(x,r)))}{\log r}  = \alpha \]
for $\eta^{(\om)}$-almost every $x \in \R^2$. 
\end{prop}

The value $\alpha$ above is called the \textit{dimension} of the model $\Sigma$ and we will denote it by $\dim(\Sigma)$. Proposition \ref{prop:exact dimension} extends \cite[Theorem 1.2]{SSS} to the plane. As the proof from \cite{SSS} extends without any changes, we omit it.

Given a model $\Sigma$, we define its \textit{similarity dimension} as
\begin{equation}\label{eq:sdim def} \sdim(\Sigma) := \frac{\int \limits_{\Om}H(p_{\om_1})d\PP(\om)}{-\int \limits_{\Om}\log r_{\om_1}d\PP(\om)} =  \left( \int \limits_{\Om} \log (r_{\om_1}) d\PP(\om) \right)^{-1} \int \limits_{\Om} \sum \limits_{j=1}^{k_{\om_1}} p_j^{(\om_1)}\log p_j^{(\om_1)} d\PP(\om),
\end{equation}
where $H(p) = -\sum p_j \log p_j$ denotes the Shannon entropy of a probability vector $p = (p_j)$. Here and throughout the paper we use logarithms in base $2$. Note that if $\PP$ is the Bernoulli measure with a marginal $q = (q_i)_{i \in I}$, then
\begin{equation}\label{eq:sdim Bernoulli} \sdim(\Sigma) = \frac{\sum \limits_{i \in I} q_i H(p_i)}{- \sum \limits_{i \in I} q_i \log r_i}. 
\end{equation}
For a given $\om \in \Om$ and $u,v \in \X_n^{(\om)}$ define
\[ d^{(\om)}(u,v) := |f^{(\omega)}_u(0) - f^{(\omega)}_v(0)|\]
and
\begin{equation}\label{eq:delta def} \Delta_n^{(\om)} = \Delta_n^{(\om)}(\Sigma)  :=  \begin{cases} \min\{ d^{(\om)}(u,v) : u,v \in \X_n^{(\om)}, u \neq v\} & \text{ if } |\X_n^{(\om)}| > 1 \\
 0 & \text{ if } |\X_n^{(\om)}| = 1.  \end{cases}\end{equation}

Below we state our main result on the dimension of measures $\eta^{(\om)}$, extending \cite[Theorem 1.3]{SSS} to the plane, in the spirit of the results of \cite{HRd} for deterministic self-similar measures. It shows that if $\dim(\Sigma) < \min\{2, \sdim(\Sigma)\}$, then the system has a super-exponential condensation of cylinders (in probability). In the case of deterministic self-similar measures on the plane, such alternative requires, in general, an assumption that there is a non-trivial rotation in the system (see \cite{HRd}), hence we make a similar assumption in the random case. Moreover, the result is trivial in the degenerate case (i.e., with all $\Phi_i$ consisting of a single map -- then all $\eta^{(\om)}$ are singletons), hence our standing assumption will be the following one:

\begin{assumption}\label{assmp:non-deg and rotation}
	The selection measure $\PP$ is $T$-invariant and ergodic. Moreover, 
	\begin{enumerate}[(i)]
		\item\label{it:non-degenerate}  there exists $i_0 \in I$ such that $\PP(\om_1 = i_0) > 0$ and not all maps in $\Phi^{(i_0)}$ are equal,
		\item\label{it:non-real} there exists $i_1 \in I$ such that $\PP(\om_1 = i_1) > 0$ and $\vphi_{i_1} \notin \R$.
	\end{enumerate}
\end{assumption}

We will say that model $\Sigma$ is \textit{non-degenerate} if the item \ref{it:non-degenerate} of Assumption \ref{assmp:non-deg and rotation} holds. Our main dimension result is the following:

\begin{thm}\label{thm:main dim} Let $\Sigma$ be a model satisfying Assumption \ref{assmp:non-deg and rotation}. If $\dim(\Sigma) < \min\{ 2, \sdim(\Sigma)\}$, then
\[ \frac{\log \Delta_n^{(\cdot)}}{n} \stackrel{\PP}{\longrightarrow} -\infty, \]
i.e. for every $M>0$,
\[ \PP\left( \left\{ \om \in \Om : \Delta_n^{(\om)} \leq e^{-Mn} \right\} \right) \stackrel{n \to \infty}{\longrightarrow} 1. \]
\end{thm}

Finally, we state our result on the Fourier decay of random measures $\eta^{\pom}$.
Let $\Pk(\R^2)$ denote the set of Borel probability measures on $\R^2 \cong \C$. 
The Fourier transform of $\mu\in \Pk(\R^2)$ is
$$
\what\mu(\xi) = \int_\C e^{2\pi i  \Re(z\ov{\xi})}d\mu(z).
$$
Denote
\[
\Dk_2(\sigma) = \{\mu\in \Pk(\R^2):\, |\widehat{\mu}(\xi)| = 
O_\mu(|\xi|^{-\sigma})\}\quad \text{and}\quad \Dk_2 = \bigcup_{\sig>0} 
\Dk_2(\sig).
\]
The following result extends \cite[Theorem 1.5]{SSS} to the plane. For a complex number $\beta$, we take $\lam^\beta := e^{\beta \log \lam}$, where $\log$ is the principal branch of logarithm, taking values in  $\R + (-\pi, \pi]i$ and continuous on $\C\setminus (-\infty,0]$.

\begin{thm}\label{thm:fourier main}
Let $I$ be a finite set and fix non-zero complex numbers $\beta_i, i \in I$. Let $\PP$ be a Bernoulli measure on $\Om = I^\N$. Then there exists a Borel set $\Gk \subset \Om \times \D$ such that
\begin{enumerate}[(i)]
	\item For $\PP$-almost every $\om \in \Om$,
	\[ \hdim\left( \left\{ \lam \in \D_*: (\om, \lam) \notin \Gk \right\} \right) \leq 1;\]
	\item if $(\om, \lam) \in \Gk$ and $\Sigma = \left( (\Phi^{(i)})_{i \in I}, (p_i)_{i \in I}, \PP \right)$ is any non-degenerate model with selection measure $\PP$, such that the linear part of the maps in $\Phi^{(i)}$ is $\lam^{\beta_i}$ with $|\lam^{\beta_i}| < 1$, then $\eta^{\pom} \in \Dk_2$.
\end{enumerate}
\end{thm}

%%%%%%%%%%%%%%%%%%%%%%

\section{Preliminaries}

Much of the notation follows \cite{HRd}. We will use the Landau's $O(\cdot)$ and $\Theta(\cdot)$ notation with $A = O(B)$ meaning $A \leq CB$  and $A = \Theta(B)$ meaning $cA \leq B \leq CA$ for some constants $c, C > 0$. If the constant is allowed to depend on a parameter, we indicate that in a subscript, e.g. $A = O_R(B)$ means $A \leq C(R)B$. We will usually not include the dependence on the model $\Sigma$ in the notation, but all the constants are allowed to depend on it. A statement  of the form ``the given property holds for all $n$ and $m > m(n)$'' is a shorthand for saying that for each $n$ there exists $m(n)$ such that the given property holds for all $m > m(n)$. 

\subsection{Entropy and component measures}

For a metric space $X$, we denote by $\mP(X)$ the set of all Borel probability measures on $X$. For a countable measurable partition $\Ek$ of $X$, the (Shannon) entropy of a probability measure $\mu$ with respect to $\Ek$ is defined as
\[ H(\mu, \Ek) := - \sum \limits_{E \in \Ek} \mu(E) \log \mu(E)\]
(with the convention $0 \log 0 =0$). The conditional entropy with respect to a countable partition $\Fk$ is
\[ H(\mu, \Ek | \Fk) = \sum \limits_{F \in \Fk} \mu(F)H(\mu_F, \Ek), \]
where $\mu_F = \frac{1}{\mu(F)}\mu|_F$. A useful property of entropy is the following: If $\Ek,\Fk$ are two countable partitions such that every element of $\Ek$ intersects at most $k$ elements of $\Fk$ and vice versa, then
\begin{equation}\label{eq:intersecting partitions} \left| H(\mu, \Ek) - H(\mu, \Fk)  \right| = O(\log k).
\end{equation}
Let $\Dk_n$ be the standard partition of $\R^d$ into dyadic cubes of side-length $2^{-n}$, i.e., sets of the form
\[ [k_1 2^{-n}, (k_1 + 1)2^{-n})  \times \cdots \times [k_d 2^{-n}, (k_d + 1)2^{-n}), k_1, \ldots, k_d \in \Z.\]
For a probability measure $\mu$ on $\R^d$, we denote
\[H_n(\mu) := \frac{1}{n}H(\mu, \Dk_n).\]
It follows from \eqref{eq:intersecting partitions}, that if $f(z) = \lam z + t$ with $|\lam| = \Theta(2^{-k})$, then
\begin{equation}\label{eq:entropy bounded similarity} \left| H(\mu, \Dk_n) - H(f\mu, \Dk_{n+k})  \right| = O(1).
\end{equation}
For a pair of probability measures $\mu,\nu \in \Pk([-R,R]^d)$ we denote by $\mu * \nu$ their convolution. We have (see \cite[Proposition 4.1.v]{SSS}):
\begin{equation}\label{eq:conv entropy bound}
	H(\mu * \nu, \Dk_n) \geq H(\mu, \Dk_n) - O_R(1).
\end{equation}
We emphasize that throughout the paper, $\Dk_n$ will denote both the partition of $\R$ and of $\R^2$, depending on the context. For $x \in \R^d$ we will write $D_n(x)$ for the element of $\Dk_n$ containing $x$. For $D \in \Dk_n$, we will denote by $T_D$ the unique similarity $t\mapsto 2^n t + a_D$ mapping $D$ onto the unit cube.

\begin{defn} \label{def:raw}
	For a probability measure $\mu$ on $\R^d$ and a dyadic cell $D \in \Dk_n$ satisfying $\mu(D) > 0$, the \textit{raw $D$-component} of $\mu$ is
	\[ \mu_D := \frac{1}{\mu(D)}\mu|_D \]
	and the \textit{rescaled $D$-component}  is
	\[ \mu^D := T_D \left( \mu_D \right).\]
	For $x \in \R^d$ with $\mu(D_n(x)) >0$ we write $\mu_{x,n} := \mu_{D_n(x)}$ and $\mu^{x,n} = \mu^{D_n(x)}$. These are called the $n$-level components. For measures $\eta^{\pom}$ generated by a model $\Sigma$, we use notation $\eta^{\pom, x, n} := \left( \eta^{\pom} \right)^{x,n}$.
\end{defn}

We will often consider \textit{randomly chosen components}, in the following manner. For a bounded measurable function $f : \Pk(\R^d) \to \R$ and a finite set $J \subset \N$ we set
\[ \bE_{i \in J}(f(\mu_{x,i})) := \frac{1}{\#J} \sum \limits_{i \in J} \int \limits_{\R^d} f(\mu_{x,i})d\mu(x), \]
and $\bP_{i \in J}(\mu_{x,i} \in A) := \bE_{i \in J} \mathds{1}_{A}(\mu_{x,i})$. In other words, a level $i$ is chosen uniformly among $J$ and then an $i$-level component $\mu_{x,i}$ is chosen independently, with $x$ drawn according to $\mu$. We will also consider analogous expressions for the rescaled components. For example,
\begin{equation}\label{eq:expect}
\bE_{0 \leq i \leq n} \left( H_m(\mu^{x,i}) \right)  = \frac{1}{n+1} \sum \limits_{i=0}^n \sum \limits_{D \in \Dk_i} \mu(D) H_m(\mu^D).
\end{equation}
We will also sometimes use the notation $\bP_{i \in J}$ to denote a uniform probability over the set $J$, for example
\[ \bP_{0 \leq k \leq n} \left( \bE_{i = k} H_m(\eta^{(T^k \om)}_{x,i}) > \alpha)  \right) = \frac{1}{n+1}\cdot  \#\left\{ 0 \leq k \leq n : \bE_{i = k} H_m(\eta^{(T^k \om)}_{x,i}) > \alpha \right\}.  \]
Throughout the paper, $\bP$ and $\bE$ will denote the probability and expectation defined above, whereas $\PP$ and $\E$ will denote the selection measure of the model $\Sigma$ and the expectation corresponding to it.

The next lemma expresses the important ``global entropy from local entropy'' principle.

\begin{lem}{\cite[Lemma 3.4]{H}}\label{lem: local to global entropy} For $R \geq 1,\ \mu \in \mP([-R,R]^d)$ and integers $m < n$,
	\[ H_n(\mu) = \bE_{0 \leq i \leq n} \left( H_m(\mu^{x,i}) \right) + O\left( \frac{m + \log R}{n} \right) = \frac{1}{n} \sum \limits_{k=0}^{n-1} \frac{1}{m} H(\mu, \Dk_{k+m}|\Dk_k) + O\left( \frac{m + \log R}{n} \right). \]
\end{lem}

Note that the second equality in the last lemma is simply an application of \eqref{eq:expect}.

\subsection{Orthogonal projections and saturation}

We will write $W \leq \R^d$ to denote that $W$ is a linear subspace of $\R^d$, and by $W^\perp$ we will denote the orthogonal complement of $W$ in $\R^d$. The symbol $\R\PP^1$ stands for the projective line, which we identify with the  set of  $1$-dimensional linear subspaces of the plane. Formally, we will treat elements of $\R\PP^1$ both as $1$-dimensional linear subspaces of $\R^2$ and elements of the group $ \s^1/ \{ -1, 1\}$. 

For $W \leq \R^d$, let $\pi_W : \R^d \to W$ denote the orthogonal projection onto $W$. Let $\Dk_n^{W} := \pi_W^{-1}(\Dk_n)$ be the partition of $\R^d$ induced by the dyadic partition of $W$.
Observe that, by definition,\footnote{Formally, in order to consider dyadic partitions of $W$, one has to fix a basis in $W$. However, the particular choice does not influence any of our calculations as long as we choose a basis which is orthonormal in the standard inner product on $\R^d$, as entropy $H(\mu, \Dk_n)$ on $W$ for any two such bases will differ by at most a universal constant (independent of $n$) by \eqref{eq:intersecting partitions}.}
\begin{equation}\label{eq:entpro}
H(\pi_W\mu,\Dk_m) = H(\mu, \Dk_m^W).
\end{equation}
\noindent Moreover, let $\Dk^{W \oplus W^\perp}_n := \Dk_n^{W} \vee \Dk_n^{W^\perp}$ be the dyadic partition of $\R^d$ parallel to $W$. Note that for  $\mu \in \Pk(\R^d)$,
\begin{equation}\label{eq:entpro2}
|H(\mu, \Dk_n) - H(\mu, \Dk^{W \oplus W^\perp}_n)| = O \left( 1\right),
\end{equation}
and for $\mu \in \mP\left([-R, R]^d\right)$,
\begin{equation}\label{eq:entropy cond W}
\begin{split}
\frac{1}{n}H(\mu, \Dk^{W \oplus W^\perp}_n | \Dk^W_n) & = \frac{1}{n}H(\mu, \Dk^{W^\perp}_n | \Dk^W_n) \\
& = \frac{1}{n}H(\mu, \Dk_n | \Dk^W_n) + O(1/n) \\
& \leq \dim(W^\perp) + O_R(1/n).
\end{split}
\end{equation}
Recall the following definitions from \cite{HRd}:

\begin{defn}
	Let $V \leq \R^d$ be a linear subspace and $\eps > 0$. A measure $\mu \in \mP(\R^d)$ is \textit{$(V,\eps)$-concentrated} if there is a translate $W$ of $V$ such that $\mu(W^{(\eps)}) \geq 1 - \eps$, where $W^{(\eps)} = \{ x \in \R^d : \dist(x,W) < \eps \}$ is the $\eps$-neighborhood of $W$.
\end{defn}

\begin{defn} \label{def:saturated}
	Let $V \leq \R^d$ be a linear subspace and $\eps > 0$. A measure $\mu \in \mP(\R^d)$ is \textit{$(V,\eps)$-saturated at scale $m$}, or $(V,\eps,m)$-saturated, if
	\[ H_m(\mu) \geq H_m(\pi_{V^\perp}\mu) + \dim V - \eps. \]
\end{defn}

We will make use of several lemmas, proved in \cite{H,HRd}. The first one is immediate from \eqref{eq:entropy cond W} and Definition~\ref{def:saturated}. It shows that $(V,\eps,m)$-saturation (for small $\eps$) means (at least morally) that conditional measures of $\mu$ with respect to the projection $\pi_{V^\perp}$ have (on average) almost full entropy in scale $m$.

\begin{lem}\cite[Lemma 3.9]{HRd}\label{lem:sat by cond}
	A measure $\mu \in \mP(\R^d)$ is $(V, \eps+O(1/m), m)$-saturated if and only if
	\[ \frac{1}{m}H(\mu, \Dk_m^V | \Dk_m^{V^\perp}) \geq \dim V - \left(\eps + O(1/m)\right). \]
\end{lem}

\begin{lem}\cite[Lemma 3.21.5]{HRd}\label{lem: saturation on sum} If $\mu\in\mP([0,1)^d)$ is both $(V_1, \eps, m)$ and $(V_2,\eps,m)$-saturated, and $\angle(V_1, V_2) > \delta > 0$, then $\mu$ is $(V_1 + V_2, \eps', m)$-saturated, where $\eps' = 2\eps + O(\frac{1}{m} \log \frac{1}{\delta})$. 
\end{lem}

\subsection{Inverse theorem for entropy}

The following is the Hochman's inverse theorem for convolutions in $\R^d$ \cite{H, HRd}. This deep result is the central tool in proving Theorem \ref{thm:main dim}.

\begin{thm}\cite[Theorem 2.8]{HRd}\label{thm: inverse convolution}
	For every $R,\eps > 0$ and $m \in \N$ there is $\delta = \delta(\eps, R, m) > 0$ such that for every $n > n(\eps, R, \delta, m)$, the following holds: If $\mu, \nu \in \mP([-R, R]^d)$ and
	\[ H_n(\mu * \nu) < H_n(\mu) + \delta, \]
	then there exists a sequence $V_0, \ldots, V_n \leq \R^d$ of subspaces such that
	\[ \bP_{0 \leq i \leq n} \begin{pmatrix} \mu^{x,i} \text{ is } (V_i, \eps, m)\text{-saturated and }\\
	\nu^{y,i} \text{ is } (V_i,\eps)\text{-concentrated} \end{pmatrix} > 1 - \eps, \]
where for each $i$, the random components $\mu^{x,i}, \nu^{y,i}$ are chosen independently of one another.
\end{thm}

\subsection{Projective cocycle}

Although the map $(\om, V) \mapsto H_m(\pi_V \eta^{(\om)})$ is not continuous on $\Om \times \R\PP^1$, it is close to being continuous, in the sense of the two lemmas below.

\begin{lem}\label{lem:V cont}
Let $\mu \in \mP\left([0,1)^2\right)$. For $W,V \in \R\PP^1$, if $\|\pi_W - \pi_V\| < 2^{-m}$, then
\[ |H_m(\pi_W \mu) - H_m(\pi_V \mu)| \leq O(1/m). \]
\end{lem}

\begin{proof}
Follows directly from \cite[Lemma 3.2.3]{HRd}.
\end{proof}

For a given model $\Sigma$, let $R = R(\Sigma) > 0$ be large enough to guarantee \begin{equation}\label{eq:R def} \overline{f^{(i)}_j\left( B(0,R) \right)} \subset B(0,R) \text{ for all } i \in I,\ j \in \{1, \ldots, k_i\}.
\end{equation}
Then $\supp\left(\eta^{(\om)}\right) \subset B(0,R)$ for every $\om \in \Om$. We will also denote $B_u^{(\om)} = f^{\om}_{u}(B(0,R))$ for $\om \in \Om$ and $u \in \X^{(\om)}_n$. For a pair of infinite sequences $\om, \tau \in \Om$, denote by $\om \wedge \tau$ the longest common prefix of $\om$ and $\tau$.

\begin{lem}\label{lem:om cont}
Fix a model $\Sigma$. If $\om, \tau \in \Om$ are such that $|\om \wedge \tau| \geq O(m)$, then
\[ \sup \limits_{W \in \R\PP^1} \left| H_m(\pi_W \eta^{(\om)}) - H_m(\pi_W \eta^{(\tau)}) \right| \leq O(1/m). \] 
\end{lem}
\begin{proof}
Let $k:= |\om \wedge \tau|$ be large enough to have $\diam(f_u^{(\om \wedge \tau)}(B(0,R))) < 2^{-m}$ for every $u \in \X^{(\om \wedge \tau)}_k$, where $R$ is defined in \eqref{eq:R def}. As the systems $\Phi^{(i)}$ are uniformly contracting, this will hold provided $k \geq O(m)$ (uniformly in $\om, \tau$). Let $\nu^{(\om, k)}, \nu^{(\tau, k)}$ be, respectively, images of measures $\overline{\eta}^{(\om)}, \overline{\tau}^{(\om)}$ via the $k$-truncated coding maps $\Pi^{(k)}_\om, \Pi^{(k)}_\tau$ defined in \eqref{eq:truncated pi}. As $\om$ and $\tau$ agree on the first $k$ coordinates, we have $\nu^{(\om, k)} = \nu^{(\tau, k)}$, thus
\begin{equation}\label{eq:truncated entropies}  H_m(\pi_W \nu^{(\om, k)}) = H_m(\pi_W \nu^{(\tau, k)}) \text{ for every } W \in \R\PP^1.
\end{equation}
On the other hand, we have $\| \Pi_\om(u) -  \Pi^{(n)}_\om (u)\| \leq 2^{-m}$ for every $u \in \X^{\pom}_\infty$ and $\| \Pi_\tau(u) -  \Pi^{(n)}_\tau (u)\| < 2^{-m}$ for every $u \in \X^{(\tau)}_\infty$, 
so (similarly to \cite[Lemma 3.2.3]{HRd}),
\[ \sup \limits_{W \in \R\PP^1} \left| H_m(\pi_W \eta^{(\om)}) - H_m(\pi_W\nu^{(\om, k)}) \right|, \sup \limits_{W \in \R\PP^1} \left| H_m(\pi_W \eta^{(\tau)}) - H_m(\pi_W\nu^{(\tau, k)}) \right| \leq O\left(\frac{1}{m}\right) .\]
Combining this with \eqref{eq:truncated entropies} finishes the proof.
\end{proof}

	We will also need the following  basic lemma on the effect of shifting $\om$ on the entropy:
	
	\begin{lem}\label{lem:shifted entropy comp}
		Fix a model $\Sigma$. Then for every $\om \in \Om, j,m \in \N$,
		\[
		\sup \limits_{W \leq \R^2} \left| H_m(\pi_W \eta^{\pom}) - H_m(\pi_W \eta^{(T^j \om)}) \right| \leq O(j/m).
		\]
	\end{lem}
	
	\begin{proof}
		Fix $W \leq \R^2$. By the self-similarity relation \eqref{eq:dyn self-sim},
		\begin{equation}\label{eq:self-similarity} \pi_W \eta^{\pom} = \sum \limits_{u \in \X^{\pom}_{j}} p_u^{\pom}  \pi_W f_u^{\pom} \eta^{(T^j \om)}.
		\end{equation}
		As each $f_u^{\pom}$ above is a similarity contracting by at most $r_{\min}^{j} = O(2^{-cj})$ for some $c = c(\Sigma)$, we get by \eqref{eq:intersecting partitions} that
		\[ \left| H_{m}(\pi_W f_u^{\pom} \eta^{(T^j \om)}) - H_{m}( \pi_W \eta^{(T^j \om)}) \right| \leq O\left( j/m \right). \]
		Combining this with \eqref{eq:self-similarity} and the convexity bound for entropy (see e.g.\ \cite[eq.\ (20)]{SSS}) gives
		\[
		\begin{split} \Big| H_{m} & (\pi_W \eta^{\pom} )  -  H_{m}(\pi_W \eta^{(T^j \om)}) \Big|\\
		& \leq \left| H_m \left( \sum \limits_{u \in \X^{\pom}_{j}} p_u^{\pom} \pi_W f_u^{\pom} \eta^{(T^j \om)} \right) - \sum \limits_{u \in \X^{\pom}_{j}} p_u^{\pom}  H_m \left(\pi_W f_u^{\pom} \eta^{(T^j \om)} \right) \right|  + O(j/m)\\
		&\leq \frac{1}{m}H\left(  \left(p_u^{\pom}\right)_{u \in \X^{\pom}_{j}} \right) + O(j/m) \\
			&  \leq \frac{\log \left( \max\{k_i : i \in I \} \right)^{j}}{m} + O(j/m) =O(j/m).
		\end{split}
	\]
	\end{proof}

\begin{lem}\label{lem:Delta inequality} For every $\om \in \Om,\ n,j \in \N$ inequality $\Delta^{(\om)}_{n + j} \leq \Delta^{(\om)}_{n}$ holds provided $|\X^{(\om)}_n| > 1$.
\end{lem}

\begin{proof}
	Take $u,v \in \X^{(\om)}_n$ such that $u \neq v$ and $\Delta^{(\om)}_{n} = |f^{(\om)}_u(0) - f^{(\om)}_v(0)|$. Take any $w \in \X^{(T^n \om)}_j$. Then $uw,vw \in \X^{(\om)}_{n+j}$ and
	\[\begin{split}
		\Delta^{(\om)}_{n+j} &  \leq |f^{(\omega)}_{uw}(0) - f^{(\omega)}_{vw}(0)| =  | f^{(\omega)}_{u} \circ f^{(T^n\omega)}_w(0) - f^{(\omega)}_{v} \circ f^{(T^n\omega)}_w(0)|  \\ & = |\lam_{\om_1}\cdots \lam_{\om_n} f^{(T^j\omega)}_w(0) + t^{(\om)}_u - \lam_{\om_1}\cdots \lam_{\om_n} f^{(T^j\omega)}_w(0) - t^{(\om)}_v| \\
		&  = |t^{(\om)}_u - t^{(\om)}_v| =  |f^{(\omega)}_{u}(0) - f^{(\omega)}_{v}(0)| =  \Delta^{(\om)}_{n}.
	\end{split}\]
\end{proof}

Given a model $\Sigma = \left( (\Phi^{(i)})_{i \in I}, (p_i)_{i \in I}, \PP \right)$, let $G_\Sigma \subset \R\PP^1$ be its \textit{rotation group}, i.e. the closed group generated by the set of rotations $\{ \vphi_i : i \in I,\ \PP(\om_1 = i) > 0 \} \subset  \s^1= \R\PP^1$. Note that if $\vphi_i = e^{2\pi i \rho}$ with $\rho \notin \Q$ for some $i \in I$, then $G_\Sigma = \R\PP^1$, and $G_\Sigma$ is finite otherwise. Let $ \fm$ be the Haar (probability) measure on $G_{\Sigma}$ (so $ \fm$ is the Lebesgue measure if $G_\Sigma = \R\PP^1$ and a uniform measure on the finite set $G_\Sigma$ otherwise). Moreover, for each $W \in \R\PP^1$, define a Borel probability measure $ \fm^{(W)}$ on $G_{\Sigma}W$ by setting
\begin{equation}\label{eq:xi W}  \fm^{(W)} = W \fm.
\end{equation}
Note that under Assumption \ref{assmp:non-deg and rotation}\ref{it:non-real}, measures $ \fm^{(W)}$ are not supported on a single atom. We will make use of the following ergodic property:
\begin{lem}\label{lem:skew product ergodic}
	Let $\PP$ be $T$-invariant and ergodic. Then for $\PP$-a.e.\ $\om \in \Om$,
	\begin{equation}\label{eq:skew product ergodic} \frac{1}{n}\sum \limits_{k=0}^{n-1} \delta_{\left(T^k\om,  \vphi^{-1}_{\om_k} \cdots \vphi^{-1}_{\om_1}  W \right)} \stackrel{w^{*}}{\longrightarrow} \PP \otimes  \fm^{(W)},\ \text{ uniformly in } W\in\R\PP^1.
	\end{equation} 
\end{lem}

\begin{proof}
	Consider the skew-product $F : \Om \times G_\Sigma \to \Om \times G_\Sigma,\ F(\om, V) = (T \om, \vphi^{-1}_{\om_1}V)$. Note that $ \fm$ is a stationary measure for $F$, i.e. it satisfies
	\[	 \fm = \int \limits_{\Om} \vphi_{\om_1}^{-1} \fm d\PP(\om) = \sum \limits_{i \in I} \PP(\om_1 = i) \vphi^{-1}_i  \fm. \]
	Moreover, it is a standard observation that $ \fm$ is the unique probability measure on $G_\Sigma$ with this property.\footnote{If $G_\Sigma$ contains an irrational rotation, this can be verified e.g.\ using the Fourier coefficients of $ \fm$. Otherwise, $G_\Sigma$ is finite and one can use the fact that it is equal to the closed semigroup generated by the set $\{ \vphi_i: i \in I, \PP(\om_1 = i) > 0 \}$, see e.g.\ \cite[Theorem 1]{Wright}.} It is therefore an ergodic stationary measure for $F$ (in the sense of \cite[Chapter 5]{VianaLec}, see \cite[Theorem 5.14]{VianaLec}) and hence the product measure $\PP \otimes  \fm$ is ergodic (and invariant) for the the skew-product $F$ by \cite[Proposition 5.13]{VianaLec}. It follows now from \cite[Lemma 3.3]{GSSY} that for $\PP$-a.e.\ $\om \in \Om$ we have 
	\begin{equation}\label{eq:skew product ergodic xi}\frac{1}{n}\sum \limits_{k=0}^{n-1} \delta_{\left(T^k\om, \vphi^{-1}_{\om_k} \cdots \vphi^{-1}_{\om_1} V \right)} \stackrel{w^{*}}{\longrightarrow} \PP \otimes  \fm^{},\ \text{ uniformly in } V\in G_\Sigma.
	\end{equation}
	This extends to \eqref{eq:skew product ergodic}, by applying \eqref{eq:xi W} to \eqref{eq:skew product ergodic xi} together with the fact that the weak$^*$ topology on the set of Borel probability measures on $\Om \times \R\PP^1$ can be metrized by a metric which is rotation-invariant on the second coordinate.
\end{proof}

\subsection{Dynamical time change}

Let us introduce some useful notation. We can assume that $R = R(\Sigma)>0$ defined by \eqref{eq:R def} satisfies $2Rr_{\min} > 1/2$. For a given $\om \in \Om$ and $k = 0,1,2 \ldots$,  define $k' = k'(\om) = k'(\om, k)$ as the unique natural number satisfying
\begin{equation}\label{eq:prime def} 2R \prod \limits_{i=1}^{k'(\om)} r_{\om_i} \leq 2^{-k} \leq 2R \prod \limits_{i=1}^{k'(\om) - 1} r_{\om_i}.
\end{equation}
Note that $\diam \left( \supp \left( f^{(\om)}_u \eta^{(T^{k'} \om)} \right)  \right) \leq 2^{-k}$ for every $u \in \X^{(\om)}_{k'}$.  It follows easily from the definition that $n'(\om)$ is an ``almost additive cocycle'', i.e. there exists $c > 0$ depending only on the model, such that for every $\om \in \Om, n,k \in \N$, we have
\begin{equation}\label{eq:cocycle}
	|n'(\om) + k'(T^{n'}\om) - (n+k)'(\om)| \leq c.
\end{equation}
We will need the following simple lemma. 

\begin{lem}\label{lem:k k' comparison}
	There exists $C\geq 1$ depending only on the model, such that for every $\om \in \Om$ and $A \subset \N$, 
\begin{equation}\label{eq:freq general upper bound}
\limsup \limits_{n \to \infty} \frac{1}{n} \#\{ 0 \leq k < n: k'(\om) \in A  \} \leq C \limsup \limits_{n \to \infty} \frac{1}{n}  \#\{ 0 \leq k  < n: k \in A  \}.
\end{equation}
	Assume additionally that the model satisfies $r_{\mathrm{max}} \leq 1/2$.  Then $C$ can be taken such that for $n \geq 1$,
\begin{equation}\label{eq:freq half upper bound}
\#\{ 0 \leq k < n/C: k \in A  \} - C \leq \#\{ 0 \leq k < n: k'(\om) \in A  \}.
\end{equation}
\end{lem}
\begin{proof}
For integers $a \leq b$ denote $[a,b] = \{a, a+1, \ldots, b\}$. Fix $\om \in \Om$ and define $g : \N \to \N$ by $g(k) = k'(\om)$ and note that it is a non-decreasing map. Moreover, there exist integers $B$ and $M$ depending only on the model, such that
\begin{equation}\label{eq:g properties}
g(0) \leq M,\ g([0,n]) \subset [0, Bn + M] \text{ and } \#g^{-1}(\{n\}) \leq B \text{ for every } n \geq 0.
\end{equation}
The first two properties follow from $r_{\max} < 1$ (and the fact that $R$ depends only on the model $\Sigma$) and the last one from $r_{\min} > 0$. Therefore
\[ \{ 0 \leq k < n: k'(\om) \in A  \} = [0,n-1] \cap g^{-1}\left( A \cap g\left( [0, n-1] \right) \right) \subset g^{-1}\left( A \cap [0, B(n-1) + M] \right),\]
hence
\[ \# \{ 0 \leq k < n: k'(\om) \in A  \} \leq B \# \{ 0 \leq k < Bn: k \in A  \} + B (M+1).\]
This yields \eqref{eq:freq general upper bound}. If we assume that $r_{\mathrm{max}} \leq 1/2$, then $g$ has an additional property: $\# g^{-1}(\{k\}) \geq 1$ for every $k \in g(\N) = [g(0), \infty)$. Combining this with \eqref{eq:g properties} gives
\[
\begin{split} \# \{ 0 \leq k < n: k'(\om) \in A  \}  & = \!\!\!\sum \limits_{j \in [g(0), g(n-1)] \cap A}\!\!\!
 \#\left(g^{-1}(\{ j\}) \cap [0, n-1] \right) \geq \!\!\!\sum \limits_{j \in [g(0), g(n-1) - 1] \cap A} \!\!\!\#\left(g^{-1}(\{ j\}) \right)\\[1.1ex]
& \geq \# \left( [g(0), g(n-1) - 1] \cap A \right) \geq \# \left( [0, g(n-1)] \cap A \right) - (g(0) + 1) \\
& \geq \# \left( \left[0, \frac{n-1}{B}\right] \cap A \right) - (M + 1).
\end{split}
\]
This proves \eqref{eq:freq half upper bound}.
\end{proof}

\section{Excluding saturation on subspaces of positive dimension}

In this section we make the main preparations which will allow us to apply Hochman's inverse theorem (Theorem \ref{thm: inverse convolution}) in order to prove Theorem \ref{thm:main dim}. To avoid repetition of the assumptions:\\

\textbf{For the rest of the paper, we assume that $\Sigma$ is a model satisfying Assumption \ref{assmp:non-deg and rotation}.}\\

Let us first discuss the scheme of the proof of Theorem \ref{thm:main dim}. It follows the strategy of \cite{SSS}, which relies on an application of Theorem \ref{thm: inverse convolution} to the family of random measures $\eta^{(\om)}$  via the following decomposition, which is a direct consequence of \eqref{eq:dyn self-sim}:
\begin{equation}\label{eq:convolution decomposition} \eta^{(\om)} = \nu^{(\om, n'(\om))} * \tau^{(\om, n'(\om))},
\end{equation}
where for $\om\in \Om,\ n \in \N$, we denote by $\nu^{(\om, n)}$ the image of the measure $\overline{\eta}^{(\om)}$ via the $n$-truncated coding map $\Pi^{(n)}_\om$ defined in \eqref{eq:truncated pi}, i.e. $\nu^{(\om, n)} =  \sum \limits_{u \in \X^{(\om)}_{n}} p^{(\om)}_u \delta_{f^{(\om)}_u(0)}$, and
\[ \tau^{(\om, n)} := \lam_{\om_1} \cdots \lam_{\om_n}\eta^{(T^n \om)}. \]
The idea is to exclude the possibility that subspaces $V_i$ produced by Theorem \ref{thm: inverse convolution} satisfy $\dim V_i > 0$ with positive frequency. This will give an upper bound on the entropy of $\nu^{(\om, n'(\om))}$ and yield information on super-exponential concentration of its atoms in scale $2^{-n}$ if the dimension drop occurs. As Theorem \ref{thm:main dim} holds trivially if $\alpha := \dim(\Sigma) = 2$, we can assume $\alpha < 2$. Then, the case $\dim V_i = 2$ will be excluded using this assumption together with the uniform entropy dimension properties of $\eta^{(T^{n'} \om)}$ (Proposition \ref{prop:unif ent dim}) --- this portion of the proof is similar to the one-dimensional case of \cite{SSS}, with some technical adjustments. The main new difficulty in the planar setting is excluding the possibility that
$\dim V_i = 1$ with positive frequency. If $\alpha < 1$, this again follows from the uniform entropy dimension property, hence the remaining case is $\alpha \geq 1$. To deal with it, we will give a lower bound on the entropy of projections of the components  $\eta^{(T^{n'} \om), x, i}$ in all directions. The argument is similar to the proof of \cite[Lemma 3.7]{HR}, but we allow the reducible case (with all $\vphi_i \in \pi\Q$). The rough idea is to use the assumptions that the rotation group $G_\Sigma$ is non-trivial and $\alpha < 2$ to show that with positive probability projections $\pi_W \eta^{(T^{n'}\om), x, i}$  have entropy dimension strictly greater than the obvious lower bound $\alpha - 1$ (uniformly in $W$, for small enough scales). This will exclude saturation on one-dimensional subspaces. For technical reasons, during most of the proof we will assume $r_{\max} \leq 1/2$. In the final step, we will show how to reduce the general case to the one with this additional assumption.

Let us now state more precisely the result on excluding the saturation mentioned above. It will be convenient to introduce the following notation.

\begin{defn} \label{def:satdim}
	For $\mu \in \mP(\R^d), \eps>0, m,n \in \N$, we define the {\em saturation dimension} by
	\[ \satdim(\mu, \eps, m, k) : = \max \left\{  \dim V : V \leq \R^d \text{ and } \bP_{i = k}\left( \mu^{x, i} \text{ is } (V, \eps, m)\text{-saturated} \right) \geq 1-\eps \right\}. \]
\end{defn}

The saturation dimension gives the maximal dimension of subspaces on which most of the rescaled components are saturated. It should not be seen as a usual notion of a dimension of a measure and we introduce it only for a notational convenience. The goal of this section is to prove

\begin{prop}\label{prop:sat-dim 0}
	Assume $\alpha < 2$ and $r_{\mathrm{max}} \leq 1/2$. Then there exists $\eps_0>0$ such that for every $\eps \in (0, \eps_0)$ and $m>m(\eps)$
	\[ \lim \limits_{n \to \infty}\ \PP \left( \left\{ \om \in \Om :  \bP_{0 \leq k \leq n} \left( \satdim(\eta^{(T^{n'}\om)}, \eps, m, k) = 0 \right)  \geq 1 - \eps \right\} \right) = 1. \]
\end{prop}

\subsection{Orthogonal projections onto one-dimensional subspaces}
First, we study orthogonal projections of measures $\eta^{\pom}$ and their components onto one-dimensional subspaces. This will allow us to exclude (with high probability) the saturation of $\eta^{(T^{n'} \om)}$ on one-dimensional subspaces.  As mentioned above, for that part we need to study the case $1 \leq \alpha < 2$.  First, we need to express the dimension of measures $\eta^{(\om)}$ in terms of entropy. It is well known that the Hausdorff dimension coincides with the information dimension for exact dimensional measures \cite{Young}. It therefore follows from Proposition \ref{prop:exact dimension} that for $\PP$-a.e.\ $\om \in \Om$ the information dimension of $\eta^{(\om)}$ exists and equals $\alpha:= \dim(\Sigma)$, i.e.,
\begin{equation}\label{eq:inf dim}
	\lim \limits_{n \to \infty} H_n(\eta^{(\om)}) = \alpha.
\end{equation}

The next point is a suitable for our needs version of a trivial observation: A measure on the plane with information dimension $\alpha \geq 1$ has projections of dimension at least $\alpha - 1$.

\begin{lem}\label{lem:proj dim geq comp}
Assume $\alpha \geq 1$. Then for $\PP$-a.e.\ $\om \in \Om$ the following holds: For every $\eps > 0$ and $m > m(\om, \eps)$
\[
	\inf \limits_{W \in \R\PP^1}\ H_m(\pi_W \eta^{(\om)}) \geq \alpha - 1 - \eps.
\]

\end{lem}
\begin{proof}
	First, note the following fact, valid for every $\mu \in \Pk \left([-R,R]^d\right), \alpha \geq 1, \eps > 0$ and $m > m(\eps, R)$:
	\begin{equation}\label{eq:entropy proj implication}\text{if } \left|H_m(\mu) - \alpha\right| < \frac{\eps}{2}, \text{ then } \inf \limits_{W \in \R\PP^1}\ H_m(\pi_W \mu) \geq \alpha - 1 - \eps.
	\end{equation}
	Indeed, if there exists $W \in \R\PP^1$ such that 
	\[ H_m(\pi_W \mu) < \alpha - 1 - \eps, \]
	then by \eqref{eq:entropy cond W},
	\[
	\begin{split} H_m(\mu) & = \frac{1}{m} H( \pi_W\mu , \Dk_{m}) + \frac{1}{m} H(\mu , \Dk^{W \oplus W^\perp}_{m} | \Dk_m^W) + O\left(\frac{1}{m} \right) \leq \frac{1}{m} H( \pi_W\mu , \Dk_{m}) + 1 + O_R\left(\frac{1}{m} \right)\\
		&  < \alpha - \eps + O_R\left(\frac{1}{m} \right) \leq \alpha - \eps/2
	\end{split}\]
for $m > m(\eps, R)$. Using \eqref{eq:entropy proj implication} with $\mu = \eta^{(\om)}$ we see that the claim of the Lemma follows from \eqref{eq:inf dim}. 
\end{proof}

Now, we show that the entropy $H_m(\pi_V \eta^{\pom})$ cannot be too close to the lower bound $\alpha - 1$ in $ \fm^{(W)}$-almost all directions $V \in \R\PP^1$.

\begin{lem}\label{lem:averge entropy increase}
Assume $1 \leq \alpha < 2$. There exists $p_0>0$ and $\kappa>0$ such that for $\PP$-a.e.\ $\om \in \Om$,
\begin{equation}\label{eq:xi proj dim} \liminf \limits_{n \to \infty}\ \inf \limits_{W \in \R\PP^1}\  \fm^{(W)}\left(\left\{ V \in \R\PP^1 : H_n(\pi_V \eta^{(\om)}) \geq \alpha - 1 + \kappa \right\} \right) \geq p_0.
\end{equation}
\end{lem}
\begin{proof}
Let $\kappa = \frac{2 - \alpha}{16}>0$. Let $p_0>0$ and $\delta_0 > 0$ be such that if a Borel set $A \subset \R\PP^1$ satisfies $ \fm^{(W)}(A) \geq 1 - p_0$ for at least one $W \in \R\PP^1$, then $A$ has diameter larger than $\delta_0$ in the angle metric on $\R\PP^1$. Such $p_0$ and $\delta_0$ exist as $ \fm^{(W)}$ are not concentrated on singletons due to Assumption \ref{assmp:non-deg and rotation}\ref{it:non-real}, and can be taken uniformly among $W$ due to \eqref{eq:xi W}.  We can therefore conclude the following: For every given $\om \in \Om$, if there exists $W \in \R\PP^1$ with $ \fm^{(W)}\left(\left\{ V \in \R\PP^1 : H_n(\pi_V \eta^{(\om)}) \geq \alpha - 1 + \kappa \right\} \right) \leq p_0$, then there are $V_1, V_2 \in \R\PP^1$ with
\begin{equation}\label{eq:V1 V2} \angle(V_1, V_2) > \delta_0\ \text{ and }\  H_n(\pi_{V_i} \eta^{(\om)}) \leq \alpha - 1 + \kappa\ \text{ for } i=1,2.
\end{equation} We will show that \eqref{eq:xi proj dim} holds for every $\om \in \Om$ satisfying \eqref{eq:inf dim}. Fix such $\om$ and for given $n \in \N$ assume
\[\inf \limits_{W \in \R\PP^1}\  \fm^{(W)}\left(\left\{ V \in \R\PP^1 : H_n(\pi_V \eta^{(\om)}) \geq \alpha - 1 + \kappa \right\} \right) < p_0.\]
Then there exist $V_1, V_2$ satisfying \eqref{eq:V1 V2}. If $n$ is large enough, we have for $i = 1,2$, in view of \eqref{eq:entpro} and \eqref{eq:entpro2},
\[
\begin{split} \alpha - \kappa & \leq H_n(\eta^{(\om)}) =  \frac{1}{n} H(\pi_{V_i}\eta^{(\om)} , \Dk_{n}) + \frac{1}{n} H(\eta^{(\om)} , \Dk^{V_i \oplus V_i^\perp}_{n} | \Dk^{V_i}_{n}) + O\left(\frac{1}{n} \right) \\
& \leq \alpha - 1 + \kappa + \frac{1}{n} H(\eta^{(\om)} , \Dk^{V_i \oplus V_i^\perp}_{n} | \Dk^{V_i}_{n}) + O\left(\frac{1}{n} \right),
\end{split}
\]
hence
\[ \frac{1}{n}H(\eta^{(\om)} , \Dk^{V_i^\perp}_{n} | \Dk^{V_i}_{n}) \geq 1 - \left(2\kappa + O\left( \frac{1}{n} \right) \right). \]
By Lemma \ref{lem:sat by cond}, $\eta^{(\om)}$ is $(V^\perp_i, 2\kappa + O(1/n), n)$-saturated for $i=1,2$. As $\angle(V_1, V_2) > \delta_0$, we conclude using Lemma \ref{lem: saturation on sum} that if $n$ is large enough, then $\eta^{(\om)}$ is $\left(\R^2, 8\kappa, n \right)$-saturated, so
\[ H_n(\eta^{(\om)}) \geq 2 - 8\kappa = \alpha + 8\kappa > \alpha. \]
As $\om$ was chosen to satisfy \eqref{eq:inf dim}, this cannot happen provided that $n$ is large enough. This contradiction shows
\[\inf \limits_{W \in \R\PP^1}\  \fm^{(W)}\left(\left\{ V \in \R\PP^1 : H_n(\pi_V \eta^{(\om)}) \geq \alpha - 1 + \kappa \right\} \right) \geq p_0\]
for all $n$ large enough.
\end{proof}

Next, we improve the above entropy bound to \emph{all} directions using ergodicity of $\PP \otimes  \fm^{(W)}$. This is the point where the technical assumption $r_{\mathrm{max}} \leq 1/2$ is introduced.

\begin{lem}\label{lem:proj entropy improved}
Assume  $1 \leq \alpha < 2$ and $r_{\mathrm{max}} \leq 1/2$. There exists $\kappa' > 0$ such that for $\PP$-a.e.\ $\om \in \Om$,
\begin{equation}\label{eq:proj entropy improved} \liminf \limits_{n \to \infty}\ \inf \limits_{W \in \R\PP^1} H_n\left(\pi_W \eta^{(\om)} \right) \geq \alpha - 1 + \kappa'.
\end{equation}
\end{lem}
\begin{proof}
By Lemma \ref{lem: local to global entropy}, for $n \geq n(m,\eps)$,
\begin{equation}\label{eq:proj entropy average} H_n\left(\pi_W \eta^{(\om)}\right) \geq \frac{1}{n} \sum \limits_{k=0}^{n-1} \frac{1}{m} H(\pi_W \eta^{(\om)}, \Dk_{k+m}|\Dk_k) - \eps. \end{equation}
For $k \geq 1$ we have by \eqref{eq:dyn self-sim},
\[ \pi_W \eta^{(\om)} = \sum \limits_{u \in \X^{(\om)}_{k'}} p_u^{\pom} \pi_W\left( f_u^{(\om)} \eta^{(T^{k'}\om)} \right). \]
As $\diam \left( \supp \left( f^{(\om)}_u \eta^{(T^{k'} \om)} \right)  \right) \leq 2^{-k}$, we have
\[ \frac{1}{m}H\left( \pi_W\left( f_u^{(\om)} \eta^{(T^{k'}\om)} \right), \Dk_{k+m}|\Dk_k\right) = \frac{1}{m}H\left( \pi_W\left( f_u^{(\om)} \eta^{(T^{k'}\om)} \right), \Dk_{k+m}\right) + O(1/m). \]
Combining \eqref{eq:proj entropy average} with the last two equalities, we obtain using  concavity of the conditional entropy and the fact that $f^{(\om)}_u$ is a similarity contracting by $\Theta(2^{-k})$ for $u \in \X^{(\om)}_{k'}$ together with \eqref{eq:entropy bounded similarity}:
\begin{equation}\label{eq:proj entropy average bound}\begin{split}
\frac{1}{n}H\left(\pi_W \eta^{(\om)}, \Dk_n\right) & \geq \frac{1}{n} \sum \limits_{k=0}^{n-1} \frac{1}{m} \sum \limits_{u \in \X^{(\om)}_{k'}} p_u^{\pom} H\left( \pi_W\left( f_u^{(\om)} \eta^{(T^{k'}\om)} \right), \Dk_{k+m}|\Dk_k\right) - O\left( \eps \right) \\
&\geq \frac{1}{n} \sum \limits_{k=0}^{n-1} \frac{1}{m} \sum \limits_{u \in \X^{(\om)}_{k'}} p_u^{\pom} H\left( \pi_W\left( f_u^{(\om)} \eta^{(T^{k'}\om)} \right), \Dk_{k+m}\right) - O\left( \eps \right) - O(1/m) \\
& = \frac{1}{n} \sum \limits_{k=0}^{n-1} \frac{1}{m} \sum \limits_{u \in \X^{(\om)}_{k'}} p_u^{\pom} H\left( \pi_{\vphi_{\om_{k'}}^{-1} \cdots \vphi_{\om_1}^{-1} W} \eta^{(T^{k'}\om)} , \Dk_{m}\right) - O\left( \eps \right) - O(1/m) \\
& = \frac{1}{n} \sum \limits_{k=0}^{n-1}  H_m\left( \pi_{\vphi_{\om_{k'}}^{-1} \cdots \vphi_{\om_1}^{-1} W} \eta^{(T^{k'}\om)}\right) - O\left( \eps \right) - O(1/m).
\end{split}\end{equation}
On the other hand, by Lemma \ref{lem:averge entropy increase}, for $m$ large enough,
\begin{equation}\label{eq:symbol direction measure bound} \inf \limits_{W \in \R\PP^1}\ \PP \otimes  \fm^{(W)} \left(\left\{ \left(\om, V\right) \in \Om \times \R\PP^1 : H_m(\pi_V \eta^{(\om)}) \geq \alpha - 1 + \kappa \right\} \right) \geq p_0/2.
\end{equation}
It follows now from Lemmas \ref{lem:V cont} and \ref{lem:om cont} that, if $m$ is large enough, there exists an open set $\mU_m \subset \Om \times \R\PP^1$ (a small enough open neighbourhood of the set from \eqref{eq:symbol direction measure bound}) satisfying
\[\mU_m \subset \left\{ \left(\om, V\right) \in \Om \times \R\PP^1 : H_m(\pi_V \eta^{(\om)}) \geq \alpha - 1 + \kappa/2 \right\}\  \text{ and }\  \inf \limits_{W \in \R\PP^1}\ \PP \otimes  \fm^{(W)} (\mU_m) \geq p_0/2. \]
Recall the skew-product map $F : \Om \times \R\PP^1 \to \Om \times \R\PP^1$ defined as $F(\om, V) = (T \om, \vphi^{-1}_{\om_1}V)$. By Lemma \ref{lem:skew product ergodic}, for $\PP$-a.e.\ $\om \in \Om$,
\begin{equation}\label{eq:F Um first freq bound}  \liminf \limits_{n \to \infty} \inf \limits_{W \in \R\PP^1} \frac{1}{n}\sum \limits_{k=0}^{n-1}\mathds{1}_{\mU_m}(F^k(\om, W)) \geq p_0/2,
\end{equation}
and hence by the claim \eqref{eq:freq half upper bound} from Lemma \ref{lem:k k' comparison} (here we use that $r_{\max}\le 1/2$),
\begin{equation}\label{eq:F Um freq bound} \liminf \limits_{n \to \infty} \inf \limits_{W \in \R\PP^1} \frac{1}{n}\sum \limits_{k=0}^{n-1}\mathds{1}_{\mU_m}(F^{k'(\om)}(\om, W)) \geq p_0/2C.
\end{equation}
On the other hand, by Lemma \ref{lem:proj dim geq comp}, for $m>m(\eps)$ the inequality
\[	\inf \limits_{W \in \R\PP^1}\ H_m(\pi_W \eta^{(\om)}) \geq \alpha - 1 - \eps\]
holds for $\om$ from a subset of $\Om$ of $\PP$-measure at least $1-\eps$.  Therefore, by the ergodic theorem, for $\PP$-a.e.\ $\om \in \Om$,
\[ \liminf \limits_{n \to \infty} \frac{1}{n} \left\{ 0 \leq k < n:  \inf \limits_{W \in \R\PP^1}\ H_m(\pi_W \eta^{(T^k\om)}) \geq \alpha - 1 - \eps \right\} \geq 1 -\eps, \]
hence by the claim \eqref{eq:freq general upper bound} in Lemma \ref{lem:k k' comparison},
\begin{equation}\label{eq:proj entropy freq bound} \liminf \limits_{n \to \infty} \frac{1}{n} \left\{ 0 \leq k < n: \inf \limits_{W \in \R\PP^1}\ H_m(\pi_W \eta^{(T^{k'} \om)}) \geq \alpha - 1 - \eps \right\} \geq 1-C\eps.
\end{equation}
Combining \eqref{eq:F Um freq bound} and \eqref{eq:proj entropy freq bound} with \eqref{eq:proj entropy average bound} we conclude that for $\PP$-a.e.\ $\om \in \Om$, we obtain
\[
\begin{split}
\liminf \limits_{n \to \infty}\ & \inf \limits_{W \in \R\PP^1} H_n (\pi_W \eta^{(\om)}) \\
& \geq (\alpha - 1 + \kappa/2)\frac{p_0}{2C} + (\alpha - 1 - \eps)(1 - C\eps - \frac{p_0}{2C}) - O(\eps) - O(1/m).
\end{split}
\]
If $\eps>0$ was chosen small enough and $m$ large enough, the above inequality gives \eqref{eq:proj entropy improved} for some $\kappa' > 0$.
\end{proof}

 Next, we extend the previous statement to components and replace $\om$ by $T^{n'}(\om)$. We will use this result later as one of the main elements of the proof of Proposition \ref{prop:sat-dim 0}.

\begin{lem}\label{lem:proj entropy improved comp}
Assume $1 \leq \alpha < 2$ and $r_{\mathrm{max}} \leq 1/2$. There exist $p_1 > 0$ and $\kappa''>0$ such that for $\PP$-a.e.\ $\om \in \Om$, for every $\eps > 0$ and $m > m(\eps)$ the following holds\footnote{Here and throughout the paper, $\pi_W \eta^{(\om), x, i}:=\pi_W \left( \left(\eta^{(\om)}\right)^{x, i} \right)$ denotes the projection of the rescaled component $\eta^{(\om), x, i}$  onto $W$ (and \emph{not} the component of the projection).}
\[ \liminf \limits_{n \to \infty}\ \frac{1}{n} \# \left\{ 0 \leq k < n : \inf \limits_{W \in \R\PP^1} \bP_{i=k} \left( H_m(\pi_W \eta^{(T^{n'} \om), x, i}) \geq \alpha - 1 + \kappa'' \right) > p_1  \right\} \geq 1 - \eps \]
\end{lem}

\begin{proof}
First note that  rescaling   and projection commute up to a translation, i.e., for given $x \in \R^2$ and $W \in \R\PP^1$, the compositions $ \pi_W\circ T_{D_k(x)}$ and $T_{D_k(\pi_W(x))} \circ \pi_W$ differ by a translation. Recalling Definition \ref{def:raw}, we can use this fact together with \eqref{eq:intersecting partitions}, \eqref{eq:entropy bounded similarity}, and \eqref{eq:entpro} to obtain
\begin{equation}\label{eq:projected component entropy}
\begin{split}
\bE_{i = k} H_m(\pi_W\eta^{({ T^{n'}}\om), x, i}) & = \bE_{i = k} \frac{1}{m}H(\eta^{({ T^{n'}}\om), x, i}, \pi_W^{-1}\Dk_m) \\
& =\bE_{i = k} \frac{1}{m}H(\eta^{({ T^{n'}}\om)}_{x, i}, T_{D_k(x)}^{-1}\pi_W^{-1}\Dk_m) \\
& = \bE_{i = k} \frac{1}{m}H(\eta^{({ T^{n'}}\om)}_{x, i}, \pi_W^{-1}T_{D_k(\pi_W x)}^{-1}\Dk_m) + O(1/m) \\
& = \bE_{i=k} \frac{1}{m}H(\pi_W \eta^{({ T^{n'}}\om)}_{x,i}, \Dk_{m+k}) + O(1/m) \\
& = \frac{1}{m}H(\pi_W \eta^{({ T^{n'}}\om)}, \Dk_{m+k} | \Dk_{k}) + O(1/m). \\
\end{split}
\end{equation}
Applying \eqref{eq:dyn self-sim} to $\eta^{(T^{n'} \om)}$ with the number of iterates equal to $k'(T^{n'}\om)$ gives
\[\label{eq:self-sim n'} \eta^{(T^{n'} \om)} = \sum \limits_{u \in \X_{k'(T^{n'}\om)}^{(T^{n'}\om)}} p_u^{(T^{n'}\om)} \cdot f_u^{(T^{n'}\om)}\eta^{(T^{k'(T^{n'}\om) + n'}\om)}.\]
Each $f_u^{(T^{n'}\om)}$ above is a similarity contracting by $\Theta(2^{-k})$ and $\diam \left( \supp \left(f_u^{(T^{n'}\om)}\eta^{(T^{k'(T^{n'}\om) + n'}\om)} \right)  \right) \leq 2^{-k}$. Combining this with  \eqref{eq:projected component entropy}, concavity of the conditional entropy and  \eqref{eq:entropy bounded similarity} yields a calculation similar to \eqref{eq:proj entropy average bound}:
\[
\begin{split}
\bE_{i = k} H_m(\pi_W\eta^{( T^{n'}\om), x, i}) & \geq \sum \limits_{u \in \X_{k'}^{(\om)}} p_u^{(T^{n'} \om)} \frac{1}{m} H\left(\pi_W f_u^{(T^{n'}\om)}\eta^{(T^{k'(T^{n'}\om) + n'}\om)} , \Dk_{m+k}| \Dk_{k}\right) - O(1/m) \\
& = \sum \limits_{u \in \X_{k'}^{(\om)}} p_u^{(T^{n'} \om)} \frac{1}{m} H\left(\pi_W f_u^{(T^{n'}\om)}\eta^{(T^{k'(T^{n'}\om) + n'}\om)}, \Dk_{m+k}\right) - O(1/m) \\
& = \sum \limits_{u \in \X_{k'}^{(\om)}} p_u^{(T^{n'} \om)} \frac{1}{m} H\left( \pi_{ \vphi^{-1}_{\om_{n' + k'(T^{n'}\om)}} \cdots \vphi^{-1}_{\om_{n'+1}} W} \left( \eta^{(T^{k'(T^{n'}\om) + n'}\om)} \right), \Dk_{m}\right) - O(1/m)\\
& \geq \inf \limits_{V \in \R\PP^1} \frac{1}{m}H\left(\pi_V \eta^{(T^{k'(T^{n'}\om) + n'}\om)}, \Dk_m \right) - O(1/m).
\end{split}\]
Applying \eqref{eq:cocycle} and Lemma \ref{lem:shifted entropy comp} to the inequality above gives
\begin{equation}\label{eq:average comp entropy}
\bE_{i = k} H_m(\pi_W\eta^{(T^{n'}\om), x, i})  \geq \inf \limits_{V \in \R\PP^1} \frac{1}{m}H\left(\pi_V \eta^{(T^{(n+k)'}\om)}, \Dk_m \right) - O(1/m).
\end{equation}

By Lemma \ref{lem:proj entropy improved}, for every $\eps > 0$ and $m > m(\eps)$,
{ \[ \PP \left( \left\{ \om \in \Om : \inf \limits_{W \in \R\PP^1} \frac{1}{m}H\left(\pi_W \eta^{(\om)}, \Dk_m\right) \geq \alpha - 1 + \kappa'/2 \right\} \right) \geq 1 - \eps.  \] }
Therefore, by the ergodic theorem, for $\PP$-almost every $\om \in \Om$,
\[
{ \lim \limits_{n \to \infty} }  \frac{1}{n}\#\left\{0 \leq k < n : \inf \limits_{W \in \R\PP^1} \frac{1}{m}H\left(\pi_W \eta^{(T^{ k } \om)}, \Dk_m\right) \geq \alpha - 1 + \kappa'/2 \right\} \geq 1 - \eps,
\]
and hence also, 
\[
{ \begin{split} \liminf \limits_{n \to \infty} & \frac{1}{n}\#\left\{0 \leq k < n : \inf \limits_{W \in \R\PP^1} \frac{1}{m}H\left(\pi_W \eta^{(T^{(n+k)' } \om)}, \Dk_m\right) \geq \alpha - 1 + \kappa'/2 \right\} \\
	& = \liminf \limits_{n \to \infty}  \frac{1}{n}\#\left\{n \leq k < 2n : \inf \limits_{W \in \R\PP^1} \frac{1}{m}H\left(\pi_W \eta^{(T^{ k' } \om)}, \Dk_m\right) \geq \alpha - 1 + \kappa'/2 \right\}  \\
	& \geq 2  \liminf \limits_{n \to \infty}  \frac{1}{2n}\#\left\{0 \leq k < 2n : \inf \limits_{W \in \R\PP^1} \frac{1}{m}H\left(\pi_W \eta^{(T^{ k' } \om)}, \Dk_m\right) \geq \alpha - 1 + \kappa'/2 \right\} - 1\\
	& \geq 2(1 - C\eps) -1 \\
	& = 1 - 2C\eps,
\end{split}
}
\]
where in the last inequality we have applied  \eqref{eq:freq general upper bound} from Lemma \ref{lem:k k' comparison}.
Combining this with \eqref{eq:average comp entropy} gives for $m > m(\eps)$ and $\PP$-a.e.\ $\om \in \Om$:
{
\begin{equation}\label{eq:average comp entropy freq} \liminf \limits_{n \to \infty}\ \frac{1}{n} \# \left\{ 0 \leq k < n :  \inf \limits_{W \in \R\PP^1} \bE_{i=k} H_m(\pi_W \eta^{( T^{n'}\om ), x, i})  \geq \alpha - 1 + \kappa'/4  \right\} \geq 1 -  2C\eps .
\end{equation}
}
As $H_m(\pi_W \eta^{({T^{n'}\om }),x,i}) \leq 1 + O(1/m)$, we obtain
\[
\begin{split}\bE_{i=k} H_m(\pi_W \eta^{({ T^{n'}\om }), x, i})  \leq\ &\bP_{i = k} \left( H_m(\pi_W \eta^{({ T^{n'}\om }), x, i}) \leq \alpha - 1 + \kappa'/8 \right) \left( \alpha - 1 + \kappa'/8\right)  \\
& +\bP_{i = k} \left( H_m(\pi_W \eta^{({ T^{n'}\om }), x, i}) \geq \alpha - 1 + \kappa'/8 \right) \left( 1 + O(1/m) \right),
\end{split}
\]
thus
\[ \bP_{i = k} \left( H_m(\pi_W \eta^{({ T^{n'}\om }), x, i}) \geq \alpha - 1 + \kappa'/8 \right) \geq \frac{\bE_{i=k} H_m(\pi_W \eta^{({T^{n'}\om }), x, i}) - \left( \alpha - 1 + \kappa'/8 \right)}{1 + O(1/m)}. \]
Applying the above inequality to \eqref{eq:average comp entropy freq} gives for $m > m(\eps)$:

\[ \liminf \limits_{n \to \infty}\ \frac{1}{n} \# \left\{ 0 \leq k < n : \inf \limits_{W \in \R\PP^1} \bP_{i=k} \left( H_m(\pi_W \eta^{( T^{n'}\om ), x, i}) \geq \alpha - 1 + \kappa'/8 \right) > \kappa'/16 \right\} \geq 1 - 2C\eps .\] 
This finishes the proof.
\end{proof}

\subsection{Uniform entropy dimension}
We shall study now the uniform entropy dimension properties of the random measures $\eta^{\pom}$. As we are going to apply the inverse theorem to the convolution \eqref{eq:convolution decomposition}, in the scale $2^{-n}$ we need to study $\eta^{(T^{n'(\om)}\om)}$ rather than $\eta^{\pom}$ itself. The main goal of this subsection is to prove the following:

\begin{prop}\label{prop:unif ent dim}
Let $\Sigma$ be a model satisfying Assumption \ref{assmp:non-deg and rotation}. Then, for every $\eps > 0$ and $m > m(\eps)$,
\begin{equation}\label{eq:unif ent dim} \lim \limits_{n \to \infty} \PP \left( \left\{ \om \in \Om : \bP _{0 \leq i \leq n} \left( \left| H_{m}\left( \eta^{(T^{n'}\om), x, i} \right) - \alpha \right| < \eps \right) > 1 -\eps \right\} \right) = 1.
\end{equation}
\end{prop}

\begin{rem}
	A probability measure $\mu$ on $\R^d$ has \textit{uniform entropy dimension} $\alpha$ if for every $\eps > 0$ and $m > m(\eps)$ it satisfies $\liminf \limits_{n \to \infty} \bP _{0 \leq i \leq n} \left( \left| H_{m}\left(\mu^{x, i} \right) - \alpha \right| < \eps \right) > 1 -\eps$ (see \cite[Def. 5.1]{H}). In \cite[Theorem 4.7]{SSS} it was proved that for models on the real line, almost every measure $\eta^{\pom}$ is of uniform entropy dimension $\alpha$. The proof of Proposition \ref{prop:unif ent dim} presented below follows closely the approach of \cite{SSS}, with some additional geometrical ingredients required in the planar case and an adjustment allowing the change of $\eta^{\pom}$ to $\eta^{(T^{n'}\om)}$. Omitting the latter modification, one can prove that $\eta^{\pom}$ is of uniform entropy dimension $\alpha$ for $\PP$-a.e.\ $\om \in \Om$ in our setting as well. We do not include the details, as we formally do not use this property.
\end{rem}

First, we need a version of \eqref{eq:inf dim} with $\om$ replaced by $T^{n'(\om)} \om$.
	\begin{lem}\label{lem:n' entropy dim}
		For every $q \in \N$
		\[  H_{qn}\left( \eta^{(T^{n'}\om)} \right)\overset{\PP}{\longrightarrow} \alpha, \]
		i.e. for every $\delta > 0$
		\begin{equation}\label{eq:n' entropy dim}\lim \limits_{n \to \infty} \PP \left( \left\{ \om \in \Om : \left| H_{qn}\left( \eta^{(T^{n'}\om)} \right) - \alpha \right| < \delta \right\} \right) = 1.
		\end{equation}
	\end{lem}
	
	\begin{proof}
		The almost sure convergence in \eqref{eq:inf dim} implies convergence in probability, hence
		\begin{equation}\label{eq: entropy dim}
			\lim \limits_{n \to \infty} \PP \left( \left\{ \om \in \Om : \left| H_{n}\left( \eta^{(\om)} \right) - \alpha \right| < \delta \right\} \right) = 1 \text{ for every } \delta>0.
		\end{equation}
		By the ergodic theorem applied to the function $\Om \ni \om \mapsto \log r_{\om_1}$ we have for $\PP$-a.e.\ $\om \in \Om$: 
		\[ \lim \limits_{n \to \infty} \frac{n'(\om)}{n} = L := \left( - \int \limits_{\Om} \log r_{\om'_1}d\PP(\om') \right)^{-1}. \]
		Fix $\eps > 0$. By the above convergence,
		\begin{equation}\label{eq:lyap convergence}
			\lim \limits_{n \to \infty} \PP \left( \left\{ \om \in \Om : |n'(\om) - Ln| \leq \eps n \right\} \right) = 1.
		\end{equation}
		By Lemma \ref{lem:shifted entropy comp}, if $n \in \N, \om \in \Om$ are such that $|n'(\om) - Ln| \leq \eps n$, then
		\[ \left| H_{qn}\left( \eta^{(T^{n'}\om)} \right) - H_{qn}\left( \eta^{(T^{ \lceil Ln \rceil}\om)} \right) \right| \leq O\left( \frac{\eps n + 1}{qn} \right) = O(\eps / q) \]
		for $n$ large enough (depending on $\eps$). Therefore for fixed $q \in \N,\ \delta > 0$, and $\eps$ small enough to guarantee $O(\eps / q) \leq \delta / 2$, we have for each $n \in \N$ large enough:
		\[
		\begin{split} \PP & \left( \left\{ \om \in \Om : \left| H_{qn}\left( \eta^{(T^{n'}\om)} \right) - \alpha \right| < \delta \right\} \right) \\
			&\geq \PP \left( \left\{ \om \in \Om : \left| H_{qn}\left( \eta^{(T^{\lceil Ln \rceil}\om)} \right) - \alpha \right| < \delta/2 \right\} \right) - \PP \left( \left\{ \om \in \Om : |n'(\om) - Ln| > \eps n \right\} \right)\\
			& =  \PP \left( \left\{ \om \in \Om : \left| H_{qn}\left( \eta^{(\om)} \right) - \alpha \right| < \delta/2 \right\} \right) - \PP \left( \left\{ \om \in \Om : |n'(\om) - Ln| > \eps n \right\} \right),
		\end{split}
		\]
		where in the last step we have used the $T$-invariance of $\PP$. By \eqref{eq: entropy dim} and \eqref{eq:lyap convergence}, this implies \eqref{eq:n' entropy dim}.
	\end{proof}

Now we turn to the proof of Propostion \ref{prop:unif ent dim}. One difference with \cite{SSS} is that instead of the Frostman condition on balls, we need measures $\eta^{(\om)}$ to satisfy the Frostman condition on tubes, in the sense of the definition below.

\begin{defn}
	We say that a measure $\mu$ on $\R^2$ is \textit{$(C,\rho)$-Frostman on tubes} if $\mu\left( B(W + x, r) \right) \leq C r^\rho$ for every $W \in \R\PP^1,\ x \in \R^2$ and $r>0$, where $B(W + x, r)$ denotes the $r$-neighbourhood of the line $W+x$ (the $r$-tube around $W+x$).
\end{defn} 

The reason for technical complications in proving the uniform entropy dimension property, as in  \cite[Section 4.4]{SSS} and \cite[Section 5.1]{H}, can be heuristically explained as follows.
The uniform entropy dimension (or its analog in our paper) requires tight bounds for the entropy of ``most'' component measures of $\eta^{\pom}$ at ``most scales,'' in the sense of being close to $\alpha$.
These bounds are deduced from the representation of the measure as a convex linear combination of its many small copies (or rather, its ``translated versions''). However, this is literally true only for the
measure itself. For its components the problem that arises is that we condition the measure on a small dyadic cell and thus we have to deal with ``partial'' little copies whose support intersect the boundary of
the cell. This requires a 
rather tedious treatment: First showing that the contribution of the small copies intersecting the boundaries is small, in view of the Frostman on tubes property, and then working with ``truncated'' components, and finally
verifying that
the resulting error is negligible in the limit. An additional complication here, as in \cite{SSS}, is that our measures are random, with a possibility of degeneracy,
hence the Frostman on tubes property can only hold almost surely and
not uniformly.

\medskip

The following lemma is a $2$-dimensional version of \cite[Lemma 4.9]{SSS}. Recall that we assume $\Sigma$ to be a model verifying Assumption \ref{assmp:non-deg and rotation}.

\begin{lem}\label{lem:frostman on tubes}
	There exists $\rho = \rho(\Sigma) > 0$ and a set $\Om' \subset \Om$ of full $\PP$-measure, such that for any $\delta>0$ there exists a constant $C_\delta > 0$ satisfying
	\[ \liminf \limits_{n \to \infty} \frac{1}{n} \# \left\{ 0 \leq k < n : \eta^{(T^k \om)} \text{ is } (C_\delta, \rho) \text{-Frostman on tubes} \right\} \geq 1 - \delta \]
	for every $\om \in \Om'$.
\end{lem}
\begin{proof}
	By the ergodic theorem, it is enough to show that there exists $\rho > 0$ such that for every $\delta>0$ there exists $C_\delta>0$ with
	\[ \PP\left( \eta^{(\om) } \text{ is } (C_\delta, \rho) \text{-Frostman on tubes} \right) \geq 1 - \delta.\]
	Therefore, it is enough to find $\rho > 0$ and $C_0 > 0$ such that for $\PP$-a.e.\ $\om \in \Om$ there exists $r_0 = r_0 (\om) > 0$ satisfying
	\begin{equation}\label{eq:frostman on tubes}
		\eta^{(\om)}\left( B(W + x, r) \right) \leq C_0 r^\rho
	\end{equation}
	for all $W \in \R\PP^1,x \in \R^2, r \in (0, r_0)$. We will prove this inequality first for a fixed $W \in \R\PP^1$ and then explain how to extend it to the uniform statement above.
	
	{By Assumption \ref{assmp:non-deg and rotation}\ref{it:non-degenerate}, there exist $i_0 \in I$ such that $\PP(\om_1 = i_0)$ and $j',j'' \in \{1, \ldots, k_{i_0}\}$ with $t^{(i_0)}_{j'} \neq t^{(i_0)}_{j''}$.  Set $\theta = t^{(i_0)}_{j'} - t^{(i_0)}_{j''} \neq 0$.} We will use the following elementary geometrical observation. For every $\eps < |\theta| / 4$, there exists a closed arc $J_\eps \subset \R\PP^1$ with the following property:
	\begin{equation}\label{eq:tubes observation}
		\begin{split}
			\text{if } & V \notin J_\eps \text{ and } r < \eps, \text{ then for every } x,y \in \R^2, \text{ the tube } B(V+x, r) \\ & \text{can intersect at most one of the balls } B(y, \eps), B(y + \theta, \eps).
		\end{split}
	\end{equation}
	Moreover, $|J_\eps| \to 0$ as $\eps \to 0$. For the rest of the proof we fix $\eps \in (0, |\theta| / 4)$ small enough to have
	\begin{equation}\label{eq: uniform J measure}
		\sup \limits_{W \in \R\PP^1}  \fm^{(W)}(J_\eps) < 1,
	\end{equation}
	where $ \fm^{(W)}$ is given by \eqref{eq:xi W}. This is possible, as $ \fm^{(W)}$ is either the Lebesgue or a measure supported on a finite set, but not on a single atom (due to Assumption \ref{assmp:non-deg and rotation}.\ref{it:non-real}), satisfying \eqref{eq:xi W}. Let us denote for simplicity $J = J_\eps$. Fix $\ell \in \N$ large enough to have
	\[ \diam(B^{(\om)}_u) = \diam(f_u^{(\om)}(B(0,R))) < 2\eps \text{ for all } \om \in \Om, u \in \X^{(\om)}_\ell \]
	(it is enough to have $Rr_{\max}^\ell < \eps$). We claim that for $\PP$-a.e.\ $\om \in \Om$, for every $W \in \R\PP^1$ there exist infinitely many $n \in \N$ (depending on $W$) such that $\om_n = i_0$ and $\vphi_{\om_1}^{-1} \cdots \vphi_{\om_{n-1}}^{-1} W \notin J$. To prove that, let  $F : \Om \times \R\PP^1 \to \Om \times \R\PP^1,\ F(\om, V) = (T \om, \vphi^{-1}_{\om_1}V)$ be the skew-product corresponding to the model. Define
	\[ A = \left\{ (\om, V) \in \Om \times \R\PP^1 : \om_1 = i_0, V \notin J \right\}. \]
	As $A$ is open in $\Om \times \R\PP^1$ (we endow $\Om$ with the product topology), we conclude by Lemma \ref{lem:skew product ergodic} that for $\PP$-a.e.\ $\om \in \Om$,
	\begin{equation}\label{eq:good om frequency}
		\begin{split}
			\liminf \limits_{n \to \infty}\ \inf \limits_{W \in \R\PP^1}\ \frac{1}{n} \sum \limits_{k=0}^{n-1} \mathds{1}_A(F^n(\om, W)) & \geq \inf \limits_{W \in \R\PP^1}\ \PP \otimes  \fm^{(W)}(A) \\
			& = \PP(\om_1 = i_0) \inf \limits_{W \in \R\PP^1}\   \fm^{(W)}(\R\PP^1 \setminus J) =: c > 0.
		\end{split}
	\end{equation}
	Positivity of $c$ follows from  Assumption \ref{assmp:non-deg and rotation} and \eqref{eq: uniform J measure}. Fix now $W \in \R\PP^1$. As
	\[ \frac{1}{n} \sum \limits_{k=0}^{n-1} \mathds{1}_A(F^n(\om, W)) = \frac{1}{n} \# \{ 1 \leq k \leq n : \om_k = i_0, \vphi_{\om_1}^{-1} \cdots \vphi_{\om_{k-1}}^{-1} W \notin J \},\]
	we see that indeed, for every $\om$ satisfying \eqref{eq:good om frequency} there exist infinitely many $n \in \N$ such that $\om_n = i_0$ and $\vphi_{\om_1}^{-1} \cdots \vphi_{\om_{n-1}}^{-1} W \notin J$. For such $\om$, set $n_0 = 0$ and for $j \in \N$ define inductively $n_j = n_j(\om, W)$ by
	\[ n_{j+1} = \min \{ n\geq n_j + \ell: \om_n = i_0 \text{ and } \vphi_{\om_1}^{-1} \cdots \vphi_{\om_{n-1}}^{-1} W \notin J\} = \min \{ n\geq n_j + \ell: F^{n-1}(\om, W) \in A \}. \]
	As the event $F^{n-1}(\om, W) \in A$ occurs for at most $j\ell$ iterates $n \in \{1, \ldots, n_j\}$, we see that
	\[  \frac{\ell j}{n_j} \geq \frac{1}{n_j} \sum \limits_{k=0}^{n_j-1} \mathds{1}_A(F^n(\om, W)),\]
	hence by \eqref{eq:good om frequency} there exists $K(\om) \in \N$ such that for $j \geq K(\om)$ we have
	\begin{equation}\label{eq:n_j linear growth} \frac{n_j}{j} \leq C' := \frac{2\ell}{c} < \infty.
	\end{equation}
	Note that $K$ depends only on $\om$, even though the sequence $n_j$ depends both on $\om$ and $W$. As $n_j = \sum \limits_{k=0}^{j-1} \left( n_{k+1} - n_k \right) \leq C'j$ for such $j$, we obtain
	\begin{equation}\label{eq:large gap freq} \# \left\{ 0 \leq k < j : n_{k+1} - n_k \geq 2C' \right\} \leq j/2\ \text{ for } j \geq K(\om). 
	\end{equation}
	For $j \in \N$ write $\om^j := T^{n_j - 1} \om$. For $\om \in \Om, V \in \R\PP^1$ and $r>0$ define
	\[ \phi^{(\om)}(V,r) := \sup \limits_{x \in \R^2} \eta^{(\om)}\left( B(V+x,r) \right). \]
	Applying \eqref{eq:dyn self-sim} and recalling that $f_u^{(\om)}$ contracts by at most $r_{\min}^{k}$ for $u \in \X^{(\om)}_k$, it is easy to see that
	\[
	\phi^{(\om)}(V, r) \leq \phi^{(T^k \om)} \left(\vphi_{\om_1}^{-1} \cdots \vphi_{\om_k}^{-1} V, \frac{r}{r_{\min}^k}\right).
	\]
	In particular applying that to $W$ with $k = n_1$ gives
	\begin{equation}\label{eq:tube measure start}
		\phi^{(\om)}(W, r) \leq \phi^{(\om^1)}\left( \vphi_{\om_1}^{-1} \cdots \vphi_{\om_{n_1 - 1}}^{-1} W, \frac{r}{r_{\min}^{n_1}}  \right).
	\end{equation}
	Similarly, for $j \geq 1$ we have by \eqref{eq:dyn self-sim} the following decomposition:
	\begin{equation}\label{eq:tube measure decomp}
		\begin{split}
			\eta^{(\om^j)} & \left(B\left(\vphi_{\om_1}^{-1} \cdots  \vphi_{\om_{n_j - 1}}^{-1} W + x, r\right)\right) \\
			& = \sum \limits_{u \in \X^{(\om_j)}_{n_{j+1} - n_j}} p_u^{(\om^j)}f_u^{(\om^j)}\eta^{(\om^{j+1})} \left(B\left(\vphi_{\om_1}^{-1} \cdots  \vphi_{\om_{n_j - 1}}^{-1} W + x, r\right)\right).
		\end{split}
	\end{equation}
	As $\om^j_1 = i_0 $ and $n_{j+1} - n_{j} \geq \ell$, there are $v, \tilde{v} \in \X_{n_{j+1} - n_j}$ such that $v_1 = j, \tilde{v}_1 = j'$ and all other entries are equal. Therefore $f^{(\om_j)}_v (0) - f^{(\om_j)}_{\tilde{v}} (0) = t^{i_0}_j - t^{i_0}_{j'} = \theta$, hence $B^{(\om^j)}_v = B^{(\om^j)}_{\tilde{v}} + \theta$ . As $|v| = |\tilde{v}| \geq \ell$, these are balls of radius smaller than $\eps$. As also $\vphi_{\om_1}^{-1} \cdots \vphi_{\om_{n_j - 1}}^{-1}W \notin J$, we have by \eqref{eq:tubes observation} that for $r<\eps$, the tube $B\left(\vphi_{\om_1}^{-1} \cdots  \vphi_{\om_{n_j - 1}}^{-1} W + x, r\right)$ is disjoint with at least one of the balls $B^{(\om^j)}_v, B^{(\om^j)}_{\tilde{v}}$. As for each $u \in \X^{(\om_j)}_{n_{j+1} - n_j}$ the measure $f_u^{(\om_j)}\eta^{(\om^{j+1})}$ is supported in the ball $B_u^{(\om^j)}$, we conclude that at least one of the terms in the sum \eqref{eq:tube measure decomp} has to be zero. Assume without loss of generality that it is the term corresponding to $v$. Therefore for $r \in (0,\eps)$ holds
	\begin{equation}\label{eq:tube measure step} 
		\begin{split}
			\eta^{(\om^j)} & \left(B\left(\vphi_{\om_1}^{-1} \cdots  \vphi_{\om_{n_j - 1}}^{-1} W + x, r\right)\right) \\
			& = 
			\sum \limits_{u \neq v} p_u^{(\om^j)}f_u^{(\om^j)}\eta^{(\om^{j+1})} \left(B\left(\vphi_{\om_1}^{-1} \cdots  \vphi_{\om_{n_j - 1}}^{-1} W + x, r\right)\right) \\
			& \leq \sum \limits_{u \neq v} p_u^{(\om^j)} \eta^{(\om^{j+1})} \left(B\left(\vphi_{\om_1}^{-1} \cdots  \vphi_{\om_{n_{j+1} - 1}}^{-1} W + \left(f_u^{(\om^j)}\right)^{-1}(x), \frac{r}{r_{\min}^{n_{j+1} - n_j}}\right)\right) \\
			& \leq \left( 1 - p_{\min}^{n_{j+1} - n_j} \right) \phi^{(\om^{j+1})}\left(\vphi_{\om_1}^{-1} \cdots  \vphi_{\om_{n_{j+1} - 1}}^{-1} W,  \frac{r}{r_{\min}^{n_{j+1} - n_j}} \right),
		\end{split}
	\end{equation}
	where $p_{\min} := \min \left\{ p^{(i)}_j : i \in I, j \in \{1, \ldots, k_i \} \right\} > 0$. Starting with \eqref{eq:tube measure start} and iterating \eqref{eq:tube measure step}, we obtain using \eqref{eq:large gap freq}:
	\begin{equation}\label{eq:tube measure together} \phi^{(\om)}(W, \eps r_{\min}^{n_j}) \leq \prod \limits_{k=1}^{j-1}(1 - p_{\min}^{n_{k+1} - n_k}) \leq (1 - p_{\min}^{2C'})^{j/2}\ \text{ for } j \geq K(\om).
	\end{equation}
	Set now $r_0 := \eps r_{\min}^{n_{K(\om)}}$ and note that for every $r \in (0, r_0)$ there exists $j \geq K(\om)$ such that
	\[\eps r_{\min}^{n_{j+1}} < r \leq \eps r_{\min}^{n_j}.\]
	Therefore, by \eqref{eq:tube measure together} and \eqref{eq:n_j linear growth}, there exist constants $C_0, \rho$ such that for $r < r_0$,
	\[ \eta^{(\om)}\left(B\left(W + x , r \right)\right) \leq \phi^{(\om)}(W, \eps r_{\min}^{n_j}) \leq (1 - p_{\min}^{2C'})^{j/2} \leq (1 - p_{\min}^{2C'})^{\frac{n_{j+1}}{2C'} - 1/2} \leq C_0 r^{\rho}. \]
	Note that $C_0$ and $\rho$ depend only on the model $\Sigma$. As $r_0$ depends only on $\om$ and not on $W$, the bound is uniform $W \in \R\PP^1$, hence \eqref{eq:frostman on tubes} holds.
\end{proof}

Recall the definition of $k'(\om)$ from \eqref{eq:prime def}. For fixed $\om \in \Om, n>k$, and $D \in \Dk_k$, define (following \cite{SSS}) the $n$-truncated $k$-component of $\eta^{(\om)}$ by
\[ \eta^{(\om)}_{n, [D]} := \frac{1}{Z_{D,n}} \sum \limits_{u \in \X^{(\om)}_{n, [D]}} p^{(\om)}_u f^{(\om)}_u \eta^{(T^n \om)}, \]
where
\[ \X^{(\om)}_{n, [D]}  := \left\{ u \in \X^{(\om)}_n : B_u^{(\om)} \subset D \right\} \]
and $Z_{D,n}$ is a normalizing constant making $\eta^{(\om)}_{n, [D]}$ a probability measure. We also write
\[\eta^{(\om)}_{n, [x,k]} := \eta^{(\om)}_{n, D_k(x)}.  \]

The second step needed for the proof of Proposition \ref{prop:unif ent dim} is an analog \cite[Proposition 4.11]{SSS} in the planar setting. Together with Lemma \ref{lem:frostman on tubes}, these are the main technical ingredients of the proof of the uniform entropy property from \cite{SSS}.

\begin{lem}\label{lem:TV distance}
	For any $\eps, C, \rho > 0$ there exists $N_\eps = N(\eps, C, \rho)$ such that the following holds: If $\om \in \Om, k \in \N$ are such that $\eta^{(T^{k'}\om)}$ is $(C, \rho)$-Frostman on tubes, and we take $n = k'(\om) + N_\eps$, then
	\[ \bP_{i=k}\left(d_{TV}\left( \eta^{(\om)}_{x,i}, \eta^{(\om)}_{n,[x,i]}  \right) < \eps \right) > 1 - \eps, \]
	where $d_{TV}$ denotes the total variation distance in the space of finite Borel measure on $\R^2$.
\end{lem}

\begin{proof}
	The proof follows closely the proof of \cite[Proposition 4.11]{SSS}. We will therefore present with more details the initial portion of the proof where the Frostman on tubes property is used and only sketch the remainder of the argument.
	
	Given $\eps > 0$, choose $\delta \in (0,1/2)$ such that any tube $B(W+x,r)$ with $r \leq 2R\delta$ satisfies $\eta^{(T^{k'}\om)}(B(W+x,r)) < \frac{\eps^2}{128}$ (by the Frostman on tubes assumption) and pick $N_\eps$ such that $r_{\max}^{N_\eps} < \delta/2$.
	
	Fix $u \in \X^{(\om)}_{k'}$ and consider $v \in \X^{(T^{k'}\om)}_{N_\eps}$ such that $B^{(\om)}_{uv}$ has non-empty intersection with some $D \in \Dk_k$, but is not entirely contained in it. Then for $n = k'(\om) + N_\eps$ the measure $f^{(\om)}_{uv} \eta^{(T^n \om)}$ is supported inside the $2R\delta r_{\om_1}\cdots r_{\om_{k'}}$-neighbourhood $T_v$ of the boundary $\partial D$ of the cube $D$. As $\partial D$ is a subset of a union of $4$ lines, we see that $\left(f^{(\om)}_u\right)^{-1}(T_v)$ is a subset of a union of $4$ tubes of the form $B(W, x+r)$ with $r \leq 2R\delta$. Therefore, by the choice of $\delta$, $f^{(\om)}_u\eta^{(T^{k'}\om)} (T_v) =  \eta^{(T^{k'}\om)}\left(\left(f^{(\om)}_u\right)^{-1}(T_v)\right) < \frac{\eps^2}{32}$ for each such neighbourhood $T_v$. On the other hand, as $2R\delta r_{\om_1}\cdots r_{\om_{k'}} \leq \delta 2^{-k}$ by \eqref{eq:prime def} and as $f^{(\om)}_u\eta^{(T^{k'}\om)}$ is supported in a ball of diameter at most $2^{-k}$, we see that $f^{(\om)}_u\eta^{(T^{k'}\om)}$ can give a positive mass to at most $4$ such tubular neighbourhoods $T_v$. Using this together with the relation  $f^{(\om)}_u\eta^{(T^{k'}\om)} = \sum \limits_{v \in \X^{(T^{k'}\om)}_{N_\eps}} p^{(T^{k'}\om)}_v f^{(\om)}_{uv} \eta^{(T^n \om)}$ (following from \eqref{eq:dyn self-sim}), we see that setting
	\[ \X^{(\om)}_n[\Dk_k] := \left\{ z \in \X^{\om}_n : B_z^{(\om)} \subset D \text{ for some } D \in \Dk_k \right\} \]
	we have for each $u \in \X^{(\om)}_{k'}$:
	\[ \sum \limits_{v : uv \notin \X^{(\om)}_n[\Dk_k]} p_v^{(T^{k'} \om)} \leq \sum \limits_{v : uv \notin \X^{(\om)}_n[\Dk_k]} p_v^{(T^{k'} \om)} f^{(\om)}_{uv} \eta^{(T^n \om)}(T_v) \leq f^{(\om)}_u\eta^{(T^{k'}\om)} \left( \bigcup \limits_{v : uv \notin \X^{(\om)}_n[\Dk_k] } T_v \right) < \frac{\eps^2}{8}.\]
	By the calculation from the proof of \cite[p. 86]{SSS}\footnote{We have the constant $8$ rather than $4$ as in \cite{SSS}, since in $\R^2$ each ball $B_z^{(\om)},\ z \in \X^{(\om)}_n$ can intersect at most $4$ dyadic cells $D \in \Dk_k$.}
	\[ \bP_{i=k}\left(d_{TV}\left( \eta^{(\om)}_{x,i}, \eta^{(\om)}_{n,[x,i]}  \right) \geq \eps \right) \leq \frac{8}{\eps} \sum \limits_{z  \notin \X^{(\om)}_n[\Dk_k]} p_z^{(\om)}, \]
	hence 
	\[ \bP_{i=k}\left(d_{TV}\left( \eta^{(\om)}_{x,i}, \eta^{(\om)}_{n,[x,i]}  \right) \geq \eps \right) \leq \frac{8}{\eps} \sum \limits_{u \in \X^{(\om)}_{k'}} p^{(\om)}_u \sum \limits_{v : uv \notin \X^{(\om)}_n[\Dk_k]} p_v^{(T^{k'}\om)} < \eps.\]
\end{proof}

The following lemma is similar to \cite[Corollary 4.18]{SSS}. The difference is that we bound the entropy of the components $\eta^{(T^{n'}\om)}$ instead of $\eta^{(\om)}$.

\begin{lem}\label{lem: frostman entropy bound}
	For each $\eps, C, \rho > 0$, there exists $m_\eps = m(\eps, C, \rho)$ such that the following holds: For each $\om \in \Om$, if $m \geq m_\eps$ and $n,k \in \N$ are such that $\eta^{(T^{(n+k)'} \om)}$ is $(C, \rho)$-Frostman and $H_m(\eta^{(T^{(n+k)'} \om)}) > \alpha - \eps/8$, then
\begin{equation}\label{eq: frostman entropy bound}\bP_{i = k} \left( H_m \left( \eta^{(T^{n'}\om), x, i} \right) > \alpha - \eps \right) > 1 - \eps.
\end{equation}
\end{lem}

\begin{proof}
The proof has two steps:
\begin{enumerate}[(i)]
	\item\label{it: original Frostman entropy conditions} there exist $N_\eps = N(\eps,C,\rho)$ such that \eqref{eq: frostman entropy bound} holds for $m > m(\eps)$, provided that $\eta^{(T^{k'(T^{n'}\om) + n'(\om)} \om)}$ is $(C, \rho)$-Frostman and $H_m(\eta^{(T^{k'(T^{n'}\om) + n'(\om) + N_\eps} \om)}) > \alpha - \eps/4$;
	\item\label{it: better Frostman entropy conditions} for $m > m(\eps)$, instead of the Frostman and entropy conditions in \ref{it: original Frostman entropy conditions} above, it is enough to verify that $\eta^{(T^{(n+k)'} \om)}$ is $(C, \rho)$-Frostman and $H_m(\eta^{(T^{(n+k)'} \om)}) > \alpha - \eps/8$.
\end{enumerate}
The proof of  \ref{it: original Frostman entropy conditions} is exactly the same as the proof of \cite[Corollary 4.18]{SSS}, with $T^{n'} \om$ replacing $\om$. We provide it for completeness. Fix $\delta =\delta(\eps) \in (0,\eps)$ small enough to guarantee that $d_{TV}(\nu, \nu') < \delta$ implies $|H_m(\nu) - H_m(\nu')| < \eps/2$ for all $\nu,\nu' \in \Pk([-R, R]^2)$ (see \cite[Proposition 4.1(vi)]{SSS} which remains valid in higher dimensions). Let $N_\eps = N(\delta(\eps), C, \rho)$ be as in Lemma \ref{lem:TV distance}. Then, applying Lemma \ref{lem:TV distance} to $T^{n'}\om$ gives that if $\eta^{(T^{k'(T^{n'}\om) + n'(\om)} \om)}$ is $(C, \rho)$-Frostman, then
\begin{equation}\label{eq: high prob tv distance} \bP_{i=k}\left(d_{TV}\left( \eta^{(T^{n'}\om)}_{x,i}, \eta^{(T^{n'}\om)}_{k'(T^{n'}\om) + N_\eps,[x,i]}  \right) < \delta \right) > 1 - \delta \geq 1 - \eps.
\end{equation}
If $\eta^{(T^{n'}\om)}_{x,i}$ is a component from the event above, then for the dyadic cell $D=D_k(x)$ we have  \[ d_{TV}\left( \eta^{(T^{n'}\om),x,i}, T_D \eta^{(T^{n'}\om)}_{k'(T^{n'}\om) + N_\eps,[x,i]}  \right) = d_{TV}\left( \eta^{(T^{n'}\om)}_{x,i},\eta^{(T^{n'}\om)}_{k'(T^{n'}\om) + N_\eps,[x,i]}  \right)   < \delta\]
and hence by the choice of $\delta$ holds
\begin{equation}\label{eq:rescaled entropy comparision} \left| H_m\left(\eta^{(T^{n'}\om),x,i} \right) - H_m \left(T_D \eta^{(T^{n'}\om)}_{k'(T^{n'}\om) + N_\eps,[x,i]} \right) \right| < \eps/2 .
\end{equation}

Recall that $\eta^{(T^{n'}\om)}_{k'(T^{n'}\om) + N_\eps,[x,i]}$ is a convex combination of measures of the form $f^{(T^{n'}\om)}_u \eta^{(T^{k'(T^{n'}\om) + n'(\om) + N_\eps} \om)},\ u \in \X^{(T^{n'}\om)}_{k'(T^{n'}\om) + N_\eps}$, with weights from some probability vector $(q_u)_{u \in \X^{(T^{n'}\om)}_{k'(T^{n'}\om) + N_\eps}}$ . Each $f^{(T^{n'}\om)}_u$ above contracts by $\Theta_{N_\eps}(2^{-k})$, hence $T_Df^{(T^{n'}\om)}_u$ contracts by $\Theta_{N_\eps}(1)$. Therefore, we can apply \eqref{eq:entropy bounded similarity} together with the concavity of entropy to obtain
\[
\begin{split} H_m \left( T_D \eta^{(T^{n'}\om)}_{k'(T^{n'}\om) +  N_\eps,[x,i]}\right) & \geq \sum \limits_{u \in \X^{(T^{n'}\om)}_{k'(T^{n'}\om) + N_\eps}} q_u H_m\left( T_D f^{(T^{n'}\om)}_u \eta^{(T^{k'(T^{n'}\om) + n'(\om) + N_\eps} \om)} \right) \\
& \geq \sum \limits_{u \in \X^{(T^{n'}\om)}_{k'(T^{n'}\om) + N_\eps}} q_u \left( H_m\left( \eta^{(T^{k'(T^{n'}\om) + n'(\om) + N_\eps} \om)} \right) - O_{N_\eps}(1/m) \right) \\
& =  H_m\left( \eta^{(T^{k'(T^{n'}\om) + n'(\om) + N_\eps} \om)} \right) - O_{N_\eps}(1/m).
\end{split}
\]
If $H_m(\eta^{(T^{k'(T^{n'}\om) + n'(\om) + N_\eps} \om)}) > \alpha - \eps/4$, then for $m > m(\eps)$ the above inequality combined with \eqref{eq:rescaled entropy comparision} and \eqref{eq: high prob tv distance} proves the claim \ref{it: original Frostman entropy conditions}.

For  step \ref{it: better Frostman entropy conditions}, note that by \eqref{eq:cocycle}, measure $\eta^{(T^{k'(T^{n'}\om) + n'(\om)} \om)}$ can be obtained from $\eta^{(T^{(n+k)'} \om)}$ by at most $c$ iterates of the self-similarity relation \eqref{eq: one step dyn self-sim}, or vice-versa. Therefore, if $\eta^{(T^{(n+k)'} \om)}$ is $(C, \rho)$-Frostman, then $\eta^{(T^{k'(T^{n'}\om) + n'(\om)} \om)}$ is $(O(C), \rho)$-Frostman, so it suffices to verify the Frostman condition for $\eta^{(T^{(n+k)'} \om)}$ and adjust the constants. Similarly, by Lemma \ref{lem:shifted entropy comp},
\[ \left| H_m(\eta^{(T^{k'(T^{n'}\om) + n'(\om) + N_\eps} \om)}) - H_m(\eta^{(T^{(n+k)'} \om)})  \right| \leq O\left( \frac{c + N_\eps}{m} \right),\]
hence if $m$ is large enough (depending on $\eps$), it is enough to verify $H_m(\eta^{(T^{(n+k)'} \om)}) > \alpha - \eps/8$.
\end{proof}

{ The remainder of the proof of Proposition \ref{prop:unif ent dim} is similar to the proof of \cite[Theorem 4.7]{SSS} from \cite[Lemma 4.9 and Proposition 4.11]{SSS}.}\\

\begin{proof}[\bf Proof of Proposition  \ref{prop:unif ent dim}]
{  Fix $\eps>0$.  By \eqref{eq:inf dim} and Lemma \ref{lem:frostman on tubes}, there exist $C_\eps, \rho>0$ and $m_\eps \in \N$ such that for $\PP$-a.e.\ $\om \in \Om$ and every $m \geq m_\eps$,
\begin{equation}\label{eq:frostman and entropy frequency}
\begin{split}
\liminf \limits_{n \to \infty}&\ \frac{1}{n} \# \left\{ 1 \leq k \leq n : \eta^{(T^k \om)} \text{ is } (C_\eps,\rho)\text{-Frostman on tubes and } H_m(\eta^{(T^{k} \om)}) \geq \alpha - \eps/8 \right\}\\
&\geq 1-\eps.
\end{split}
\end{equation}
Lemma \ref{lem: frostman entropy bound} gives for $\om \in \Om$ and $m \geq m_\eps$ (possibly after increasing $m_\eps$):
\[ \bP _{1 \leq i \leq n} \left( H_{m}\left( \eta^{(T^{n'}\om), x, i} \right) > \alpha - \eps \right) \geq \frac{|\Theta^{(\om)}_n|}{n}( 1 -\eps), \]
where
\[
\begin{split} \Theta^{\pom}_n & = \left\{ 1 \leq k \leq n: \eta^{(T^{(n+k)'} \om)} \text{ is } (C_\eps, \rho)\text{-Frostman and } H_m(\eta^{(T^{(n+k)'} \om)}) > \alpha - \eps/8  \right\} \\
	& = \left\{ n+1 \leq k \leq 2n: \eta^{(T^{k'} \om)} \text{ is } (C_\eps, \rho)\text{-Frostman and } H_m(\eta^{(T^{k'} \om)}) > \alpha - \eps/8  \right\}.
\end{split}\]
Combining \eqref{eq:frostman and entropy frequency} with the claim \eqref{eq:freq general upper bound} of Lemma \ref{lem:k k' comparison} yields
\[ \liminf \limits_{n \to \infty} \frac{|\Theta^{(\om)}_n|}{n} \geq 1 - D\eps \]
for some constant $D>0$. Consequently, for $\PP$-a.e.\ $\om \in \Om$,
\[ \liminf \limits_{n \to \infty}  \bP _{1 \leq i \leq n} \left(  H_{m}\left( \eta^{(T^{n'}\om), x, i} \right)> \alpha - \eps \right) \geq 1 - \eps'\]
with $\eps' = \eps'(\eps)$ such that $\eps'(\eps) \to 0$ as $\eps \to 0$. Therefore for $m>m(\eps)$
\[ \lim \limits_{n \to \infty}\ \PP \left( \left\{ \om \in \Om : \bP _{1 \leq i \leq n} \left(  H_{m}\left( \eta^{(T^{n'}\om), x, i} \right) \geq \alpha - \eps \right) > 1 - 2\eps' \right\} \right) = 1. \]
On the other hand, combining Lemma \ref{lem:n' entropy dim} (for $q=1$) and Lemma \ref{lem: local to global entropy}  gives for $m > m(\eps)$
\[ \lim \limits_{n \to \infty}\ \PP \left( \left\{ \om \in \Om :  \left| \bE _{1 \leq i \leq n}  H_{m}\left( \eta^{(T^{n'}\om), x, i} \right) - \alpha \right| < \eps  \right\} \right) = 1. \]
As $0 \leq H_{m}\left( \eta^{(T^{n'}\om), x, i} \right)  \leq 2$, the last two equalities finish the proof.
}

\end{proof}

\subsection{Saturation on subspaces of positive dimension}
Now we are ready to combine Lemma \ref{lem:proj entropy improved comp} and Proposition \ref{prop:unif ent dim} to prove Proposition \ref{prop:sat-dim 0}.

\begin{proof}[\bf Proof of Proposition \ref{prop:sat-dim 0}]
	Assume first $\alpha \geq 1$. Then Proposition \ref{prop:unif ent dim} implies that for $\eps < \frac{2-\alpha}{2}$ and $m>m(\eps)$ holds
	\begin{equation}\label{eq:satdim < 2}\lim \limits_{n \to \infty}\ \PP \left( \left\{ \om \in \Om :   \frac{1}{n}\# \left\{ 0 \leq k < n : \satdim(\eta^{(T^{n'}\om)}, \eps, m, k) < 2 \right\}  \geq 1 - \eps \right\} \right) = 1.
	\end{equation}
Indeed, we need to bound from above the probability that the saturation dimension of $\eta^{(T^{n'}\om)}$ is equal to $2$, with the the given parameters. The latter means, by 
Definition~\ref{def:satdim}, that, except in 
an $\eps$ proportion of scales and component measures in that scale, we have $H_m(\eta^{(T^{n'}\om,x,i})\ge 2-\eps$, see
Definition~\ref{def:saturated}. But this contradicts the inequality in the event appearing in \eqref{eq:unif ent dim} by our choice of $\eps$, hence the claim \eqref{eq:satdim < 2}. By Lemma \ref{lem:proj entropy improved comp}, for $m > m(\eps)$ we have
\[
\begin{alignedat}{2}
\lim\limits_{n \to \infty}\ &\PP   \bigg( \bigg\{ \om \in \Om :  \frac{1}{n}   \# \bigg\{  0 \leq & & k < n : \\
&  & &\inf \limits_{W \in \R\PP^1} \bP_{i=k} \left( H_m(\pi_W \eta^{(T^{n'} \om), x, i}) \geq \alpha - 1 + \kappa'' \right) > p_1 \bigg\}  \geq 1 - 2\eps \bigg\} \bigg)  \\
& = 1. & &
\end{alignedat}
\]
On the other hand, combining the above equality with Proposition \ref{prop:unif ent dim} and 
	\[  \frac{1}{m} H(\eta^{(T^{n'}\om),x,i} , \Dk^{W}_{m} | \Dk^{W^\perp}_m) = H_m(\eta^{(T^{n'}\om),x,i}) - H_m(\pi_{W^\perp}\eta^{(T^{n'}\om),x,i}) + O\left(\frac{1}{m} \right), \]
 see \eqref{eq:entropy cond W}, give for small enough $\eps > 0$ and $m > m(\eps)$:
	\[
	\begin{alignedat}{2}
	\lim\limits_{n \to \infty}\ & \PP  \bigg( \bigg\{ \om \in \Om :   \frac{1}{n}\# \bigg\{ & & 0 \leq k < n :  \\
	& & & \inf \limits_{W \in \R\PP^1} \bP_{i=k} \left( \frac{1}{m} H(\eta^{(T^{n'}\om),x,i} , \Dk^{W}_{m} | \Dk_m^{W^\perp}) \leq 1 - \kappa''/2 \right) > p_1/2 \bigg\} \geq 1 - \eps \bigg\} \bigg) \\
	& = 1. & &
	\end{alignedat}\]
	By Lemma \ref{lem:sat by cond}, this implies for $\eps < \min\{ \kappa''/4, p_1/4  \}$ and $m > m(\eps)$,
\[ \lim \limits_{n \to \infty}\ \PP \left( \left\{ \om \in \Om :  \frac{1}{n}\# \left\{ 0 \leq k < n :  \satdim(\eta^{(T^{n'}\om)}, \eps, m, k) \neq 1 \right\} \geq 1 - \eps \right\} \right) = 1. \]
	Recalling \eqref{eq:satdim < 2} finishes the proof in the case $\alpha \geq 1$.
	If $\alpha < 1$, then Proposition \ref{prop:unif ent dim} implies that \eqref{eq:satdim < 2} holds with $\satdim(\eta^{(T^{n'}\om)}, \eps, m, k) < 1$, giving directly the assertion of the proposition.
\end{proof}

\section{Proof of Theorem \ref{thm:main dim}}

In this section we will use Proposition \ref{prop:sat-dim 0} and  Theorem \ref{thm: inverse convolution} to prove Theorem \ref{thm:main dim}. It will be useful for us to apply the inverse theorem in the following form, obtained directly from Theorem \ref{thm: inverse convolution}:

\begin{thm}\label{thm: inverse convolution entropy}
	For every $R,\eps > 0$ and $m \in \N$ there exist $\eps' = \eps'(\eps)>0$, $\delta = \delta(\eps, R, m) > 0$,
	 such that for every $n > n(\eps, R, \delta, m)$, the following holds: If $\mu, \nu \in \mP([-R, R]^d)$ are such that
	\[  \frac{1}{n}\# \left\{ 0 \leq k < n : \satdim(\mu, \eps', m, k) = 0 \right\} \geq 1 - \eps'\  \text{ and }\ H_n(\mu * \nu) < H_n(\mu) + \delta, \]
	then
	\[H_n(\nu) < \eps.\]
\end{thm}
\begin{proof} We can assume $\eps \leq 1$. Fix $\eps'>0$ and let $\delta = \delta(\eps', R, m)$ and $n(\eps', R, \delta, m)$ be from Theorem \ref{thm: inverse convolution}. If for $n > n(\eps', R, \delta, m)$ we have $H_n(\mu * \nu) < H_n(\mu) + \delta$, then by Theorem \ref{thm: inverse convolution} there exists a sequence $V_0, \ldots, V_{n-1} \leq \R^d$ of subspaces such that
	\begin{equation}\label{eq: sat con} \bP_{0 \leq i < n} \begin{pmatrix} \mu^{x,i} \text{ is } (V_i, \eps', m)\text{-saturated and }\\
	\nu^{y,i} \text{ is } (V_i,\eps')\text{-concentrated} \end{pmatrix} > 1 - \eps'.
	\end{equation}
{ Therefore,
	\[
	\begin{split}\frac{1}{n}\#  \Big\{ 0 \leq k < n & : \satdim(\mu, \sqrt{\eps'}, m, k) < \dim V_k  \Big\} \\
		& \leq  \frac{1}{n}\# \left\{ 0 \leq k < n : \bP_{i=k} \left( \mu^{x,i} \text{ is not }  (V_i, \sqrt{\eps'}, m)\text{-saturated}  \right) \geq \sqrt{\eps'} \right\} \\
		&\leq\frac{1}{n}\# \left\{ 0 \leq k < n : \bP_{i=k} \left( \mu^{x,i} \text{ is not }  (V_i, \eps', m)\text{-saturated}  \right) \geq \sqrt{\eps'} \right\}\\
		&  \overset{\mathclap{\text{Chebyshev's ineq.}}}{\leq}\ \ \ \ \ \ \ \  \  \frac{1}{n\sqrt{\eps'}}\sum \limits_{k=0}^{n-1} \bP_{i=k} \left( \mu^{x,i} \text{ is not }  (V_i, \eps', m)\text{-saturated}  \right) \\
		& = \frac{1}{\sqrt{\eps'}} \bP_{0 \leq i < n} \left( \mu^{x,i} \text{ is not }  (V_i, \eps', m)\text{-saturated}  \right) \\
		& \leq \sqrt{\eps'},
	\end{split}\]
	hence 
		\[  \frac{1}{n}\# \left\{ 0 \leq k < n :  \satdim(\mu, \sqrt{\eps'}, m, k) \geq \dim V_k \right\} \geq 1 - \sqrt{\eps'}. \]
	If $\frac{1}{n}\# \left\{ 0 \leq k < n : \satdim(\mu, \eps', m, k) = 0 \right\} \geq 1 - \eps'$}, then we can conclude
	\[  \frac{1}{n}\# \left\{ 0 \leq i < n : \dim V_i = 0 \right\} \geq 1 - \eps' - \sqrt{\eps'} \geq 1 - 2 \sqrt{\eps'} \]
	and hence by \eqref{eq: sat con},
	\begin{equation}\label{eq:concentration corollary} \bP_{0 \leq i \leq n} \left( \nu^{x,i} \text{ is } (\{ 0 \},\eps')\text{-concentrated} \right) \geq 1 - 2\sqrt{\eps'} - \eps' \geq 1 -3\sqrt{\eps'}.
	\end{equation}
	Set $m_\eps' = \lfloor\log (1/\eps')\rfloor$ (so that $2^{-m_\eps'} \approx \eps'$). Note that if $\nu^{x,i}$ is $(\{0\},\eps')$-concentrated then $H_{m'_\eps}(\nu^{x,i}) \leq O(1/m_{\eps'}) + 2\eps' = O(1/m_{\eps'})$ for $\eps'$ small enough. Therefore, \eqref{eq:concentration corollary} together with Lemma \ref{lem: local to global entropy} gives for $n$ large enough:
	\[\begin{split}
	H_n(\nu) & = \bE_{0 \leq i \leq n} \left( H_{m_\eps'}(\nu^{x,i}) \right) + O\left( \frac{m_\eps' + \log R}{n} \right) \\
	& \leq 3\sqrt{\eps'} O(1) + (1 - 3\sqrt{\eps'})O(1/m_\eps') + O\left( \frac{m_\eps' + \log R}{n} \right) \\
	& = 3\sqrt{\eps'} O(1) + (1 - 3\sqrt{\eps'})O\left(\frac{1}{\log(1/\eps')}\right) + O\left( \frac{\log(1/\eps') + \log R}{n} \right).
	\end{split}\]
	The last expression can be made smaller than $\eps$ if $\eps'$ is small enough (depending on $\eps$) and $n = n(\eps')$ is large enough. This finishes the proof.
\end{proof}

The remainder of the proof of Theorem \ref{thm:main dim} follows closely the proof of \cite[Theorem 1.3]{SSS}. The only essential difference is that instead of the uniform entropy dimension of $\eta^{(\om)}$, we need to use Proposition \ref{prop:sat-dim 0}. We include a detailed sketch for the convenience of the reader.

Following \cite{SSS}, we have the following proposition, based on the decomposition \eqref{eq:convolution decomposition}. For $n \in \N,\ \eps>0, \delta>0$, and $q \in \N$ let
\[ \mG^{(n)}_{\eps, \delta, q} := \left\{ \om \in \Om : \bP_{i = n} \left( \frac{1}{qn} \left| H\left( \nu^{(\om, n')}_{x,i} * \tau^{(\om, n')}, \Dk_{(q+1)n} \right) - H \left( \tau^{(\om, n')}, \Dk_{(q+1)n} \right) \right| < \delta \right) > 1 - \eps \right\}, \]
where $n' = n'(\om)$ is defined as in \eqref{eq:prime def}.
\begin{prop}\label{prop: G measure}
Let $\Sigma$ be a model. Then for every $\eps, \delta > 0$ and $q \in \N$ holds
\[ \lim \limits_{n \to \infty} \PP\left( \mG^{(n)}_{\eps, \delta, q} \right) = 1.\]
\end{prop}
\begin{proof}

As $|\lam_{\om_1}^{-1} \cdots \lam_{\om_{n'}}^{-1}| = r_{\om_1}^{-1} \cdots r_{\om_{n'}}^{-1} = O(2^n)$, rescaling by $\lam_{\om_1}^{-1} \cdots \lam_{\om_{n'}}^{-1}$ gives (recall \eqref{eq:entropy bounded similarity})
\begin{equation}\label{eq: tau to eta entropy} H(\tau^{(\om, n')}, \Dk_{(q+1)n})  = H(\eta^{(T^{n'} \om)}, \Dk_{qn}) + O(1),
\end{equation}
and we therefore obtain from Lemma \ref{lem:n' entropy dim}:\footnote{Obtaining \eqref{eq:tau entropy} is the first instance mentioned in Remark \ref{rem:SSSgap}, where a slightly different argument than the one given in \cite{SSS} is needed. In \cite[page 90, item (ii)]{SSS}, the $T$-invariance of $\PP$ is invoked to obtain equality $ \PP\left( \left\{ \om \in \Om :  \left| \frac{1}{qn}H(\eta^{(T^{n'} \om)}, \Dk_{qn}) - \alpha \right| < \delta  \right\} \right) = \PP\left( \left\{ \om \in \Om :  \left| \frac{1}{qn}H(\eta^{\pom}, \Dk_{qn}) - \alpha \right| < \delta  \right\} \right)$, which together with \eqref{eq:inf dim} and \eqref{eq: tau to eta entropy} implies \eqref{eq:tau entropy}. However, as $n'$ depends on $\om$, the $T$-invariance of $\PP$ seems insufficient to obtain the above equality. On the other hand, it is easy to see that it holds if $\PP$ is Bernoulli, as then $T^{n'} \om$ has the same distribution as $\om$. For general ergodic measures, one can prove and apply Lemma \ref{lem:n' entropy dim} also in the setting of \cite{SSS}.}
\begin{equation}\label{eq:tau entropy} \lim \limits_{n \to \infty} \PP\left( \left\{ \om \in \Om :  \left| \frac{1}{qn}H(\tau^{(\om, n')}, \Dk_{(q+1)n}) - \alpha \right| < \delta  \right\} \right)  = 1.
\end{equation}
Moreover, \eqref{eq:inf dim} implies
\begin{equation}\label{eq:cond entropy alpha} \lim \limits_{n \to \infty} \frac{1}{qn} H\left(\eta^{(\om)}, \Dk_{(q+1)n}|\Dk_{n}\right) =  \lim \limits_{n \to \infty} \frac{1}{qn} \left( H(\eta^{(\om)}, \Dk_{(q+1)n}) - H(\eta^{(\om)}, \Dk_{n}) \right) = \alpha
\end{equation}
for $\PP$-a.e.\ $\om \in \Om$. By \eqref{eq:convolution decomposition} and concavity of the conditional entropy, we get for every $\om \in \Om$:
\[
\begin{split}  H\left(\eta^{(\om)}, \Dk_{(q+1)n}|\Dk_{n}\right) & \geq \bE_{i=n} H\left(\nu^{(\om, n')}_{x,i} * \tau^{(\om, n')}, \Dk_{(q+1)n}|\Dk_{n}\right) \\
& = \bE_{i=n} H(\nu^{(\om, n')}_{x,i} * \tau^{(\om, n')}, \Dk_{(q+1)n}) - O(1),
\end{split} \]
where the last equality follows from the fact that each measure $\nu^{(\om, n')}_{x,i} * \tau^{(\om, n')}$ is supported in a ball of diameter at most $O(2^{-n})$. Combining this with \eqref{eq:cond entropy alpha} gives for every $\delta > 0$:
\begin{equation}\label{eq:expected convolution entropy} \lim \limits_{n \to \infty} \PP \left( \left\{ \om \in \Om : \frac{1}{qn} \bE_{i=n} H(\nu^{(\om, n')}_{x,i} * \tau^{(\om, n')}, \Dk_{(q+1)n}) \leq \alpha + \delta \right\} \right) = 1.
\end{equation}
On the other hand, by \eqref{eq:conv entropy bound} we have
\[ H(\nu^{(\om, n')}_{x,i} * \tau^{(\om, n')}, \Dk_{(q+1)n}) \geq H(\tau^{(\om, n')}, \Dk_{(q+1)n}) - O(1),\]
hence by \eqref{eq:tau entropy},
\[ \lim \limits_{n \to \infty} \PP\left( \left\{ \om \in \Om : \frac{1}{qn}H(\nu_{x,i}^{(\om, n')} * \tau^{(\om, n')}, \Dk_{(q+1)n}) \geq \alpha - \delta \text{ for every } x \in \R^2\right\} \right) = 1 \]
for every $\delta > 0$. The above combined with \eqref{eq:expected convolution entropy} implies that for any $\delta, \eps>0$ holds
\[ \lim \limits_{n \to \infty} \PP\left( \left\{ \om \in \Om : \bP_{i=n} \left( \left| \frac{1}{qn}H(\nu_{x,i}^{(\om, n')} * \tau^{(\om, n')}, \Dk_{(q+1)n}) - \alpha \right| < \delta \right) > 1-\eps \right\} \right) = 1. \]
Comparing this with \eqref{eq:tau entropy} finishes the proof.
\end{proof}

The next theorem is analogous to \cite[Theorem 4.3]{SSS}. The only difference in the proof is the use of the two-dimensional Hochman's inverse theorem and an application of Proposition \ref{prop:sat-dim 0} to exclude the possibility of $\tau^{(\om, n')}$ being saturated on one-dimensional subspaces for a positive frequency of levels. Similarly to \cite[Theorem 4.3]{SSS} being a ``random version'' of \cite[Theorem 1.3]{H}, our theorem can be viewed as an extension of \cite[Theorem 1.5]{HRd} to the setting of models in our special two-dimensional case.

\begin{thm}\label{thm:main entropy on scales}
	Let $\Sigma$ be a model satisfying the assumptions of Theorem \ref{thm:main dim} and such that $r_{\max} \leq 1/2$. If $\dim(\Sigma) < 2$, then
	\[ \frac{1}{n}H\left( \nu^{(\cdot, n'(\cdot))}, \Dk_{(q+1)n}| \Dk_n \right) \stackrel{\PP}{\longrightarrow} 0 \text{ for all } q \in \N,\]
	where $n' = n'(\om)$ is defined as in \eqref{eq:prime def}.
\end{thm}

\begin{proof}
Let $\eps_0 > 0$ be from Proposition \ref{prop:sat-dim 0}. Consider a sequence $\eps_j \in (0, \eps_0)$ with $\sum \limits_{j=1}^\infty \eps_j < \infty$ and let $\eps'_j = \eps'_j(\eps_j) < \eps_0$ be as in Theorem \ref{thm: inverse convolution entropy}. Using Proposition   \ref{prop:sat-dim 0}, for $j \in \N$ choose $m_j \in \N$ such that for all $n$ large enough holds\footnote{This is the second place mentioned by Remark \ref{rem:SSSgap}. Again, \cite[(37)]{SSS} was proved by applying $T$-invariance of $\PP$ to conclude from the uniform entropy dimension property \cite[Corollary 4.13]{SSS} that one can choose $m_j$ such that for all $n$ large enough  $\PP \left( \left\{ \om \in \Om ; \bp_{1 \leq i \leq n} \left( \left| H_{m_j} \left( \eta^{(T^{n'} \om),x,i}  \right) - \alpha \right| < \eps_j \right) > 1 - \eps_j \right\} \right) = \PP \left( \left\{ \om \in \Om ; \bp_{1 \leq i \leq n} \left( \left| H_{m_j} \left( \eta^{(\om),x,i}  \right) - \alpha \right| < \eps_j \right) > 1 - \eps_j \right\} \right) > 1  - \eps_j$ holds. Similarly as before, the above equality is justified if $\PP$ is Bernoulli, while for general ergodic measures one can invoke Proposition \ref{prop:unif ent dim}, which holds also in the setting of \cite{SSS}.}
\begin{equation}\label{eq:sat-dim 0 prob} \PP \left( \left\{ \om \in \Om :  \frac{1}{n}\# \left\{ 0 \leq k < n :  \satdim(\eta^{(T^{n'}\om)}, \eps'_j, m_j, k) = 0 \right\} \geq 1 - \eps'_j \right\} \right) > 1 -\eps_j/2.
\end{equation}
For each $j \in \N$, let $\delta_j = \delta_j(\eps_j, R)$ be given by Theorem \ref{thm: inverse convolution entropy} and choose $n_j \geq n(\eps_j, R, \delta_j, m_j)$, where $n(\eps_j, R, \delta_j, m_j)$ is from Theorem \ref{thm: inverse convolution entropy}, such that \eqref{eq:sat-dim 0 prob}  holds for $n = n_j$ and also
\[ \PP\left( \mG^{(n_j)}_{\eps_j, \frac{\delta_j}{2}, q} \right) \geq 1 - \eps_j/2. \]
It exists due to Proposition \ref{prop: G measure}.
Let $\Om_j$ be the intersection of $\mG^{(n_j)}_{\eps_j, \frac{\delta_j}{2}, q}$ with the event from \eqref{eq:sat-dim 0 prob} for $n = n_j$. Then
\begin{equation}\label{eq:Om_j measure}
\PP(\Om_j) \geq 1 - \eps_j.
\end{equation}
Note that by rescaling by $r_{\om_1} \cdots r_{\om_{n'}} = \Theta(2^{-n})$, we have (due to \eqref{eq:entropy bounded similarity}) that
\[ H \left( \tau^{(\om, n')}, \Dk_{(q+1)n} \right) = H \left( \eta^{(T^{n'}\om)}, \Dk_{qn} \right) + O(1), \]
and for $i=n$, by applying $T_{\Dk_n(x)}$ (see the calculation in \cite[p.\ 91]{SSS}),
\[ H\left( \nu^{(\om, n')}_{x,i} * \tau^{(\om, n')}, \Dk_{(q+1)n} \right) = H\left( \nu^{(\om, n'), x,i} * \eta^{(T^{n'}\om)}, \Dk_{qn} \right) + O(1).  \]
Therefore, if $\nu^{(\om, n'_j)}_{x,i}$ is a level-$n_j$ component such that
\begin{equation}\label{eq:convolution vs non-convolution} \frac{1}{qn_j} \left| H\left( \nu^{(\om, n'_j)}_{x,i} * \tau^{(\om, n'_j)}, \Dk_{(q+1)n_j} \right) - H \left( \tau^{(\om, n'_j)}, \Dk_{(q+1)n_j} \right) \right| < \delta_j/2,
\end{equation}
then also, provided that $n_j$ is large enough (depending on $\delta_j$),
\[ \frac{1}{qn_j} \left| H\left( \nu^{(\om, n'_j),x,i} * \eta^{(T^{n'_j}\om)}, \Dk_{qn_j} \right) - H \left( \eta^{(T^{n'_j}\om)}, \Dk_{qn_j} \right) \right| < \delta_j. \] 
Consequently, for $\om \in \Om_j$ and any level-$n_j$ component $\nu^{(\om, n'_j)}_{x,i}$ satisfying \eqref{eq:convolution vs non-convolution}, we have by 
Theorem \ref{thm: inverse convolution entropy},
applied with $\mu=\eta^{(T^{n'_j}\om)},\ \nu = \nu^{(\om, n'_j), x, i}$, that
\[ H_{qn_j}\left(\nu^{(\om, n'_j),x,i}\right) < \eps_j,\]
hence for $\om \in \Om_j$ holds
\[ \bP_{i = n_j} \left( H_{qn_j}\left(\nu^{(\om, n'_j),x,i}\right) < \eps_j \right) > 1 - \eps_j.\]
As $H_{qn_j}\left(\nu^{(\om, n'_j),x,i}\right) \leq 2$ for every component, this gives
\[ \frac{1}{qn_j}H\left( \nu^{(\om, n'_j)}, \Dk_{(q+1)n_j} | \Dk_{n_j} \right) = \bE_{n_j} \left( H_{qn_j}\left( \nu^{(\om, n'_j), x, i} \right) \right) \leq 3 \eps_j. \]
Thus, by \eqref{eq:Om_j measure}
\[ \PP \left( \left\{ \om \in \Om : \frac{1}{qn_j}H\left( \nu^{(\om, n'_j)}, \Dk_{(q+1)n_j} | \Dk_{n_j} \right) \leq 3 \eps_j \right\} \right) \geq 1 - \eps_j. \]
As $\sum \limits_{j=1}^\infty \eps_j < \infty$, the Borel-Cantelli Lemma yields
\[ \lim \limits_{j \to \infty} \frac{1}{qn_j}H\left( \nu^{(\om, n'_j)}, \Dk_{(q+1)n_j} | \Dk_{n_j} \right) = 0\ \text{ for } \PP \text{-a.e. } \om \in \Om.  \]
Since the only requirement on $n_j$ was to be large enough, we can repeat the argument assuming that $n_j$ is a subsequence of any given sequence of integers diverging to infinity. Therefore,
\[ \frac{1}{qn}H\left( \nu^{(\cdot, n'(\cdot))}, \Dk_{(q+1)n}| \Dk_n \right) \stackrel{\PP}{\longrightarrow} 0,\]
concluding the proof.
\end{proof}

We are now ready to finally conclude the proof of Theorem \ref{thm:main dim} under the additional assumption $r_{\mathrm{max}} \leq 1/2$.\\

\begin{prop}\label{prop:main dim 1/2}
Theorem \ref{thm:main dim} holds under additional assumption $r_{\mathrm{max}} \leq 1/2$.
\end{prop}
\begin{proof}
The deduction of Theorem \ref{thm:main dim} from Theorem \ref{thm:main entropy on scales} is exactly the same as the one of \cite[Theorem 1.3]{SSS} from \cite[Theorem 4.3]{SSS}. As in our case the proof of Theorem  \ref{thm:main entropy on scales} uses the assumption $r_{\mathrm{max}} \leq 1/2$, we (so far) obtain the result under this additional assumption. We provide only a sketch, and refer the reader to \cite[Proof of Theorem 1.3]{SSS}.

Assume $\alpha := \dim(\Sigma) < \min\{ 2, \sdim(\Sigma)\}$ and $r_{\mathrm{max}} \leq 1/2$. Fix $q \in \N$. By \eqref{eq:inf dim}, as $\|\Pi^{(n')}_\om - \Pi_\om\| \leq O(2^{-n})$, we have for $\PP$-a.e.\ $\om \in \Om$:
\[ \lim \limits_{n \to \infty} H_n\left( \nu^{(\om, n')} \right) = \alpha \]
(see \cite[Lemma 4.15]{SSS} and its proof). Therefore, Theorem \ref{thm:main entropy on scales} yields
\[ \frac{1}{n}H\left( \nu^{(\cdot, n'(\cdot))}, \Dk_{(q+1)n} \right) = H_n(\nu^{(\cdot, n'(\cdot))}) + \frac{1}{n}H\left( \nu^{(\cdot, n'(\cdot))}, \Dk_{(q+1)n}| \Dk_n \right) \overset{\PP}{\longrightarrow} \alpha. \]
As $\alpha < \sdim(\Sigma)$, there exists $\delta> 0$ such that
\begin{equation}\label{eq:qn entropy < sdim} \lim \limits_{n \to \infty} \PP \left( \left\{ \om \in \Om : \frac{1}{n}H\left( \nu^{(\om, n')}, \Dk_{(q+1)n} \right) < \sdim(\Sigma) - \delta \right\} \right) = 1.
\end{equation}
Recall that
\[ \nu^{(\om, n')} = \sum \limits_{u \in \X^{(\om)}_{n'}} p^{(\om)}_u \delta_{f^{(\om)}_u(0)}.\]
The crucial observation is the following: For given $\om \in \Om, n \in \N$, if each value $f^{(\om)}_u(0)$ among $u \in \X^{(\om)}_{n'}$ belongs to a different element of the partition $\Dk_{(q+1)n}$, then
\[ \frac{1}{n}H\left( \nu^{(\om, n')}, \Dk_{(q+1)n} \right) = \frac{1}{n}H\left(p_{\om_1} \otimes \cdots \otimes p_{\om_{n'}} \right) = \frac{n'}{n} \cdot \frac{\sum \limits_{i=1}^{n'}H(p_{\om_{i}})}{n'} =: s(\om,n). \]
Therefore, if $\frac{1}{n}H\left( \nu^{(\om, n')}, \Dk_{(q+1)n} \right) < s(\om,n)$, then at least two points of the form $f^{(\om)}_u(0), u \in \X^{(\om)}_{n'}$, belong to the same element of $\Dk_{(q+1)n}$, hence $\Delta^{(\om)}_{n'} \leq \sqrt{2} 2^{-(q+1)n}$. As $cn \leq n'(\om) \leq Cn$ for every $\om \in \Om,\ n \in \N$, and some constants $c,C>0$, depending only on $\Sigma$, we conclude using the above observations and Lemma \ref{lem:Delta inequality} that for each $n \in \N$ holds
\begin{equation}\label{eq:superexp ineq}
\begin{split}
\PP & \left( \left\{ \om \in \Om : \frac{\log \Delta^{(\om)}_{\lceil Cn \rceil }}{n} \leq -(q+1) + \frac{\log \sqrt{2}}{n} \right\} \right)  \\
& \geq \PP \left( \left\{ \om \in \Om : \frac{1}{n}H\left( \nu^{(\om, n')}, \Dk_{(q+1)n} \right) < s(\om,n) \text{ and } |\X^{(\om)}_{\lfloor cn \rfloor }|>1 \right\} \right).
\end{split}
\end{equation}
By the ergodic theorem applied to the functions $\om \mapsto H_{p_{\om_1}},\ \om \mapsto \log r_{\om_1}$, we have $\lim \limits_{n \to \infty} s(\om, n) = \sdim(\Sigma)$ for $\PP$-a.e.\ $\om \in \Om$. Moreover, Assumption \ref{assmp:non-deg and rotation}\ref{it:non-degenerate} implies
\[\lim \limits_{n \to \infty} \PP \left( \left\{ \om \in \Om : |\X^{(\om)}_{\lfloor cn \rfloor }|>1 \right\} \right) = 1.\]
Combining this with \eqref{eq:qn entropy < sdim} and \eqref{eq:superexp ineq} finishes the proof.
\end{proof}

We shall now explain why the technical assumption $r_{\max} \leq 1/2$ can be omitted, i.e. how to deduce  Theorem \ref{thm:main dim} from Proposition \ref{prop:main dim 1/2}. Below we present the proof for  $\PP$ being a Bernoulli measure, as many details simplify in this case, yet it is sufficient for the proof of Theorem \ref{thm:main ac}. We give the proof of the full version for general ergodic measures in Appendix \ref{app:ergodic case}. 

\begin{proof}[\bf Proof of Theorem \ref{thm:main dim} (for Bernoulli $\PP$)]
Let $\Sigma = \left( (\Phi^{(i)})_{i \in I}, (p_i)_{i \in I}, \PP \right)$ be a model satisfying  Assumption \ref{assmp:non-deg and rotation} with $\PP = q^{\otimes\N}$ for a probability vector $q$ on $I$. Without loss of generality we can assume that $\PP$ is fully supported. Fix $k_0 \in \N$ such that $r_{\max}^{k_0} = \left(r_{\max}(\Sigma) \right)^{k_0} \leq 1/2$. For $k \geq k_0$, consider a new model $\Sigma' = \left( \left(\Psi^{(\ov{i})}\right)_{\ov{i} \in I^k}, \left(p_{\ov{i}}\right)_{\ov{i} \in I^k}, \PP' \right)$ with the index set $I' = I^k$, where for   $\ov{i} = (i_1, \ldots, i_k) \in I^k$ we put
\[\Psi^{(\ov{i})} = \left\{ f^{(i_1)}_{u_1} \circ \cdots \circ f^{(i_k)}_{u_k} : u_j \in \{1, \ldots, k_{i_j}\} \text{ for } 1 \leq j \leq k \right\},\]
together with $p_{\ov{i}} = p_{i_1} \otimes \cdots \otimes p_{i_k}$ and $\PP' = (q^{\otimes k})^{\otimes \N}$ (so $\Sigma'$ is the ``$k$-iterate'' of $\Sigma$). We claim that $\Sigma'$ verifies Assumption \ref{assmp:non-deg and rotation} for infinitely many $k \in \N$. As $r_{\max}(\Sigma') \leq 1/2$ for $k \geq k_0$, this will imply that there exists $k$ such that we can apply the already established version of Theorem \ref{thm:main dim} (i.e. Proposition \ref{prop:main dim 1/2}) to $\Sigma'$.  Clearly, as $\Sigma$ satisfies  Assumption \ref{assmp:non-deg and rotation}\ref{it:non-degenerate}, then so does $\Sigma'$ for any $k \geq 1$. Moreover, $\Sigma'$ satisfies Assumption  \ref{assmp:non-deg and rotation}\ref{it:non-real} for infinitely many $k$. Indeed, if for given $k$ it is not satisfied, then $\vphi_{i_1}\cdots\vphi_{i_k} \in \R$ for all $(i_1, \ldots, i_k) \in I^k$. As by the assumption on $\Sigma$, there exists $i \in I$ such that $\vphi_i \notin \R$, we see that $\vphi_{i_1}\cdots\vphi_{i_k} \vphi_i \notin \R$, hence $\Sigma'$ satisfies Assumption \ref{assmp:non-deg and rotation}\ref{it:non-real} for $k+1$. As $\PP'$ is ergodic for any $k$ (as a Bernoulli measure) Assumption \ref{assmp:non-deg and rotation} is indeed satisfied by $\Sigma'$ for infinitely many $k$.

Let us fix now $k \geq k_0$ such that $\Sigma'$ satisfies Assumption \ref{assmp:non-deg and rotation}. Let $\Om = I^\N,\ \Om' := \left(I^k\right)^\N$ and note that the map $h  : \Om \to \Om'$ given by
\[h(\om_1, \om_2, \ldots) = ((\om_1, \ldots, \om_{k}), (\om_{k+1}, \ldots, \om_{2k}), \ldots)\]
is a bi-measurable bijection transporting $\PP$ to $\PP'$. It is easy to see that $\eta^{(\om)} = \eta'^{(h(\om))}$ for every $\om \in \Om$, where $\eta'^{(\cdot)}$ denotes the projected measures in the model $\Sigma'$, thus $\dim(\Sigma) = \dim(\Sigma')$. Similarly, it is a straightforward calculation that $\sdim(\Sigma) = \sdim(\Sigma')$. Moreover, for each $\om \in \Om$, we have $\Delta_{kn}^{(\om)} =\Delta_{n}^{'(h(\om))}$, where $\Delta_{n}^{'(\cdot)}$ is defined as in \eqref{eq:delta def} for the model $\Sigma'$. Therefore, if $\dim(\Sigma) < \min\{2, \sdim(\Sigma)\}$ then also $\dim(\Sigma') < \min\{2, \sdim(\Sigma')\}$, thus Proposition \ref{prop:main dim 1/2} gives
\[ \frac{\log \Delta_{kn}^{(\cdot)}}{n} \stackrel{\PP}{\longrightarrow} -\infty. \]
It remains to improve that to $\frac{\log \Delta_{n}^{(\cdot)}}{n} \stackrel{\PP}{\longrightarrow} -\infty$. This follows from Lemma \ref{lem:Delta inequality}, as Assumption \ref{assmp:non-deg and rotation}\ref{it:non-degenerate} for $\Sigma$ implies $\lim \limits_{n \to \infty} \PP \left( \left\{ \om \in \Om : |\X^{(\om)}_n |> 1 \right\} \right) = 1$.
\end{proof}

\section{Fourier decay}
In this section we combine  the methods of \cite[Theorem D]{SS16} and \cite[Theorem 1.5]{SSS} 
to show that for a ``typical'' non-degenerate model the measure $\eta^{(\om)}$ 
has power Fourier decay at infinity.

Given $1< a < b<\infty$ and $\eta>0$, let
\be \label{defH}
A_{a,b,\eta} := \{ z \in \C:\ a \le  |z| \le  b,\ \Im(z) > \eta\}.
\ee

We will parametrize the contractions of the linear maps by assuming that 
$\lam_i = \lam_1^{\beta_i}$ for some fixed complex numbers, with $\beta_1=1$, 
and use $\theta:=\lam_1^{-1}$ as a parameter, where the map $f_1$ is 
``special'' in the sense that $k_1\ge 2$  with $t_{j}^{(1)} \neq 
t_{j'}^{(1)}$ for some $j,j' \in \{1,\ldots,k_1\}$ and $\theta$ is non-real; 
we can further assume that $\Im(\theta)>0$ by symmetry  (by changing the 
coordinates via $z \mapsto \ov{z}$ if needed).

\begin{thm} \label{thm:Fourier-precise} Let $I = \{1,\ldots,N\}$, and let $\PP$ 
be a fully supported Bernoulli measure on $\Omega=I^\N$. Fix $1 < a < b< 
\infty$, and 
	let $(\beta_\kappa)_{\kappa\in I}$ be fixed  complex numbers, with 
	$\beta_1=1$, such that
	\be \label{beta:bound}
	\frac{\log a}{\log b + \pi} \le |\beta_\kappa| \le \frac{\log b + \pi}{\log 
	a}\,,\ \kappa\in I.
	\ee
	Fix a compact set
	\be \label{H-def}
	H= A_{a,b,\eta} \times [p_{\min},p_{\max}] \subset \{z: |z|>1\} \times 
	(0,1),
	\ee
	and fix also $\alpha>0$.
	
	Then there exist a Borel set $\mathcal{G}_{H,\alpha}\subset \Omega\times 
	A_{a,b,\eta}$ and a  constant $\sig=\sig(H,\alpha)>0$ such that the 
	following hold:
	\begin{enumerate}
		\item[{\rm (i)}]
		For $\PP$-almost all $\om$,
		\be \label{cond:alpha}
		\dim_H\{\theta\in A_{a,b,\eta}: (\om,\theta)\notin 
		\mathcal{G}_{H,\alpha} \} \le \alpha.
		\ee
		\item[{\rm (ii)}] If $(\om,\theta)\in\mathcal{G}_{H,\alpha}$, and 
		$\Sigma=((\Phi^{(i)})_{i \in I},(p_i)_{i \in I},\PP)$ is a 
		non-degenerate model such that:
		\begin{itemize}
			\item $k_1\ge 2,  t_{j}^{(1)} \neq t_{j'}^{(1)}$  for some 
			$j,j' \in \{1,\ldots,k_1\}$, and $p_j^{(1)}\in 
			[p_{\min},p_{\max}]$ for all  $j\in\{1,\ldots,k_1\}$,
			\item for each $\kappa \in I$, the mappings $f_j^{(\kappa)}$ have 
			contractions $\lam_\kappa = \theta^{-\beta_\kappa}= 
			e^{-\beta_\kappa \log \theta}$, where 
			$\log$ is the principal branch of logarithm, taking values in 
			$\R + (-\pi, \pi]i$ and continuous on $\C\setminus (-\infty,0]$, and all 
			contraction coefficients satisfy $a \le |\lam_\kappa|^{-1} \le b$, 
			$\kappa=1,\ldots,N$,
		\end{itemize}
		then $\eta^\pom\in\mathcal{D}_2(\sig)$.
	\end{enumerate}
\end{thm}

\begin{rem}
	Note that for any complex numbers $\theta,\theta_\kappa$, with $a < 
	|\theta|, |\theta_\kappa|< b$, the numbers $\beta_\kappa:= \frac{\log 
	\theta_\kappa}{\log \theta}$,
	where  $\log$ is the principal branch of logarithm, satisfy 
	\eqref{beta:bound} and
	$\theta_\kappa = \theta^{\beta_\kappa}$.
\end{rem}

Let us explain first how Theorem \ref{thm:fourier main} follows from Theorem \ref{thm:Fourier-precise} above. Clearly, letting
\[a\searrow 0,\ b\nearrow \infty,\ \eta \searrow 0,\ p_{\min} \searrow 0,\ p_{\max} \nearrow 1,\ \alpha \searrow 0\]
and taking a countable union of corresponding exceptional sets $\Gk_{H,\alpha} \cup \left(\Om \times \R\right)$, we obtain the conclusion of Theorem \ref{thm:fourier main} under additional assumption $\beta_1 = 1$, for models such that $t_{j}^{(1)} \neq t_{j'}^{(1)}$  for some $j,j' \in \{1,\ldots,k_1\}$. For the general case, note that for any given $c \in \C \setminus \{0\}$ the statement of Theorem \ref{thm:fourier main} is invariant under replacing simultaneously $\beta_i$ with $c\beta_i$ (by performing the change of coordinates $\lam \to \lam^{1/c}$ in the parameter space $\D$). We can therefore omit the assumption $\beta_1 = 1$ and by relabelling the elements of $I$ if necessary, we can consider all non-degenerate models, hence the full statement of Theorem \ref{thm:fourier main} is recovered.

\begin{proof}[Proof of Theorem \ref{thm:Fourier-precise}]
	Let $\Sigma$ be a non-degenerate model as in the second part of the 
	statement.                                                                  
	
	Without loss of generality, we may assume that  $t_1^{(1)}\neq t_2^{(1)}$. 
	Since translating and 
	scaling a measure does not change whether the measure is in 
	$\mathcal{D}_2(\sigma)$, we may assume also that $t_1^{(1)}=0$, and 
	$t_2^{(1)}=1$. We may also assume that $N\ge 2$, otherwise this is a 
	standard homogeneous self-similar measure  and the result follows from 
	\cite[Theorem D]{SS16}.
	
	We write $\eta^\pom_\theta$ for the measures generated by the model 
	$\Sigma$; we keep in mind that there is also a dependence on the 
	probabilities and the translations. We write $\lam_i$ for the contraction 
	of the IFS corresponding to the symbol $i$, and recall that  $\lam_i = 
	\theta^{-\beta_i}$ by assumption. By definition, $\eta^{(\om)}_\theta$ is 
	the distribution of a sum of independent random variables, hence it can be 
	represented as an infinite convolution measure
	\[
	\eta^{(\om)}_\lam = \Conv_{n\in\N} \left(\sum_{j=1}^{k_{\om_n}} p_j 
	^{(\om_n)} \delta_{\lam_{\om_1}\cdots \lam_{\om_{n-1}} 
	t_j^{(\om_n)}}\right),
	\]
	where $\delta_x$ is Dirac's delta centered at $x$.
	We thus have
	\begin{eqnarray}
	\left|\widehat{\eta_\lam^{(\om)}}(\xi)\right| & = & \prod_{n\in \N} \left| 
	\sum_{j=1}^{k_{\om_n}} p_j ^{(\om_n)} e^{2\pi i \Re(\lam_{\om_1}\cdots 
	\lam_{\om_{n-1}} t_j^{(\om_n)}\ov\xi)}\right| \nonumber \\
	& \le & \prod_{n:\,\om_n=1} \left( \left| p_1^{(1)} + p_2^{(1)} e^{2\pi i 
	\Re(\lam_{\om_1}\cdots \lam_{\om_{n-1}}  \ov\xi)}\right| + \Bigl(1 - 
	p_1^{(1)}-p_2^{(1)}\Bigr)\right) \nonumber \\
	& \le & \prod_{n:\,\om_n=1} \left( 1 - c_1 \|\Re(\lam_{\om_1}\cdots 
	\lam_{\om_{n-1}} \ov\xi)\|^2\right), \label{ineq1}
	\end{eqnarray}
	where $\|x\|$ denotes the distance from $x$ to the nearest integer and 
	$c_1>0$ is a constant depending only on $[p_{\min},p_{\max}]$.
	
	We will impose a number of conditions on $\om$ which hold $\PP$-a.e. Let 
	$\Om_1$ be the set of $\om\in \Om$ which contain infinitely many words 
	$1^52= 111112$. 
	(The symbol 2 may be replaced by any other $j\ne 1$;
	the only reason for including it is to make sure that two occurrences of 
	$1^52$ cannot overlap.) It is clear that $\PP(\Om_1) = 1$. For $\om \in 
	\Om_1$ let
	\[
	\om = W_1 W_2 W_3\ldots,\ \ \mbox{with}\ \ W_i=W'_i1^52,
	\]
	where $W'_i$ are words in the alphabet $\{1,\ldots,N\}$ not containing 
	$1^52$.
	
	For a finite word $v=v_1\ldots v_n$ let $\lam_v = \lam_{v_1\ldots v_n}:= 
	\lam_{v_1}\cdots \lam_{v_n}$. Then (\ref{ineq1}) implies
	\be \label{ineq2}
	|\widehat{\eta_\lam^{(\om)}}(\xi)| \le \prod_{n=1}^\infty \prod_{j=0}^4 
	\left(1- c_1 \left\|\Re(\lam_{W_1\ldots W_{n-1} W'_n}\cdot 
	\theta^{-j}\ov\xi)\right\|^2\right),
	\ee
	with the convention that the word $W_1\ldots W_{n-1} W'_n$ equals 
	$W'_1$ when $n=1$. 
		Notice that we take into account only the occurrences of 1's which 
		appear in blocks of $1^52$; the reason is that in this case we will
		be able to use a technical lemma from \cite{SS16}.

	Denote $\theta_v = \lam_v^{-1}$, and let
	\[
	\Theta^{(n)}:= (\lam_{W_1\ldots  W_n})^{-1} = \theta_{W_1\cdots W_n}, 
	\Theta^{(0)}=1.
	\]
	Now suppose $|\Theta^{(M)}| \le |\xi|< |\Theta^{(M+1)}|$ for some $M\in \N$ 
	(this depends on $\theta$, and we will keep this in mind), and let $\tau = 
	\ov\xi/\Theta^{(M)}$.
	Denote
	\[
	\Theta^{(M)}_{n}:= \frac{\Theta^{(M)}}{\Theta^{(M-n)}} = 
	\theta_{W_{M}\ldots W_{M-n+1}},\ \ n=1,\ldots, M;\ \ \Theta^{(M)}_0= 1,
	\]
	so that $\Theta^{(M)}_1 = \theta_{W_M}$, $\Theta^{(M)}_2 = \theta_{W_M 
	W_{M-1}}$, etc.
	Then (\ref{ineq2}) yields
	\begin{eqnarray} \nonumber
	|\widehat{\eta_\lam^{(\om)}}(\xi)|  & \le &  \prod_{n=1}^M \prod_{j=0}^4 \left(1- c_1 
	\left\|\Re(\lam_{W_1\ldots W_{n-1} W'_n}\cdot 
	\theta^{-j}\ov\xi)\right\|^2\right) \\
	& = & \prod_{n=1}^M \prod_{j=0}^4 \left(1- c_1 
	\left\|\Re(\lam_{W_1\ldots W_{n-1} W_n}\cdot \lam_1^{-5}\lam_2^{-1} 
	\theta^{-j}\Theta^{(M)}\tau)\right\|^2\right) \nonumber \\
	& = & \prod_{n=1}^M \prod_{j=0}^4 \left(1- c_1 
	\left\|\Re(\theta_{W_{n+1}\ldots W_{M}}\cdot \theta_2 
	\theta^{5-j}\tau)\right\|^2\right) \nonumber \\
	& = & \prod_{n=1}^M \prod_{j=0}^4 \left(1- c_1 
	\left\|\Re(\Theta^{(M)}_{M-n}\cdot \theta_2 
	\theta^{5-j}\tau)\right\|^2\right) \nonumber \\
	& = & 
	\prod_{n=0}^{M-1} \prod_{j=1}^5 \left( 1 - c_1 
	\left\|\Re(\Theta_n^{(M)}\cdot \theta_2 \theta^j \tau)\right\|^2\right). 
	\label{ineq3}
	\end{eqnarray}
	Note that $|\tau| \in [1, |\theta_{W_{M+1}}|]=[1,|\Theta_1^{(M+1)}|]$.
	
	\subsection{Exceptional set} Now we can define the exceptional set in 
	Theorem~\ref{thm:Fourier-precise}.  Denote $[M] = \{0,\ldots,M-1\}$.
	Fix a compact set $H$ as in \eqref{H-def} with all the parameters, in 
	particular,
	$a,b,\eta>0$. For $\rho,\delta>0$ we first define the ``bad set''
	at scale $M\in \N$ for a fixed $\om\in\Om_1$:
	\be \label{badset}
	E_{\om,M}(\rho,\delta):=
	\left\{\theta \in A_{a,b,\eta}:\! \max_{\tau:\ |\tau| \in [1, 
	|\Theta_1^{(M+1)}|]}  \#\Bigl\{n\in [M]:\,\max_{1 \le j \le 
	5}\|\Re(\Theta_n^{(M)}\cdot \theta_2 \theta^j  \tau)\| 
	\ge \rho \Bigr\} < \delta M\right\}
	\ee
	Note that this is a random set, since $\Theta_n^{(M)}$ depend on $\om$. 
	Then consider
	\be \label{set:excep}
	\Ek_\om(\rho,\delta) = \limsup_{M\to \infty} E_{\om,M}(\rho,\delta) = 
	\bigcap_{M_0=1}^\infty \bigcup_{M = M_0}^\infty E_{\om,M}(\rho,\delta).
	\ee 
	Theorem~\ref{thm:Fourier-precise} will follow from the next two 
	propositions.
	
	\begin{prop} \label{prop:decay}
		There exists a set $\Om_2\subset \Om_1$ of full $\PP$-measure, such 
		that for any $\rho,\delta>0$ and $\om\in\Om_2$, 
		we have $\eta^{(\om)} \in \Dk_2(\sig)$ whenever $\theta \in 
		A_{a,b,\eta}\setminus \Ek_\om(\rho,\delta)$, 
		where $\sig>0$ depends  on $\rho,\delta$, and $H$, but not on $\om$.
	\end{prop}
	
	\begin{prop} \label{prop:Four} Fix a compact set $H$ as in \eqref{H-def} 
	and $\alpha>0$.
		There exist positive constants $\rho = \rho(H,\alpha)$ and 
		$\delta=\delta(H,\alpha)$  such that, for $\PP$-a.e.\ $\om\in \Om_2$, 
		there exists $M_0=M_0(\om,H)$ such that for any $M\ge M_0$ and any 
		model satisfying the conditions of (ii) in 
		Theorem~\ref{thm:Fourier-precise},
		the ``bad set'' $E_{\om,M}(\rho,\delta)$ at scale $M$ defined in 
		\eqref{badset}
		can be covered by $O_H(a^{\alpha M} )$ balls of radius $O_H(a^{-M})$. 
	\end{prop}
	
	Let us define explicitly the set $\Om_2$, which will be needed later. By 
	the Law of Large Numbers there exists $ c_2>0$ such that for $\PP$-a.e.\ 
	$\om$,
	\be \label{Birk}
	\left|\{n\in\{1,\ldots,L\}:\,\om[n,n+5] = 111112\}\right| \ge  c_2L
	\ee
	for all  $L$ sufficiently large (depending on $\om$). (In fact, one can 
	take any  $c_2 \in (0, \PP([111112]))$.) For such $\om$ we have
	\be \label{Birk2}
	|W_1\ldots W_{M+1}| \le  c_2^{-1}(M+1),\ \ \mbox{for}\ M\ge M_0(\om).
	\ee
	Now we define
	\be \label{def:Om2}
	\Om_2 = \{\om\in \Om_1:\mbox{\ \eqref{Birk}\ holds  for all $L$ sufficiently 
	large}\}.
	\ee
	
	\begin{proof}[Proof of Theorem~\ref{thm:Fourier-precise} assuming 
	Propositions \eqref{prop:decay} and \eqref{prop:Four}] Let $\wt\Om$ be the 
	set of full $\PP$-measure from Proposition~\ref{prop:Four}, and let
		$$
		\Gk_{H,\alpha} := \Bigl\{(\om,\theta):\ \om \in \wt\Om,\ \theta \in 
		A_{a,b,\eta} \setminus \Ek_\om (\rho,\delta)\Bigr\}.
		$$
		The dimension estimate \eqref{cond:alpha} is immediate from 
		Proposition~\ref{prop:Four}, and uniform Fourier decay of 
		$\eta^{(\om)}$ for $(\om,\theta)\in \Gk_{H,\alpha}$ follows from 
		Proposition~\ref{prop:decay}.
		
		Observe that $\Gk_{H,\alpha}$ is a Borel set: The numbers 
		$\Theta_n^{(M)}$ depend on $\om$ and $\theta$ in a Borel manner, and 
		because $\|\cdot\|$ is continuous, the set
		\eqref{badset} may be defined in terms maxima over complex rational 
		$\tau$, rather than all $\tau$. \end{proof}
	
	\subsection{Proof of Proposition~\ref{prop:decay}} If $\theta \in 
	A_{a,b,\eta}\setminus \Ek_\om(\rho,\delta)$, then $\theta\notin 
	E_{\om,M}(\rho,\delta)$ for all $M\ge M_1 = M_1(\om,\theta)$.
	We can also make sure that $M_1(\om,\theta) > M_0(\om)$, where $M_0(\om)$ 
	is from \eqref{def:Om2}.
	Then for $\xi\in \C$ with $|\xi| \ge |\Theta^{(M_1)}|$ choose $M\ge M_1$ 
	such that $|\xi| \in [|\Theta^{(M)}|,|\Theta^{(M+1)}|)$, and obtain from 
	\eqref{ineq3} and \eqref{badset} that
	\be \label{ineq4}
	|\widehat{\eta_\lam^{(\om)}}(\xi)| \le (1-c_1 \rho^2)^{\delta M} .
	\ee
	By assumption, $|\theta_\kappa|\le b$ for all $\kappa$, hence, in view of 
	\eqref{Birk2},
	$$
	|\xi| \le |\Theta^{(M+1)} |\le b^{c_2^{-1}(M+1)},
	$$ 
	and (\ref{ineq4}) yields
	\[
	|\widehat{\eta_\lam^{(\om)}}(\xi)|  = O_\om(|\xi|^{-\sigma}),\ \ 
	\mbox{where}\ \ \sigma = -\frac{\delta  c_2 \log(1-c_1 \rho^2)}{\log b}>0,
	\]
	as desired. Note that $\sig$ depends on $H$, but not on $\om$ --- just on 
	the model (and $\rho, \delta$).
	\qed

	\subsection{Proof of Proposition~\ref{prop:Four}}
	Fix a complex number $\theta\in A_{a,b,\eta}$, a symbolic sequence $\om\in 
	\Om_1$, and $M\in \N$. We follow the general scheme of the 
	``Erd\H{o}s-Kahane's argument'' (see e.g. \cite[Section 6]{Sixty} for 
	an exposition in the basic case of Bernoulli convolutions). For 
	$n=0,\ldots, M-1$, and $j=1,\ldots,5$, let
	\be \label{eqa1}
	\Re(\Theta_n^{(M)}\cdot \theta_2\theta^j \tau) = K^{(M)}_{n,j} + 
	\eps_{n,j}^{(M)},\ \ \mbox{where}\ K^{(M)}_{n,j}\in \Z,\ |\eps_{n,j}^{(M)}| 
	\le 1/2,
	\ee
	so that $\left\|\Re(\Theta_n^{(M)}\cdot \theta_2\theta^j  \tau)\right\| = 
	\eps_{n,j}^{(M)}$. 
	We will next drop the superscript $(M)$ from the notation of $K_{n,j}$ and 
	$\eps_{n,j}$, but will keep the dependence in mind. These numbers also 
	depend on $\theta$, $\tau$, and $\om$. Observe that
	\be \label{eqa111}
	\frac{\Theta_{n+1}^{(M)}}{\Theta_n^{(M)}}= \theta_{W_{M-n}} = 
	\theta^{\beta(W_{M-n})},\ \ \ \mbox{where}\ \ \beta(v_1\ldots v_\ell):= 
	\sum_{j=1}^{\ell} \beta_{v_j}.
	\ee
	
	\medskip
	
	We will need a technical lemma from \cite{SS16}. 
	Our goal is to recover $\theta$ and $z_0$ from the system of equations
	\begin{eqnarray*}
		& & \Re(z_0) = x_0  \\
		& &\Re(\theta z_0) = x_1  \\
		&  &\Re(\theta^2 z_0) = x_2  \\
		& & \Re(\theta^3 z_0) =x_3.
	\end{eqnarray*}
	Recall that $a,b,\eta$ are fixed and for $R_0>0$ let
	$$
	V_{R_0} =\left\{\bx = (x_0,x_1,x_2,x_3):\ x_j = \Re(\theta^j z_0),\ 0\le 
	j\le 3,\ |z_0|\ge R_0,\ \theta\in A_{a,b,\eta}\right\}.
	$$
	Denote by $\Nk_\eps (V)$ the $\eps$-neighborhood of $V$ in the 
	$\ell^\infty$ metric.

	\begin{lemma}[{\cite[Lemma 2.2]{SS16}}] \label{lem-tech}
		{\em There exist $R_0>0$ and $C_3>0$ depending only on $a,b,\eta$ such 
		that
			there are continuously differentiable functions 
			$$F:\,\Nk_1(V_{R_0})\to \{\theta:\ |\theta|>1, \ \Im(\theta)>0\}\ \ 
			\mbox{and}\ \  G:\,\Nk_1(V_{R_0})\to \R,$$ such that $\theta = 
			F(\bx)$ and $y_3 = G(\bx)$ are the unique solutions to 
			\be \label{equs}
			 x_j = \Re(\theta^{j-3} (x_3+iy_3)),\ \ \ 0\le j \le 2,\ \bx \in 
			V_{R_0}. 
			\ee
			Moreover,
			$$
			\left|\frac{\partial G}{\partial x_j}\right| \le C_3,\ \ \ 0\le 
			j\le 3,\ \ \ \mbox{on}\ \ \Nk_1(V_{R_0}).
			$$
		}
	\end{lemma}
	
	\medskip
	
	All our constants will depend only on $H$; in particular,
	$C_3 = C_3(H)$.
	Denote
	\be \label{Ynj}
	Y_{n,j}=\Im(\Theta_n^{(M)}\cdot \theta_2\theta^j \tau),\ n=0, 1,\ldots, 
	M-1,\ j=1, \ldots, 5.
	\ee
	Note that $|\Theta_n^{(M)}| \ge |\theta|^{5n} \ge a^{5n}$, so we can apply 
	Lemma~\ref{lem-tech} with $z_0 = \Theta_n^{(M)}\cdot \theta^j \theta_2 
	\tau$, $j=1,2$, for $M\ge n\ge n_1 = n_1(H)$ such that $a^{5n_1} \ge 
	R_0$, to obtain
	$$
	Y_{n,j+3} = G\bigl(K_{n,j} + \eps_{n,j}, \ldots,K_{n,j+3} + 
	\eps_{n,j+3}\bigr),\ \ j=1,2.
	$$
	Let
	\be \label{def-G}
	\wtil{Y}_{n,j+3}  = G\bigl(K_{n,j} , \ldots,K_{n,j+3} \bigr),\ \ j=1,2,\ \ 
	\ n\ge n_1(H),
	\ee
	which is well-defined, since $|\eps_{n,j}|\le 1/2<1$, so we are within 
	$\Nk_1(V_{R_0})$, the domain of $G$. Moreover, writing $Y_{n,j+3} - 
	\wtil{Y}_{n,j+3}$ as a sum of path integrals of the
	partial derivatives of $G$, we obtain from
	Lemma~\ref{lem-tech} that
	\be \label{eq2}
	|Y_{n,j+3} - \wtil{Y}_{n,j+3}| \le 4 C_3\cdot 
	\max\{|\eps_{n,j}|,\ldots,|\eps_{n,j+3}|\},\ \ j=1,2,\ n\ge n_1(H)
	\ee
	(the factor $4$ comes from the estimate $\|\bx\|_1 \le 4\|\bx\|_\infty$ 
	on $\R^4$). Then, of course,
	\be \label{eq20}
	|Y_{n,j+3} - \wtil{Y}_{n,j+3}| \le 2 C_3,\ \ j=1,2,\ n\ge n_1(H).
	\ee

	The following is a modification of \cite[Lemma 2.3]{SS16}.
	
	\begin{lem} \label{lem-theta}
		There exist $C_4=C_4(H)>0$ and $n_2=n_2(H)\in \N$, such that
		\be \label{est-theta}
		\left|\theta - \frac{K_{n,5} + i \wtil{Y}_{n,5}}{K_{n,4} + 
		i\wtil{Y}_{n,4}}\right| \le C_4\cdot |\Theta_n^{(M)}\tau|^{-1}  
		\max\{|\eps_{n,1}|,\ldots,|\eps_{n,5}|\}\ \ \ \mbox{for}\ M\ge n\ge 
		n_2(H).
		\ee
	\end{lem}

	\begin{proof} By \eqref{eqa1} and \eqref{Ynj},
		\be \label{trick1}
		K_{n,j} + \eps_{n,j} + iY_{n,j} = 
		\Theta_n^{(M)}\cdot\theta_2\theta^j\tau,\ j=1,\ldots,5,
		\ee
		hence 
		$$
		|K_{n,j} + \eps_{n,j} + iY_{n,j}| \ge |\theta|^{5n+j}\cdot |\tau|\ge 
		|\theta|^{5n} \ge a^{5n},
		$$
		and then for all $n\ge 1$,
		\be \label{trick11}
		\frac{|K_{n,j+1}  + iY_{n,j+1}|}{|K_{n,j}  + iY_{n,j}|} \le 
		\frac{|\Theta_n^{(M)}\cdot\theta^{j+1}\theta_2\tau|+\half}{|\Theta_n^{(M)}\cdot\theta^{j}\theta_2\tau|-\half}
		 \le 3|\theta|\le 3b,\ \ j=1,\ldots,4.
		\ee
		In view of \eqref{eq20},
		\be \label{trick0}
		|K_{n,4}  + i\wt Y_{n,4}| \ge |\Theta_n^{(M)}\theta_2\theta^4\tau| - 
		\textstyle{\half}-2C_3 \ge \textstyle{\half} |\Theta_n^{(M)}\tau|\ \ 
		\mbox{for}\ n\ge n_2(H),
		\ee
		where $n=n_2(H) \ge n_1(H)$ is such that $\half + 2C_3 < \half a^{5n}$.
		
		We begin by estimating, for $j=1,\ldots,4$ and all $n\ge 1$, \smallskip
		\begin{eqnarray*}
			\left|\theta - \frac{K_{n,j+1} + i Y_{n,j+1}}{K_{n,j} + 
			iY_{n,j}}\right| & = & \left| \frac{K_{n,j+1} + \eps_{n,j+1} + i 
			Y_{n,j+1}}{K_{n,j} +\eps_{n,j} + iY_{n,j}} - 
			\frac{K_{n,j+1} + i Y_{n,j+1}}{K_{n,j} + iY_{n,j}}\right| \\[1.2ex]
			&\le  & \frac{|\eps_{n,j+1}|}{|K_{n,j} +\eps_{n,j} + iY_{n,j}|} + 
			\frac{|\eps_{n,j}|}{ |K_{n,j} +\eps_{n,j} + iY_{n,j}|}\cdot 
			\frac{|K_{n,j+1} + iY_{n,j+1}|}{|K_{n,j} + iY_{n,j}|} \\[1.3ex]
			& \le &  |\Theta_n^{(M)}\tau|^{-1}\cdot(|\eps_{n, j+1}| + 
			|\eps_{n,  j}|\cdot 3b)\\[1.1ex]
			& \le &(1+3b) |\Theta_n^{(M)}\tau|^{-1}\cdot\max\{|\eps_{n,j}|, 
			|\eps_{n,j+1}|\}.
		\end{eqnarray*}
		On the other hand, for $n\ge n_2(H)$ we have, using \eqref{eq2}, 
		\eqref{trick0}, and \eqref{trick11},
		\begin{eqnarray*}
			\left| \frac{K_{n,5} + i Y_{n,5}}{K_{n,4} + iY_{n,4}} - 
			\frac{K_{n,5} + i \wtil{Y}_{n,5}}{K_{n,4} + i\wtil{Y}_{n,4}}\right|
			& \le & \frac{|\wtil{Y}_{n,5}-Y_{n,5}|}{|K_{n,4} + i 
			\wtil{Y}_{n,4}|} + \frac{|\wtil{Y}_{n,4}-Y_{n,4}| \cdot 
			|K_{n,5}+iY_{n,5}|}{|K_{n,4}+iY_{n,4}|\cdot|K_{n,4} + 
			i\wtil{Y}_{n,4}|} \\[1.2ex] 
			& \le & 8C_3\cdot (1+3b)\cdot  
			|\Theta_n^{(M)}\tau|^{-1}\cdot\max\{|\eps_{n,1}|,\ldots,|\eps_{n,5}|\},
		\end{eqnarray*}
		The claim of the lemma follows, with $C_4 = (1+3b) (1+8C_3)$, by 
		combining the last two inequalities.
	\end{proof}
	
	In view of \eqref{eqa111} and \eqref{trick1},
	\be \label{eq3}
	(K_{n+1,j} + \eps_{n+1,j}) + i Y_{n+1,j} = 
	\theta^{\beta(W_{M-n})+j-5}\bigl( (K_{n,5} + \eps_{n,5}) + i 
	Y_{n,5}\bigr),\ \ j=1,\ldots,5.
	\ee
	Thus, Lemma~\ref{lem-theta} suggests the approximation that can be used to 
	run a step in the Erd\H{o}s-Kahane's argument:
	$$
	K_{n+1,j} \approx \Re\left[{\left(\frac{K_{n,5} + i \wtil{Y}_{n,5}}{K_{n,4} 
	+ i\wtil{Y}_{n,4}}\right)}^{\beta(W_{M-n})+j-5}\cdot\bigl( K_{n,5}  + i 
	\wtil{Y}_{n,5}\bigr)\right],\ \ j=1,\ldots,5.
	$$
	However, in order to make this precise, careful estimates are required.
	
	Let us fix $j\in \{1,\ldots,5\}$ and denote
	$$
	s:= \beta(W_{M-n})+j-5.
	$$
	Note that
		\be \label{eqs}
		|\theta^s|\ge 1, \text{ as } |\theta^s| = 
		|\theta_{W_{M-n}}||\theta|^{j-5} \geq |\theta^5\theta_2||\theta|^{j-5}.
		\ee
	
	In view of \eqref{eq3}, for all $n\ge 1$,
	$$
	\bigl| K_{n+1,j}  - \Re\left[\theta^s(K_{n,5}  + i Y_{n,5})\right]\bigr| 
	\le (1 + |\theta^s|) \max\{|\eps_{n,5}|, 
	|\eps_{n+1,1}|,\ldots,|\eps_{n+1,5}|\},
	$$
	and then,  by \eqref{eq2} and \eqref{eqs}, for $n\ge n_1(H)$,
	\be \label{eq31}
	\Bigl| K_{n+1,j}  - \Re\bigl[\theta^s(K_{n,5}  + i \wt Y_{n,5})\bigr]\Bigr| 
	\le (4C_3 + 2)\cdot |\theta^s| \cdot \max\{|\eps_{n,k}|, |\eps_{n+1,k}|; 
	k=1,\ldots,5\}.
	\ee
	Observe that 
	\be \label{eq32}
	|\theta^s| \le |\theta_{W_{M-n}}| \le b^{|W_{M-n}|},
	\ee
	by the a priori bounds on contraction ratios.
	
	It is convenient to introduce notation for the approximation of $\theta$ 
	from \eqref{est-theta}:
	\be \label{def-Psi}
	\Psi(K_{n,1},\ldots,K_{n,5}):= \frac{K_{n,5} + i \wtil{Y}_{n,5}}{K_{n,4} + 
	i\wtil{Y}_{n,4}} = \frac{K_{n,5} + i G(K_{n,2},\ldots,K_{n,5})}{K_{n,4} + 
	iG(K_{n,1},\ldots,K_{n,4})}\,.
	\ee
	Since $\theta\in A_{a,b,\eta}$, in view of \eqref{est-theta}, for $n\ge 
	n_3(H)\ge n_2(H)$,  the straight line segment from
	$\theta$ to $\Psi(K_{n,1},\ldots,K_{n,5})|$ will be contained in $\wt 
	A_{a,b,\eta}:= A_{\frac{1+a}{2},\, b+1,\,\frac{\eta}{2}}$.
	We next need to estimate $\left|\theta^s - 
	\Psi(K_{n,1},\ldots,K_{n,5})^s\right|$ from above. We can apply the 
	inequality:
	\be \label{trick10}
	|f(z_2)-f(z_1)| \le |z_2-z_1|\cdot \max\{|f'(z)|: z\in \wt A_{a,b,\eta}\},
	\ee
	where $z_1 = \theta,\ z_2 = \Psi(K_{n,0},\ldots,K_{n,4})$, and $f(z) = 
	z^s$, with $s = {\beta(W_{M-n})+j-5}$.  Note that
	$$
	|(z^s)'| = \frac{|s| |z^s|}{|z|} =\frac{|s|}{|z|} \cdot \exp{[\Re(s\log 
	z)]}.
	$$
	Recall that we are using the principal branch of the logarithm, so
	$$
	\Re(s\log z) = \Re (s)\cdot \log|z| - \Im (s)\cdot \arg(z) \le |s|(\log|z| 
	+ \pi),
	$$
	Thus, for $z\in \wt A_{a,b,\eta}$, using that $|s| \le e^{|s|} \le 
	e^{|\beta(W_{M-n})|+5}$ we get
	\begin{eqnarray*}
		|(z^s)'| & \le & |z|^{-1}\cdot |s|\cdot {(e^\pi |z|)}^{|s|}\cdot \\
		& \le & |z|^{-1} \cdot {(e^{\pi+1} |z|)}^{|s|}\\
		& \le & 2(1+a)^{-1} {\bigl(e^{\pi+1}(b+1)\bigr)}^{\max_\kappa 
		|\beta_\kappa|\cdot |W_{M-n}|+5} =: C_5\cdot C_6^{|W_{M-n}|},
	\end{eqnarray*}
	where $C_5=C_5(H)>0$ and $C_6=C_6(H)>1$.
	Combining this with \eqref{trick10} and Lemma~\ref{lem-theta} yields
	\begin{eqnarray} 
	&  & \bigl|\theta^s - \Psi(K_{n,1},\ldots,K_{n,5})^s\bigr| \nonumber \\
	& \le & C_4C_5\cdot C_6^{|W_{M-n}|} \cdot|\Theta_n^{(M)}\tau|^{-1} \cdot 
	\max\{|\eps_{n,1}|,\ldots,|\eps_{n,5}|\},\ \ \mbox{for}\ n\ge n_3(H). 
	\label{trick31}
	\end{eqnarray}
	Similarly to \eqref{trick0}, we have
	$$
	|K_{n,5}  + i\wt Y_{n,5}| \le |\Theta_n^{(M)}\theta_2\theta^5\tau| + 
	\textstyle{\half} + 2C_3 \le 2b^6 |\Theta_n^{(M)}\tau|,\ \ \mbox{for}\ n\ge 
	n_2(H).
	$$
	Finally, together with \eqref{eq31}, \eqref{eq32} and 
	\eqref{trick31}, recalling the definition of $s$, this yields for some $C_7 
	= C_7(H)>0$ and $C_8= C_8(H)>1$:
	\begin{eqnarray} \nonumber
	& & \left|K_{n+1,j} - \Re\Bigl[{\Bigl(\frac{K_{n,5} + i 
	\wtil{Y}_{n,5}}{K_{n,4} + 
	i\wtil{Y}_{n,4}}\Bigr)}^{\beta(W_{M-n})+j-5}\cdot\bigl( K_{n,5}  + i 
	\wtil{Y}_{n,5}\bigr)\Bigr]\right| \\[1.1ex]
	& \le & C_7\cdot C_8^{|W_{M-n}|}\cdot \max\{|\eps_{n,k}|,|\eps_{n+1,k}|:\, 
	k=1,\ldots,5\},\ \ j=1,\ldots,5,\ \ \ n\ge n_3(H).  \label{eq4}
	\end{eqnarray}
	
	Now we can state the main lemma in the Erd\H{o}s-Kahane method. Denote
	\be \label{def-Bro}
	B_n :=  C_7\cdot C_8^{|W_{M-n}|},\ \ \rho_n = (2B_n)^{-1}.
	\ee
	Note that these numbers depend on $\om$ and $M$, and recall that the 
	$K_{n,j}$ depend, additionally, on $\theta$ and $\tau$.
	
	\begin{lem} \label{lem-step}
		Fix $M\in\N$ and $\om$ such that $\om\in \Om_1$. Then for $n\ge n_3(H)$ 
		the following holds:
		
		{\rm (i)}  Fix $\theta\in A_{a,b,\eta}$ and $\tau\in \C$, with
			$|\tau| \in [1, \theta_{W_{M+1}}]$. Then, given $K_{n,1},\ldots, 
			K_{n,5}$, there are at most $2B_n+1$ possibilities for each of 
			$K_{n+1,1}\ldots,K_{n+1,5}$. 
		
		{\rm (ii)} if $\max\{|\eps_{n,1}|,\ldots,|\eps_{n,5}|; 
		|\eps_{n+1,1}|,\ldots,|\eps_{n+1,5}|\}<\rho_n$, 
		then each of $K_{n+1,1}\ldots,K_{n+1,5}$ is uniquely determined by 
		$K_{n,1},\ldots, K_{n,5}$, independent of $\theta\in A_{a,b,\eta}$ 
		and $\tau\in \C$, with $|\tau| \in [1, \theta_{W_{M+1}}]$.
		\end{lem}
	
	The claims of the lemma are immediate from (\ref{eq4}), using the fact 
	that  $K_{n+1,1}\ldots,K_{n+1,5}$ are integers and keeping in mind  the 
	definition \eqref{def-G} (showing that $\wtil{Y}_{n,4}$ and 
	$\wtil{Y}_{n,5}$ depend only on $K_{n,1}, \ldots, K_{n,5}$).

	\subsection{Probabilistic estimates}
	
	To continue with the proof, we will need some probabilistic estimates, 
	which are taken almost verbatim from \cite[Section 5.1]{SSS}. Let $T$ be 
	the left shift on $\Om$ and let
	\[
	X_1 = X_1(\om) = \min\{i\ge 1:\ \om[i,i+5]=111112\},\ \ \ X_{n+1}(\om) = 
	X_1(T^{X_1(\om)+\dots +X_n(\om)}\om),\ \ n\ge 1.
	\]
	Thus, $X_{n+1}(\om)$ is the waiting time between the $n$-th and $(n+1)$-st 
	appearance of the word $1^52$ in the sequence $\om$.
	The process $(X_n)_{n\in \N}$ is defined almost surely; more precisely, on 
	the set $\Om_1\subset \Om$ of full $\PP$-measure.
	Since $\PP$ is Bernoulli, $(X_n)_{n\in \N}$ is an i.i.d.\ sequence of 
	exponential random variables.
	Note that $$X_n = |W_n|, \ n\ge 2,$$ by construction.

	Next we state the property which provides the full measure set appearing in 
	Theorem~\ref{thm:Fourier-precise}.
	
	\begin{lem} \label{lem:proba}
		Consider the process $(X_n)_{n\in \N}$ defined as above. There exists a 
		positive constant $L_1$ such that for $\PP$-a.e.\ $\om$, for any 
		$\varrho>0$ and all $M$ sufficiently large (depending on $\varrho$ and 
		$\om$),
		\be \label{cond2}
		\max\left\{\sum_{n\in \Psi} X_n(\om):\ \Psi \subset \{0,\ldots,M-1\}, \ 
		|\Psi| \le \varrho M\right\} \le L_1\cdot \log(1/\varrho)\cdot \varrho 
		M.
		\ee
	\end{lem}
	
	The proof repeats verbatim the proof of \cite[Lemma 5.2]{SSS}.
	
	\begin{rem}
		Observe that for any given $\wtil{\varrho}>0$, we have for $\PP$-almost 
		all $\om$ and all $n$ sufficiently large (depending on $\om$ and 
		$\wtil\varrho$) that
		\be \label{cond1}
		X_n(\om) \le \wtil{\varrho} n;
		\ee
		Indeed, $\frac{1}{n} \sum_{i=1}^n X_i \to \E[X_1]< \infty$ almost 
		surely, by the Law of Large Numbers, and hence $X_n(\om)/n\to 0$ for 
		$\PP$-a.e.\ $\om$.
	\end{rem}
	
	Denote by $\Om_3\subset \Om_1$ the set of sequences $\om$ such that 
	\eqref{Birk}, \eqref{cond2}, and \eqref{cond1} hold. Thus $\PP(\Om_3) = 1$.
	
	\subsection{Conclusion of the proof} We next define a variant of the 
	exceptional set \eqref{badset} whose dimension is easier to estimate. For 
	$\om\in \Om_3$, $\delta>0$, and $M\in \N$ let
	$$
	\wt E_{\om,M}(\delta):=
	\left\{\theta \in A_{a,b,\eta}:\! \max_{\tau:\ |\tau| \in [1, 
	|\Theta_1^{(M+1)}|]}  \#\Bigl\{n\in [M]:\,\max_{1 \le j \le 
	5}\|\Re(\Theta_n^{(M)}\cdot \theta_2 \theta^j \tau)\| 
	\ge \rho_n\Bigr\} < \delta M\right\},
	$$
	where $\rho_n$ is from \eqref{def-Bro}.
	
	\begin{lem} \label{lem:inclu} Let $\delta>0$.
		For all $\om\in \Om_3$  there exists $M_2(\om,\delta)$ such that for 
		$M\ge M_2(\om)$ holds
		\be \label{eq:inclu}
		\wt E_{\om,M}(2\delta) \supset E_{\om,M}(\delta,\rho),
		\ee
		with 
		$$
		\rho = (2C_7)^{-1} C_8^{-1/c_2\delta},
		$$
		where $c_2>0$ is the constant from \eqref{Birk}.
	\end{lem}
	
	\begin{proof}[Proof of the lemma] In view of \eqref{Birk2}, there are fewer 
	than $\delta M$ integers $n\in [M]$ for which $|W_{M-n}| > 
	(c_2\delta)^{-1}$, for $M$ sufficiently large,
		depending on $\om$ and $\delta$, hence there are fewer than $\delta M$ 
		integers $n\in [M]$ for which $\rho_n \le \rho$, see \eqref{def-Bro}. 
		Let  $\theta \notin \wt E_{\om,M}(2\delta)$.
		Then there exists $\tau$, with $1 \le |\tau| \le |\Theta_1^{(M+1)}|$, 
		such that 
		$$
		\#\Bigl\{n\in [M]:\,\max_{1 \le j \le 5}\|\Re(\Theta_n^{(M)}\cdot 
		\theta_2\theta^j  \tau)\| 
		\ge \rho_n\ge  \rho\Bigr\} \ge \delta M,
		$$
		which means that $\theta\notin E_{\om,M}(\delta,\rho)$.
	\end{proof}
	
	It remains to estimate the number of balls of radius $O_H(a^{-M})$ needed 
	to cover $\wt E_{\om,M}(2\delta)$. Suppose that $\theta \in \wt 
	E_{\om,M}(2\delta)$. Choose appropriate $\tau$
	from the definition of $\wt E$, and find the corresponding $K_{n,j}$ and 
	$\eps_{n,j}$. Given the sequence $K_{n,j}$ for $n=0,\ldots,M-1$ and 
	$j=1,\ldots,5$, Lemma~\ref{lem-theta}, with 
	$n=M$, provides a ball of radius
	\[ C_4 |\Theta^{(M)}_M \tau|^{-1} = C_4 |\theta_{W_M\ldots 
	W_1}\tau|^{-1} \leq C_4a^{-|W_M| \cdots |W_1| } =   O_H(a^{-M})  \]
	to cover $\theta$. It remains to estimate the number of sequences
	$K_{n,j}$ that can arise this way.
	
	Let $\Psi_M$ be the set of $n\in [M-1]$ where we have 
	$\max\{|\eps_{n,j}|,|\eps_{n+1,j}|:\ j=1,\ldots,5\}\ge \rho_n$. By the 
	definition of $\wt E_{\om,M}(2\delta) $, we have $|\Psi_M| \le 4\delta 
	M$.
	By Lemma~\ref{lem-step}, for a fixed $\Psi_M$ the number of possible 
	sequences $\{K_{n,j}:\, n\le M,\ j=1,\ldots,5\}$ is at most
	$$
	\Bk_M := \prod_{n\in \Psi_M} (2B_n+1)^5,
	$$
	times the number of ``beginnings'' $\{K_{n,j}:\, n\le n_3(H),\ 
	j=1,\ldots,5\}$. To estimate the number of choices for $K_{n,j}$, observe 
	that
	$$
	 |K_{n,j}| \le |\Theta_n^{(M)} \theta_2 \theta^j\tau| + 1 \le b^{|W_{M}| 
	\cdots |W_{M-n+1}|+7} |\Theta_1^{(M+1)}| \le b^{|W_{M}| \cdots 
	|W_{M-n+1}|+7 + |W_{M+1}|}. 
	$$
	It follows that the number of sequences $\{K_{n,j}:\, n\le n_3(H),\ 
	j=1,\ldots,5\}$ is bounded above by
	
$${
		\begin{split}
		& \left(2b^{|W_{M}| \cdots |W_{M-n_{3}(H)+1}|+7 + |W_{M+1}|} + 
		1\right)^{n_3(H)} \\
		& \leq \left(3b^{|W_{M}| \cdots |W_{M-n_{3}(H)+1}|+7 + 
		|W_{M+1}|}\right)^{n_3(H)} \\
		& \leq 3^{n_3(H)}b^{7n_3(H)}b^{n^2_3(H)\max \left\{ |W_{M+1-k}|\ :\ 0 
		\leq k \leq n_3(H)  \right\}}.
		\end{split}
	}
	$$

	Let $\wtil\varrho = \wtil\varrho(H,\alpha)>0$ be such that
	$${
		b^{\wtil\varrho \cdot(n^2_3(H))} < a^{\alpha/2}.
	}
	$$
	Then \eqref{cond1} implies that for $\PP$-a.e.\ $\om$, for $M$ sufficiently 
	large, 
	$${ 
		|W_{M+1-k}| \le (M+1)\wtil\varrho,\ k=0,\ldots, n_3(H),
	}
	$$
	so the number of choices of $\{K_{n,j}:\, n\le n_3(H),\ j=1,\ldots,5\}$ is 
	bounded above by $O_H(1)\cdot a^{\alpha M/2}$.
	
	Now we estimate $\Bk_M$. In view of (\ref{cond2}) in Lemma \ref{lem:proba} 
	and (\ref{def-Bro}),
	\[
	\Bk_M \le \exp \Bigl(O_H(1)\cdot  \sum_{n\in \Psi_M} |W_{M-n}| \Bigr) \le 
	\exp\bigl(O_H(1)\cdot L_1\cdot \log(1/4\delta)\cdot \delta M\bigr),
	\]
	for $M$ sufficiently large. Combining everything, we obtain that the number 
	of balls of radius $O_H(a^{-M})$ needed to cover $\wtil{E}_{2\delta,M}$ is 
	not greater than
	\begin{equation}\label{eqEcov}
	a^{\alpha M/2}\cdot \sum \limits_{i\le 4\delta M} \binom{M}{i}\cdot 
	\exp\left(O_H(1)\cdot L_1\cdot \log(1/4\delta)\cdot \delta M\right).
	\end{equation}
	{By Stirling's formula, there exists $C>0$ such that
		\[ \sum \limits_{i\le \delta M} \binom{M}{i} \leq \exp \left( C \delta 
		\log (1/\delta) M \right) \text{ for } \delta < e^{-1} \text{ and all } 
		M > 1. \]
		Therefore, the covering number from \eqref{eqEcov} can be made smaller 
		than $a^{\alpha M}$} by taking appropriate $\delta=\delta(H,\alpha)>0$. 
		This concludes the proof of Proposition~\ref{prop:Four}, and now 
		Theorem~\ref{thm:Fourier-precise} is proved completely.
\end{proof}

\section{Absolute continuity of self-similar measures}

This section is devoted to proving Theorem \ref{thm:main ac}. The proof follows closely the one of \cite[Theorem 1.1]{SSS}, with small adjustments needed in the planar case. We will therefore omit some of the details, which follow exactly as in \cite{SSS}.

Fix distinct translations $t_1, \ldots, t_k \in \C$ and a probability vector $\bp = (p_1, \ldots, p_k)$ with strictly positive entries. Denote $t = (t_1, \ldots, t_k)$. It follows from \cite[Theorem 10]{SimonVago} that the set
\[ \Ak = \left\{ \lam \in \D_*^k : \nu^{\bp}_{\lam, t} \text{ is absolutely continuous} \right\} \]
is $F_\sigma$ (and hence Borel measurable). Fix non-zero complex numbers $1 = \beta_1, \beta_2, \ldots, \beta_k$, and write for short $\nu^{\bp}_\lam = \nu^{\bp}_{(\lam^{\beta_1}, \ldots, \lam^{\beta_k}), t}$ and $s(\lam,\bp) = s((\lam^{\beta_1}, \ldots, \lam^{\beta_k}), \bp)$ for the corresponding similarity dimension (recall \eqref{eq:sim dim}). Note that we do not include $t$ in the notation, as in this section it is fixed. Applying (locally) the coarea formula \cite[Theorem 3.10]{Evans} to the map
\[ (\lam_1, \ldots, \lam_k) \mapsto \left( \frac{\log \lam_2}{\log \lam_1}, \ldots, \frac{\log \lam_k}{\log \lam_1} \right)  \]
on the set $\D_*^k \setminus \Ak$, we see that for Theorem \ref{thm:main ac} it suffices to prove the following
\begin{prop}\label{prop:ac on curves}
For every $\eps > 0$ there exists a set $E^{\bp}_t \subset \D_*$ (depending also on $\beta_i$) of zero $2$-dimensional Lebesgue measure such that if $\lam \in  \D_*\setminus E^{\bp}_t$ satisfies $|\lam|^{\beta_i} < 1$ for $i=1, \ldots, k$ and $s(\lam, \bp) > 2+\eps$, then $\nu^{\bp}_\lam$ is absolutely continuous.
\end{prop}

Proof of Proposition \ref{prop:ac on curves} occupies the rest of this section. Let us present its scheme. First of all, we will introduce a family of models $\Sigma_\lam$ (with similarity dimension arbitrarily close to $s(\lam, \bp)$) over a common probability space $(I^{\N}, \PP)$, such that $\nu^{\bp}_\lam = \int \eta^{\pom}_\lam d\PP(\om)$, where $\eta^{\pom}_\lam$ are the random measures generated by the model $\Sigma_\lam$. Therefore, to prove that a given $\nu^{\bp}_\lam$ is absolutely continuous, it will suffice to prove that $\eta^{\pom}_\lam$ is absolutely continuous for $\PP$-almost every $\om$. To that end, we will use the fact that each measure $\eta^{\pom}_\lam$ is an infinite convolution and introduce new families of models $\Sigma'_\lam, \Sigma''_\lam$ such that the corresponding random measures satisfy $\eta^{\pom}_\lam = (\eta'_\lam)^{\pom} * (\eta''_\lam)^{\pom}$. Applying Theorems \ref{thm:main dim} and \ref{thm:fourier main} to $\Sigma'_\lam$ and $\Sigma''_\lam$, we will conclude that for almost every pair $(\lam, \om)$, the measure $ (\eta'_\lam)^{\pom}$ has power Fourier decay and the measure  $(\eta''_\lam)^{\pom}$ has full Hausdorff dimension. As convolution of two such measures is absolutely continuous on $\C$ (see \cite[Lemma 4.3.(i)]{SS16'} and \cite[Lemma 2.1]{ShmerkinBC} for the proof), this will finish the proof.  This is exactly the same strategy as in \cite{SSS}. The only additional feature is that we need to verify item \ref{it:non-real} of Assumption \ref{assmp:non-deg and rotation} in order to apply Theorem \ref{thm:main dim} for typical $\lam$. This is rather straightforward.

Let us elaborate now on how the models $\Sigma_\lam$ are defined.  Fix $r \in \N$. For $i=1, \ldots, k$ write
\[g_{\lam, i} (z)= \lam^{\beta_i}z + t_i\]
and for $u \in \{1, \ldots, k\}^r$ write $g_{\lam, u} = g_{\lam, u_1} \circ \cdots \circ g_{\lam, u_r}$ and $p_u = p_{u_1} \cdots p_{u_r}$. To construct the model $\Sigma_\lam$, we will gather into a single homogeneous IFS all the compositions $g_{\lam, u}$ with a prescribed number of occurrences of each map $g_{\lam, i}$ from the original system. More precisely, we define
\[  I = I(r) := \left\{ (n_1, \ldots, n_k) \in \{0, \ldots, r \} : \sum \limits_{j=1}^k n_j = r \right\}\]
and a map $\Psi : \{1, \ldots, k\}^r \to I$ given by $\Psi(u) = \left( N_1(u), \ldots, N_k(u) \right)$, where $N_j(u)$ denotes the number of occurrences of the symbol $j$ in the word $u$. For $\vec{n} = (n_1, \ldots, n_k) \in I$ denote $k_{\vec{n}} = |\Psi^{-1}(\vec{n})|$ and enumerate $\Psi^{-1}(\vec{n}) = \left\{ u^{(\vec{n})}_1, \ldots, u^{(\vec{n})}_{k_{\vec{n}}}\right\}$. For $1 \leq j \leq k_{\vec{n}}$ set
\[ f^{(\vec{n})}_{\lam, j} = g_{\lam, u^{(\vec{n})}_j}, \]
\[ \wt{p}^{(\vec{n})}_j = k_{\vec{n}}^{-1} = p_{u^{(\vec{n})}_j} / q_{\vec{n}}, \]
where \[q_{\vec{n}} = \sum \limits_{u \in \Psi^{-1}(\vec{n})} p_u = |\Psi^{-1}(\vec{n})|p_1^{\vec{n}_1} \cdots p_k^{\vec{n}_k}.\]
Note that for a fixed $\vec{n} \in I$, all the maps $f^{(\vec{n})}_{\lam, j}$ have the same linear part $\prod \limits_{j=1}^k\lam^{\vec{n}_j\beta_j} = \lam^{\gamma_{\vec{n}}}$, where $\gamma_{\vec{n}} = \sum \limits_{j=1}^k n_j \beta_j$. Therefore, defining $\Phi^{(\vec{n})}_\lam = \left( f^{(\vec{n})}_{\lam, 1}, \ldots, f^{(\vec{n})}_{\lam, k_{\vec{n}}} \right),\ \wt{p}_{\vec{n}} = \left( \wt{p}^{(\vec{n})}_1, \ldots, \wt{p}^{(\vec{n})}_{k_{\vec{n}}} \right)$, and setting $\PP$ to be the Bernoulli measure on $I^\N$ with the marginal $q = (q_{\vec{n}})_{\vec{n} \in I}$, we obtain a model $\Sigma_\lam = \left( \left( \Phi^{(\vec{n})}_\lam \right)_{\vec{n} \in I}, \left( \wt{p}_{\vec{n}}\right)_{\vec{n} \in I}, \PP \right)$. Note that we suppress the dependence on $r$ in the notation for $\Sigma_\lam$ and that $\PP$ does not depend on $\lam$ (only on the numbers $k, r$). Let $\eta_\lam^{\pom}$ denote the random measure generated by the model $\Sigma_\lam$. It is easy to check directly (see the calculation in \cite[Proof of Lemma 6.2.(iv)]{SSS}) that for every $\lam \in \D_*$ holds
\begin{equation}\label{eq:disintegrate} \nu^{\bp}_{\lam} = \int \limits \eta_\lam^{\pom}d\PP(\om).
\end{equation}
Moreover, by \eqref{eq:sdim Bernoulli} and a direct calculation from \cite[Proof of Lemma 6.2.(v)]{SSS}, we have
\[ \sdim(\Sigma_\lam) = \frac{\sum \limits_{\vec{n} \in I} q_{\vec{n}}  H(\wt{p}_{\vec{n}})    } {-\sum \limits_{\vec{v} \in I} q_{\vec{n}} \log \prod \limits_{j=1}^k\left|\lam^{\vec{n}_j\beta_j}\right|   }  =  \frac{r H(\bp) - H(q)}{ - r \sum \limits_{i=1}^k p_i\log|\lam_i| } = \left(1 - \frac{H(q)}{r H(\bp)} \right)s(\lam, \bp).\]
As the cardinality of $I$ is polynomial in $r$ (for a fixed $k$), we see that $H(q)/r \to 0$ as $r \to \infty$. Therefore, given $\eps > 0$, there exists $r = r(\eps, \bp) \in \N$ such that for every $\lam \in \D_*$ holds
\begin{equation}\label{eq:sdim model bound}
	\sdim(\Sigma_\lam) > (1 - \eps/2)s(\lam, \bp).
\end{equation}
The following elementary observation will be useful to us:
\begin{lem}\label{lem:non-deg}
Let $\lam \in \D_*$ be such that not all the maps $g_{\lam, i}, i = 1, \ldots, k$, share the same fixed point. Then for every $r \geq 2$, the model $\Sigma_\lam$ is non-degenerate. 
Consequently, $\Sigma_\lam$ is non-degenerate for almost every $\lam \in \D_*$.
\end{lem}

\begin{proof}
As $\PP$ is fully supported, it is enough to prove that there exist $\vec{n} \in I$ and $j,j' \in \{1, \ldots, k_{\vec{n}}\}$ such that  $g_{\lam, u^{(\vec{n})}_j} \neq g_{\lam, u^{(\vec{n})}_{j'}}$. It is straightforward to check that for $i_1, i_2 \in \{1, \ldots, k\}$, the equality $g_{\lam, i_1} \circ g_{\lam, i_2} = g_{\lam, i_2} \circ g_{\lam, i_1}$ holds if and only if $g_{\lam, i_1}$ and $g_{\lam, i_2}$ share a fixed point. Therefore, if they do not share a fixed point, then $g_{\lam, i_1} \circ g_{\lam, i_2} \neq g_{\lam, i_2} \circ g_{\lam, i_1}$ and also
\[ \underbrace{g_{\lam, i_1} \circ \cdots \circ g_{\lam, i_1}}_\text{$r-2$ times} \circ g_{\lam, i_1} \circ g_{\lam, i_2} \neq \underbrace{g_{\lam, i_1} \circ \cdots \circ g_{\lam, i_1}}_\text{$r-2$ times} \circ g_{\lam, i_2} \circ g_{\lam, i_1}, \]
hence $\Sigma_\lam$ is non-degenerate for any $r \geq 2$. The maps $g_{\lam, 1}$ and $g_{\lam, 2}$ share a fixed point if and only if $\lam t_2 - \lam^{\beta_2}t_1 + t_1 - t_2 = 0$. As $t_1 \neq t_2$ by assumption, this equation can hold only on a set of $\lam$'s in $\D_*$ of Lebesgue measure zero.
\end{proof}

Let us fix now $r \geq 2$ such that \eqref{eq:sdim model bound} holds. From now on, for simplification, we will use the notation $\Sigma_\lam = \left( \left( \Phi^{(i)}_\lam \right)_{i \in I}, \left( \wt{p}_{i}\right)_{i \in I}, \PP \right)$ with
\[ \Phi^{(i)}_\lam = \left( \lam^{\gamma_i}z + \wt{t}^{(i)}_{\lam,j} \right)_{1 \leq j \leq k_i},\]
where $\gamma_i$ depend on $\beta_1, \ldots, \beta_k$. Recall that for a fixed $\om \in \Om = I^\N$, the measure $\eta^{\pom}$ (corresponding to model $\Sigma_\lam$) is the distribution of the random series
\[ \sum \limits_{n=1}^\infty \left( \prod_{j=1}^{n-1}\lam^{\gamma_j} \right) \wt{t}_{\lam,u_n}^{(\om_n)}, \]
where $u_n \in \{1, \ldots, k_{\om_n} \}$ are chosen independently according to probability vectors $\wt{p}_{\om_n}$. Fix $s \in \N$. If for a given $\om$ we keep $u_n$ distributed as above and define $(\eta'_{\lam})^{(\om)},\ (\eta''_{\lam})^{(\om)}$ to be distributed as
\[ (\eta'_{\lam})^{(\om)} \sim \sum \limits_{n\in \N, s|n } \left( \prod_{j=1}^{n-1}\lam^{\gamma_j} \right) \wt{t}_{\lam,u_n}^{(\om_n)}\ \text{ and }\ (\eta''_{\lam})^{(\om)} \sim \sum \limits_{n\in \N, s \nmid n } \left( \prod_{j=1}^{n-1}\lam^{\gamma_j} \right) \wt{t}_{\lam,u_n}^{(\om_n)},\]
then, by the independence of the terms in the series, we have
\[ \eta^{\pom} = (\eta'_{\lam})^{\pom} * (\eta''_{\lam})^{\pom}. \]
By the decomposition \eqref{eq:disintegrate} and \cite[Lemma 4.3(i)]{SS16'}, in order to finish the proof of Proposition \ref{prop:ac on curves} it suffices to prove the following two lemmas, analogous to \cite[Lemmas 6.4 and 6.5]{SSS}.

\begin{lem}\label{lem:typical fourier}
For any choice of $s \in \N$, for almost all $\lam \in \D_*$, the measure $ (\eta'_{\lam})^{\pom}$ is in $\Dk_2$ for $\PP$-almost all $\om \in  \Om$.
\end{lem}

\begin{lem}\label{lem:typical dim}
For any $s \in \N$ large enough, for almost all $\lam \in \D_*$ with $s(\lam, \bp) > 2 + \eps$, the measure $ (\eta''_{\lam})^{\pom}$ has exact dimension $2$ for $\PP$-almost all $\om \in \Om$.
\end{lem}

For both proofs, we will realize measures $(\eta'_{\lam})^{\pom},\ (\eta''_{\lam})^{\pom}$ as random measures generated by certain models $\Sigma'_\lam, \Sigma''_\lam$.

\begin{proof}[Proof of Lemma \ref{lem:typical fourier}]
Let $I' = I^s$ and for $\vec{i} = (\vec{i}_1, \ldots, \vec{i}_s) \in I^s$ consider an IFS
\[ \left( \Phi'_\lam \right)^{(\vec{i})} = \left( \lam_{\vec{i}_1}\cdots\lam_{\vec{i}_s}z + \wt{t}_{\lam, j}^{(\vec{i}_s)} \right)_{j=1}^{k_{\vec{i}_s}},\]
together with $\wt{p}'_{\vec{i}} = \wt{p}_{\vec{i}_s}$. Define a map $F : I^\N \to (I')^\N$ by 
\begin{equation}\label{eq:F def}
F(\om) = \left( (\om_{js + 1}, \ldots, \om_{(j+1)s})\right)_{j=0}^{\infty}.
\end{equation}
Clearly, $F$ maps a fully supported Bernoulli measure $\PP$ on $I^\N$ to a fully supported Bernoulli measure $F(\PP)$ on $(I')^\N$. It is straightforward to check that if $\wt{\eta}^{(\om')}_\lam$ is a random measure generated by the the model $\Sigma'_\lam =   \left( \left( \left( \Phi'_\lam \right)^{(\vec{i})} \right)_{\vec{i}\in I'}, \left( \wt{p}'_{\vec{i}} \right)_{\vec{i} \in I'}, F(\PP)  \right)$, then $(\eta'_{\lam})^{\pom} = \wt{\eta}^{(F(\om))}_\lam$. As the set of translation parts in $\left( \Phi'_\lam \right)^{(\vec{i})}$ is the same as in $\Phi^{(\vec{i}_s)}_\lam$, we see that $\Sigma'_\lam$ is non-degenerate provided that $\Sigma_\lam$ is non-degenerate. By Lemma \ref{lem:non-deg}, this holds for almost every $\lam \in \D_*$. We can therefore apply Theorem \ref{thm:fourier main} together with Fubini's Theorem to finish the proof (Fubini's theorem can be applied, as the set $\Gk$ from Theorem \ref{thm:fourier main} is Borel measurable)
\end{proof}

\begin{proof}[Proof of Lemma \ref{lem:typical dim}]
Fix $s \in \N$ and define a model $\Sigma''_\lam$ as follows. Let $I'' = I^s$. For each $\vec{i} \in I''$ define $K_{\vec{i}} = \prod \limits_{\ell = 1}^{s-1} \{1, \ldots, k_{ \vec{i}_{\ell} } \}$ and consider  the IFS
\[ \left( \Phi''_\lam \right)^{\left( \vec{i} \right)} = \left( \lam_{\vec{i}_1} \cdots \lam_{\vec{i}_s}z + \sum \limits_{\ell=1}^{s-1}\lam_{\vec{i}_1} \cdots \lam_{\vec{i}_{\ell-1}}\wt{t}^{(\vec{i}_\ell)}_{\lam, j_{\ell}} : j \in K_{\vec{i}}\right). \]
Let also $\wt{p}''_{\vec{i}}$ be the uniform probability vector on $K_{\vec{i}}$. Let $F : I^\N \to (I'')^\N$ be defined as in \eqref{eq:F def}. Similarly as before, it is easy to check that the measures $\wt{\eta}_\lam^{(\om'')}$ generated by the model $\Sigma''_\lam =   \left( \left( \left( \Phi''_\lam \right)^{(\vec{i})} \right)_{\vec{i}\in I''}, \left( \wt{p}''_{\vec{i}} \right)_{\vec{i} \in I''}, F(\PP)  \right)$ satisfy  $(\eta''_{\lam})^{\pom} = \wt{\eta}^{(F(\om))}_\lam$.
We shall now verify that the model $\Sigma''_\lam$ satisfies Assumption \ref{assmp:non-deg and rotation} for almost every $\lam \in \D_*$.  To see that the model $\Sigma''_{\lam}$ is non-degenerate, note that
\[  \sum \limits_{\ell=1}^{s-1}\lam_{\vec{i}_1} \cdots \lam_{\vec{i}_{\ell-1}}\wt{t}^{(\vec{i}_\ell)}_{\lam, j_{\ell}} = f_{\lam, j_1}^{\vec{i}_1} \circ \cdots \circ f_{\lam, j_{s-1}}^{\vec{i}_{s-1}}(0),\]
hence for a given $\vec{i} \in I''$ the set of translations occurring in the IFS $\left( \Phi''_\lam \right)^{\left( \vec{i} \right)} $ is the same as the ones occurring in $\Phi_\lam^{\left( \vec{i}_1 \right)} \circ \cdots \circ \Phi_\lam^{\left( \vec{i}_{s-1} \right)} $. Consequently, the model $\Sigma''_\lam$ is non-degenerate provided that $\Sigma_\lambda$ is non-degenerate. By Lemma \ref{lem:non-deg}, this holds for almost every $\lam \in \D_*$. Moreover, it follows from the definitions of $\Sigma_\lam$ and $\Sigma''_\lam$, that the latter model contains an IFS with the linear part $\lam^{rs}$. As for a fixed $s$
 we have $\lam^{rs} \notin \R$ for almost every $\lam \in \D_*$, we see that $\Sigma''_\lam$ verifies Assumption \ref{assmp:non-deg and rotation} for almost every $\lam \in \D_*$.

Define now
\[ \wt{\Delta}_n(\lam) = \min \left\{ |g_{\lam, u} (0)- g_{\lam, v}(0)| : u \neq v \in \{ 1, \ldots, k\}^n \right\}.\]
The alculation from \cite[Proof of Lemma 6.5]{SSS} gives for $\om \in \Om$:
\begin{equation}\label{eq:delta comp} \Delta^{(F(\om))}_n(\Sigma''_\lam) \geq \wt{\Delta}_{rsn}(\lam)\ \text{ provided }\ |\X''^{(F(\om))}_n| > 1,
\end{equation}
where $\Delta^{(F(\om))}_n(\Sigma''_\lam)$ and $\X''^{(F(\om))}_n$ are defined as in \eqref{eq:delta def} and \eqref{eq:doubleX def} respectively, for the model $\Sigma''_\lam$ (note that $\X''^{(F(\om))}_n$ does not depend on $\lam$). As $t_i$ are all distinct, we use \cite[Theorem 1.10]{HRd} in the same way as in the proof of \cite[Lemma 6.6]{SSS} to conclude that for almost every $\lam \in \D_*$ (actually for $\lam$ from outside of a set of Hausdorff dimension at most $1$) there exists $M = M(\lam)$ such that
\begin{equation}\label{eq:exp sep} \wt{\Delta}_{rsn}(\lam) \geq e^{-Mrsn} \text{ for infinitely many } n \in \N.
\end{equation}
As for every $\lam \in \D_*$, such that $\Sigma''_\lam$ is non-degenerate, we have $\lim \limits_{n \to \infty} \PP \left( \left\{ \om \in \Om : |\X''^{(F(\om))}_n |> 1 \right\} \right) = 1$, combining \eqref{eq:delta comp} with \eqref{eq:exp sep} gives
\[ \frac{1}{n}\log \Delta^{(\cdot)}_n(\Sigma''_\lam) \overset{F(\PP)}{\nrightarrow} - \infty\ \text{ for almost every } \lam \in \D_*.\]
Now  Theorem \ref{thm:main dim} implies that $\dim(\Sigma''_\lam) = \min\{2, \sdim(\Sigma''_\lam) \} $ for almost every $\lam \in \D_*$.  As shown in \cite[Proof of Lemma 6.5]{SSS},
\[ \sdim(\Sigma''_\lam) = (1-1/s)\sdim(\Sigma_\lam).\]
Recalling \eqref{eq:sdim model bound} we see that if $s$ is large enough, then for almost every $\lam \in \D_*$, the inequality $s(\lam, p) > 2 + \eps$ implies $\sdim(\Sigma''_\lam) \geq 2$ and hence $\dim(\Sigma''_\lam) = 2$ . For such $\lam$, Proposition \ref{prop:exact dimension} gives   $\dim(\eta''^{\pom}) = 2$ for $\PP$-a.e.\ $\om \in \Om$. This finishes the proof of Lemma  \ref{lem:typical dim} and concludes the proof of Proposition \ref{prop:ac on curves}. Therefore, Theorem  \ref{thm:main ac} is proved.
\end{proof}

\appendix

\section{Proof of Theorem \ref{thm:main dim} for ergodic measures}\label{app:ergodic case}

In this section we explain how to deduce  Theorem \ref{thm:main dim} from Proposition \ref{prop:main dim 1/2} for an arbitrary ergodic measure $\PP$ on $I^\N$. The idea of the proof is the same as in the already established Bernoulli case, but now the measure $\PP'$ on the ``$k$-iterate'' model $\Sigma'$ is no longer guaranteed to be ergodic, hence we have to consider its ergodic decomposition. Below are the details.

Let $\Sigma = \left( (\Phi^{(i)})_{i \in I}, (p_i)_{i \in I}, \PP \right)$ be an arbitrary model verifying Assumption \ref{assmp:non-deg and rotation} and assume that $\dim(\Sigma) < \min\{2, \sdim(\Sigma)\}$. Fix $k_0 \in \N$ such that $r_{\max}^{k_0} = \left(r_{\max}(\Sigma) \right)^{k_0} \leq 1/2$. Let $\Om = I^\N$ and for $k \geq k_0$ define $\Om'_k := \left(I^k\right)^\N$ together with the bijection $h_k  : \Om \to \Om'_k$ given by
\[h_k(\om_1, \om_2, \ldots) = ((\om_1, \ldots, \om_{k}), (\om_{k+1}, \ldots, \om_{2k}), \ldots),\]
and the shift map $T'_k : \Om'_k \to \Om'_k$ defined as 
\[ T'_k((\om_1, \ldots, \om_{k}), (\om_{k+1}, \ldots, \om_{2k}), \ldots) = ((\om_{k+1}, \ldots, \om_{2k}), \ldots) .\] Let \[\Sigma'_k = \left( \left(\Psi^{(\ov{i})}\right)_{\ov{i} \in I^k}, (p_{\ov{i}})_{\ov{i} \in I^k}, \PP'_k \right)\]
be a model with the index set $I^k$, where $\PP'_k = h_k(\PP)$ and for $\ov{i} = (i_1, \ldots, i_k) \in I^k$ we put
\[\Psi^{(\ov{i})} = \left\{ f^{(i_1)}_{u_1} \circ \cdots \circ f^{(i_k)}_{u_k} : u_j \in \{1, \ldots, k_{i_j}\} \text{ for } 1 \leq j \leq k \right\},\]
together with $p_{\ov{i}} = p_{i_1} \otimes \cdots \otimes p_{i_k}$. Note that $\PP'_k$ is $T'_k$-invariant (as $h_k \circ T^k = T'_k \circ h_k$)\footnote{$\PP'_k$ is $T'_k$-ergodic if and only if $\PP$ is ergodic for the $k$-th iterate of the shift $T$ on $\Om$. This does not have to be case for an arbitrary $T$-ergodic measure $\PP$ on $\Om$.} and $h_k$ is a measure-preserving bi-measurable bijection between $(\Om, \PP)$ and $(\Om'_k, \PP'_k)$.
Consider the ergodic decomposition of $\PP'_k$ with respect to $T'_k$ (see e.g. \cite[Theorem 6.2]{EWergodictheory}): There is a $\PP'_k$-almost surely defined assignment $\om' \to \PP'_{\om', k}$, with $\PP'_{\om', k}$ being $\T'_k$-invariant and ergodic probability measures on $\Om'_k$ such that
\begin{equation}\label{eq: ergodic decomp}
	\PP'_k = \int \limits_{\Om'_k} \PP'_{\om', k}d\PP'_k(\om').
\end{equation}
Therefore, for $\PP$-a.e.\ $\om \in \Om$, we can consider a model $\Sigma'_{\om,k} = \left( \left(\Psi^{(\ov{i})}\right)_{\ov{i} \in I^k}, (p_{\ov{i}})_{\ov{i} \in I^k}, \PP'_{h_k(\om), k} \right)$ with an ergodic selection measure $\PP'_{h_k(\om), k}$ on $\Om'_k$. It is easy to see that $\eta^{(\om)} = \eta'^{(h_k(\om))}$ for every $\om \in \Om$, where $\eta'^{(\cdot)}$ denotes  projected measures in the model $\Sigma'_{\om,k}$. Therefore, \begin{equation}\label{eq:sigma k dim} \dim(\Sigma) = \dim(\Sigma'_{\om,k}) \text{ for } \PP \text{-a.e. } \om \in \Om.
\end{equation}
We shall prove now that for $\PP$-almost every $\om \in \Om$ there exists $k$ such that we can apply the already established version of Theorem \ref{thm:main dim} to $\Sigma'_{\om,k}$. We will do so in a series of claims. To make the notation more clear, we will denote elements of $\Om$ by $\om = (\om_1, \om_2, \ldots)$ and elements of $\Om'_k$ by $\om' = ((\om'_1, \ldots \om'_k), (\om'_{k+1}, \ldots, \om'_{2k}), \ldots)$. \\

\textbf{Claim 1.} {\em For $\PP$-a.e.\ $\om \in \Om$ there exists $k(\om)$ such that $\dim(\Sigma'_{\om,k}) < \min\{2, \sdim(\Sigma'_{\om,k}) \}$ for every $k \geq k(\om)$}.\\

Define the functions $f,g : \Om \to \R$ as $f(\om) = H(p_{\om_1})$ and $g(\om) = -\log r_{\om_1}$. A direct computation from the definition \eqref{eq:sdim def} of the similarity dimension shows that
\begin{equation}\label{eq: sigma_k sdim} \sdim(\Sigma'_{\om,k}) = \frac{\int \limits_{\Om'_k}H(p_{\om'_1} \otimes \cdots \otimes p_{\om'_k})d\PP'_{\om,k}(\om')}{\int \limits_{\Om'_k} \log ( r_{\om'_1}\cdots r_{\om'_k} )  d\PP'_{\om,k}(\om')} = \frac{\int \limits_{\Om'_k} \frac{1}{k} \sum \limits_{j=0}^{k-1} f(T^j (h^{-1}_k(\om')))d\PP'_{\om,k}(\om')}{\int \limits_{\Om'_k} \frac{1}{k} \sum \limits_{j=0}^{k-1} g(T^j (h^{-1}_k(\om')))d\PP'_{\om,k}(\om')}
\end{equation}
and
\[ \sdim(\Sigma) = \frac{\int\limits_{\Om} f(\om)d\PP(\om)}{\int\limits_{\Om} g(\om)d\PP(\om)}. \]
As $\PP$ is ergodic, we have for $\PP$-a.e.\ $\om \in \Om$:
\[\lim \limits_{k \to \infty} \frac{1}{k} \sum \limits_{j=0}^{k-1} f(T^j \om) = \int\limits_{\Om} f(\om)d\PP(\om)\ \text{ and }\ \lim \limits_{k \to \infty} \frac{1}{k} \sum \limits_{j=0}^{k-1} g(T^j \om) = \int\limits_{\Om} g(\om)d\PP(\om).\]
By Egorov's theorem applied to the convergence above, for every $\eps > 0$ there exists a Borel set $E_\eps \subset \Om$ with $\PP(E_\eps) \geq 1 - \eps$ and $k(\eps) \in \N$ such that for every $k \geq k(\eps)$ and $\om \in E_\eps$ we have
\[ \left| \frac{1}{k} \sum \limits_{j=0}^{k-1} f(T^j \om) - \int\limits_{\Om} f(\om)d\PP(\om) \right| < \eps\ \text{ and }\ \left| \frac{1}{k} \sum \limits_{j=0}^{k-1} g(T^j \om) - \int\limits_{\Om} g(\om)d\PP(\om) \right| < \eps\]
By the Chebyshev inequality, the set $E'_\eps = \{ \om \in \Om : \PP'_{k, h(\om)}(h_k(E_\eps)) \geq 1 - \sqrt{\eps} \}$ satisfies $\PP(E'_\eps) \geq 1 - \sqrt{\eps}$. By \eqref{eq: sigma_k sdim}, this implies that for $\delta > 0$, if $\eps< \eps(\delta)$ then $|\sdim(\Sigma'_{\om,k}) - \sdim(\Sigma)| < \delta$ for every $\om \in E'_\eps$ and $k \geq k(\eps)$. We can therefore conclude, in view of
 $\dim(\Sigma) < \min\{2, \sdim(\Sigma)\}$ together with \eqref{eq:sigma k dim}, that if $\eps > 0$ is small enough, then $\dim(\Sigma'_{\om,k}) < \min\{2, \sdim(\Sigma'_{\om,k}) \}$ holds for $\PP$-almost every $\om \in E'_\eps$ and $k \geq k(\eps)$. Consequently, the conclusion of the claim holds for $\PP$-a.e.\ $\om \in  \bigcup \limits_{n = N}^\infty E'_{\frac{1}{n}}$ with $N$ large enough. As this is a set of full measure, the claim is proved.\\

\textbf{Claim 2.} {\em For $\PP$-a.e.\ $\om \in \Om$ there exists $k(\om)$ such that model $\Sigma'_{\om, k}$ satisfies the item \ref{it:non-degenerate} of Assumption \ref{assmp:non-deg and rotation} for every $k \geq k(\om)$}.\\

Let $i_0 \in I$ satisfy Assumption \ref{assmp:non-deg and rotation}\ref{it:non-degenerate} for the original model $\Sigma$. First, note that for $\PP$-a.e.\ $\om \in \Om$, the following implication is true for every $k \in \N$:
\begin{equation}\label{eq: Sigma k fails implication}
	\Sigma'_{\om,k} \text{ does not satisfy Assumption \ref{assmp:non-deg and rotation}\ref{it:non-degenerate}} \Longrightarrow i_0 \notin \{\om_1, \ldots \om_k\}.
\end{equation}
Indeed, fix $k \in \N$ and observe that if $\Sigma'_{\om,k}$ does not satisfy  Assumption \ref{assmp:non-deg and rotation}\ref{it:non-degenerate}, then $\PP'_{h_k(\om),k}\left( \left\{ \om' \in \Om'_k :  i_0 \notin \{ \om'_{1}, \ldots, \om'_k \} \right\} \right) = 1$. As for $\PP$-a.e.\ $\om$, the sequence $h_k(\om)$ is typical is for its own ergodic element $\PP'_{h_k(\om),k}$\footnote{This can be formally seen e.g. from the proof of \cite[Theorem 6.2]{EWergodictheory} combined with \cite[Theorem 5.14 p. (ii)]{EWergodictheory}.}, this implies \eqref{eq: Sigma k fails implication}. It now follows from \eqref{eq: Sigma k fails implication} that
\[
\begin{split} \PP & \left( \left\{ \om \in \Om : \Sigma'_{k,\om} \text{ does not satisfy Assumption \ref{assmp:non-deg and rotation}\ref{it:non-degenerate} for infinitely many } k \in \N,\ \right\} \right) \\
	& \leq \PP \left( \left\{ \om \in \Om : i_0 \notin \{\om_1, \ldots \om_k\} \text{ for infinitely many } k \in \N \right\} \right) \\
	& = \PP \left( \left\{ \om \in \Om : i_0 \notin \{\om_1, \om_2, \ldots \} \right\} \right) \\
	& = 0,
\end{split}
\]
by Assumption \ref{assmp:non-deg and rotation}\ref{it:non-degenerate} for $\Sigma$. This finishes the proof of Claim 2.\\

\textbf{Claim 3.} {\em For $\PP$-a.e.\ $\om \in \Om$, the model $\Sigma'_{\om, k}$ satisfies the item \ref{it:non-real} of Assumption \ref{assmp:non-deg and rotation} for infinitely many $k \in \N$}.\\

Using the same approach as in Claim 2, we can prove that for $\PP$-a.e.\ $\om \in \Om$ and every $k \in \N$, the following implication is true:
\[ \Sigma'_{\om,k} \text{ does not satisfy Assumption \ref{assmp:non-deg and rotation}\ref{it:non-real}} \Longrightarrow \vphi_{\om_1} \cdots \vphi_{\om_k} \in \R. \]
This gives
\[
\begin{split} \PP & \left( \left\{ \om \in \Om : \Sigma'_{k,\om} \text{ satisfies  Assumption \ref{assmp:non-deg and rotation}\ref{it:non-real} only for finitely many } k \in \N \right\} \right) \\
	& \leq \PP \left( \left\{ \om \in \Om : \vphi_{\om_1} \cdots \vphi_{\om_k} \notin \R \text{ only for finitely many } k \in \N \right\} \right) \\
	& = 0,
\end{split}
\]
where the last equality follows from Lemma \ref{lem:skew product ergodic} and the fact that $\xi^{(W)}$ is not concentrated on a single atom by Assumption 
\ref{assmp:non-deg and rotation}\ref{it:non-real} for $\Sigma$.\\

\textbf{Claim 4.} {\em For $\PP$-a.e.\ $\om \in \Om$ there exists $k = k(\om) \geq k_0$ such that $\dim(\Sigma'_{\om,k}) < \min\{2, \sdim(\Sigma'_{\om,k}) \}$ and the model $\Sigma'_{\om,k}$ satisfies Assumption \ref{assmp:non-deg and rotation}.}\\

This follows directly from Claims 1-3. For $k \geq k_0$ define
\[A_k:= \left\{ \om \in \Om : \text{ model } \Sigma'_{\om,k} \text{ satisfies Assumption \ref{assmp:non-deg and rotation} and }  \dim(\Sigma'_{\om,k}) < \min\{2, \sdim(\Sigma'_{\om,k}) \} \right\}.\]

\textbf{Claim 5.} {\em For each $k \geq k_0$ and $M > 0$ we have $\lim \limits_{n \to \infty} \PP\left( \left\{ \om \in \Om :  \Delta^{(\om)}_{kn} \leq e^{-Mn} \text{ and } \om \in A_k  \right\} \right) = \PP(A_k) $}.\\

Fix $M > 0$ and $k \geq k_0$. For each $\om \in A_k$, we can apply Proposition \ref{prop:main dim 1/2} to the model $\Sigma'_{\om,k}$, as $r_{\mathrm{max}}(\Sigma'_{\om,k}) \leq \left(r_{\mathrm{max}}(\Sigma)\right)^k \leq 1/2$. For $\om' \in \Om'_k$, let $\Delta_{n}^{'(\om')}$ be defined as in \eqref{eq:delta def} for the model $\Sigma'_{\om,k}$ (which we formally consider over the symbolic space $\Om'_k$).  Therefore, for $\om \in A_k$ we have
\[ \lim \limits_{n \to \infty} \PP'_{h_k(\om), k} \left( \left\{ \om' \in \Om'_k: \Delta^{'(\om')}_n \leq e^{-Mn} \right\} \right) = 1 \]
and hence by the Lebesgue's dominated convergence theorem
\begin{equation}\label{eq: limit n over Ak} \lim \limits_{n \to \infty} \int \limits_{A_k} \PP'_{h_k(\om), k} \left( \left\{ \om' \in \Om'_k: \Delta^{'(\om')}_n \leq e^{-Mn} \right\} \right) d\PP(\om) = \PP(A_k).
\end{equation}
On the other hand, note that for $\PP$-a.e.\ $\om \in A_k$, we also have
\[\PP'_{h_k(\om),k} \left( \left\{ \om' \in \Om'_k : h_k^{-1}(\om') \in A_k \right\} \right) = 1,\]
as the property of belonging to $A_k$ depends only on the model $\Sigma'_{\om,k}$ given by the corresponding ergodic element (which is the same for typical points coming from a fixed ergodic element $\PP_{h_k(\om), k}$). Moreover, for each $\om \in \Om$, we have $\Delta_{kn}^{(\om)} =\Delta_{n}^{'(h_k(\om))}$. The above observations, together with \eqref{eq: ergodic decomp} and $\PP = h^{-1}_k\PP'_k$, give
\[
\begin{split} \int \limits_{A_k} \PP'_{h_k(\om), k} & \left( \left\{ \om' \in \Om': \Delta^{'(\om')}_n \leq e^{-Mn} \right\} \right) d\PP(\om) \\
	& = \int \limits_{A_k} \PP'_{h_k(\om), k} \left( \left\{ \om' \in \Om': \Delta^{(h_k^{-1}(\om'))}_{kn} \leq e^{-Mn} \text{ and } h_k^{-1}(\om') \in A_k \right\} \right) d\PP(\om) \\
	& \leq \int \limits_{\Om} h_k^{-1}\PP'_{h_k(\om), k} \left( \left\{ \overline{\om} \in \Om: \Delta^{(\overline{\om})}_{kn} \leq e^{-Mn} \text{ and } \overline{\om} \in A_k \right\} \right) d\PP(\om) \\
	& = \PP\left( \left\{ \om \in \Om: \Delta^{(\om)}_{kn} \leq e^{-Mn} \text{ and } \om \in A_k \right\} \right) \leq \PP(A_k).
\end{split}  \]
Combining this with \eqref{eq: limit n over Ak} finishes the proof of Claim 5.\\

The following claim will conclude the proof of Theorem \ref{thm:main dim}.\\

\textbf{Claim 6.} {\em For every $M>0$, we have $\lim \limits_{n \to \infty} \PP \left( \left\{ \om \in \Om : \Delta^{(\om)}_n \leq e^{-Mn} \right\} \right) = 1$}.\\

Note first that Assumption \ref{assmp:non-deg and rotation} implies
\begin{equation}\label{eq: X om > 1}\lim \limits_{n \to \infty} \PP\left( \left\{ \om \in \Om : |\X^{(\om)}_n| > 1 \right\} \right) = 1.
\end{equation}
For $k,n \in \N$ define $\lfloor n/k \rfloor = \frac{n - (n \mod k)}{k}$, so that $n = k\lfloor n/k \rfloor + (n \mod k)$. Lemma \ref{lem:Delta inequality} gives for $n \geq k$:
\begin{equation}\label{eq: superexp inclusion}
	\begin{split} \Big\{ \om \in \Om : \Delta^{(\om)}_{k\lfloor n/k \rfloor} \leq e^{-2Mk\lfloor n/k \rfloor}&,\ \om \in A_k,\ |\X^{(\om)}_{k\lfloor n/k \rfloor}| > 1 \Big\}  \\
		& \subset \left\{ \om \in \Om : \Delta^{(\om)}_{n} \leq e^{-Mn} \text{ and } \om \in A_k \right\},
	\end{split}
\end{equation}
as if $|\X^{(\om)}_{k\lfloor n/k \rfloor}| > 1$ and $\Delta^{(\om)}_{k\lfloor n/k \rfloor} \leq e^{-2Mk\lfloor n/k \rfloor}$, then for $j = n \mod k$:
\[ \Delta^{(\om)}_n = \Delta^{(\om)}_{k\lfloor n/k \rfloor + j} \leq \Delta^{(\om)}_{k\lfloor n/k \rfloor} \leq e^{-2Mk\lfloor n/k \rfloor} \leq e^{-M(k\lfloor n/k \rfloor + j)} = e^{-Mn}.  \]
Let us define
\[ B_{n,k,M} = \left\{ \om \in \Om :  \Delta^{(\om)}_{kn} \leq e^{-Mn} \text{ and } \om \in A_k  \right\}. \]
Combining Claim 5 with \eqref{eq: X om > 1} allows us to conclude that for every $k \geq k_0$ and $M>0$,
\begin{equation}\label{eq:B convergence} \lim \limits_{n \to \infty } \PP\left( B_{\lfloor n/k \rfloor, k, Mk} \cap \{ \om \in \Om : |\X^{(\om)}_{k\lfloor n/k \rfloor}| > 1 \} \right) = \PP(A_k).
\end{equation}
As Claim 4 gives $\PP(\bigcup \limits_{k \geq k_0}A_k) = 1$, we have by \eqref{eq: superexp inclusion}:
\[\begin{split} \PP \left( \left\{ \om \in \Om : \Delta^{(\om)}_n \leq  e^{-Mn} \right\} \right) & \geq \PP \left( \bigcup \limits_{k \geq k_0} \left( A_k \cap \left\{ \om \in \Om : \Delta^{(\om)}_n \leq e^{-Mn} \right\} \right) \right)  \\
	& \geq \PP\left(\bigcup \limits_{k = k_0}^{n} \left( B_{\lfloor n/k \rfloor, k, Mk} \cap \{ \om \in \Om : |\X^{(\om)}_{k\lfloor n/k \rfloor}| > 1 \} \right) \right).
\end{split}
\]
Finally, \eqref{eq:B convergence} gives
\[ \lim \limits_{n \to \infty} \PP \left( \left\{ \om \in \Om : \Delta^{(\om)}_n \leq e^{-Mn} \right\} \right) \geq \PP(\bigcup \limits_{k \geq k_0}A_k) = 1, \]
finishing the proof. \hfill $\square$

\bibliographystyle{alpha}
\bibliography{selfsimilar_bib}

\end{document}